\documentclass[a4paper,10pt,reqno]{amsart}
\usepackage[foot]{amsaddr}
\usepackage[english]{babel}
\usepackage[utf8]{inputenc}
\usepackage{graphicx}
\graphicspath{{figures/}}
\usepackage[colorlinks=true,linkcolor=black,urlcolor=black,citecolor=black]{hyperref}
\usepackage{amsmath}
\usepackage{amsfonts}
\usepackage{amssymb}
\usepackage{amsthm}
\usepackage{fixmath}
\usepackage{tikz}
\usepackage{tikz-cd}
\usetikzlibrary{matrix}
\usetikzlibrary{arrows,arrows.meta}
\tikzcdset{arrows={line width=1pt},arrow style=tikz, diagrams={>=stealth}}
\tikzcdset{every label/.append style = {scale=1,yshift=0.2ex}, nodes={font=\Large}} 
\usepackage{caption}
\usepackage{subcaption}

\usepackage{enumerate}
\usepackage[shortlabels]{enumitem}
\usepackage{mathtools}
\usepackage[disable]{todonotes}
\numberwithin{equation}{section}
\newtheorem{theorem}{Theorem}[section]

\newtheorem{proposition}[theorem]{Proposition}

\newtheorem{remark}{Remark}[section]
\newtheorem{corollary}{Corollary}[section]

\newcommand{\x}{{{x}}}
\newcommand{\xx}{{y}}
\newcommand{\p}{u}
\newcommand{\pU}{U} 
\renewcommand{\t}{\theta}

\newcommand{\ph}{\p_h}
\renewcommand{\th}{\t_h}
\newcommand{\lh}{\lambda_h}
\newcommand{\lhp}{\lambda_{h,+}}
\newcommand{\lhm}{\lambda_{h,-}}
\newcommand{\lhpm}{\lambda_{h,\pm}}
\newcommand{\wh}{w_h}
\renewcommand{\d}{\mathop{}\!\mathrm{d}}
\newcommand{\D}{\Omega}

\newcommand{\DI}{\D_{I}}
\newcommand{\DD}{\D\cup\DI}
\newcommand{\NM}{\mathcal{N}}
\newcommand{\R}{\mathbb{R}}
\newcommand{\kernel}{\gamma}
\newcommand{\kerneld}{\gamma(\x-\xx)}
\newcommand{\V}{{V}}
\newcommand{\Va}{\V_A}
\newcommand{\Vt}{\V_\t}

\newcommand{\Vb}{\V_B}
\newcommand{\M}{M}
\renewcommand{\tt}{\tau}
\newcommand{\cker}{c_{\kernel}}
\newcommand{\ckerr}{C_{\kernel}}

\newcommand{\cpot}{c_{F}}
\newcommand{\cm}{c_{F}}
\newcommand{\param}{\xi}

\newcommand{\G}{\mathcal{G}}
\newcommand{\J}{J_k}
\DeclareMathOperator*{\esssup}{ess\,sup}
\newcommand{\A}{{A}}

\newcommand{\tend}{T}

\newcommand{\pt}{{P}} 
\newcommand{\Sh}{{S}^h_{\DD}}
\newcommand{\Sho}{S^h_\D}
\newcommand{\e}{\varepsilon}

\newcommand*{\norms}[1]{{\left|{#1}\right|}}
\newcommand*{\norm}[1]{{\left\lVert#1\right\rVert}}
\newcommand*{\inner}[1]{{\left(#1\right)}}
\newcommand*{\norml}[1]{{\left\lVert#1\right\rVert}}
\newcommand*{\duala}[1]{{\left\langle{#1}\right\rangle}_{\Va}}
\newcommand*{\dualt}[1]{{\left\langle{#1}\right\rangle}_{\Vt}}
\newcommand*{\dualh}[1]{{\left\langle{#1}\right\rangle}_{H^1(\D)}}
\newcommand*{\dual}[1]{{\left\langle{#1}\right\rangle}}
\newcommand{\inp}[1]{\overline{#1}_{\tt}}
\newcommand{\inpl}[1]{\underline{#1}_{\tt}}
\newcommand{\inpt}[1]{\hat{#1}_{\tt}}

\newcommand{\inpp}[1]{\overline{#1}_{+,\tt}}
\newcommand{\inpm}[1]{\overline{#1}_{-,\tt}}

\newcommand{\Q}{Q}
\newcommand{\QC}{\hat{Q}}
\DeclareMathOperator*{\argmin}{arg\,min}

\begin{document}
\title[Non-isothermal nonlocal phase-field models]{Non-isothermal nonlocal phase-field models with a double-obstacle potential}
\author[O. Burkovska]{Olena Burkovska${}^{*}$}

\address{${}^*$ Computer Science and Mathematics Division, Oak Ridge National Laboratory, One Bethel Valley Road, TN 37831, USA}
\email{burkovskao@ornl.gov}
\thanks{
This material is based upon work supported by the U.S. Department of Energy, Office of Advanced Scientific Computing Research, Applied Mathematics Program under the award numbers ERKJ345 and ERKJE45; and was performed at the Oak Ridge National Laboratory, which is managed by UT-Battelle, LLC under Contract No. De-AC05-00OR22725. 
The US government retains and the publisher, by accepting the article for publication, acknowledges that the US government retains a nonexclusive, paid-up, irrevocable, worldwide license to publish or reproduce the published form of this manuscript, or allow others to do so, for US government purposes. DOE will provide public access to these results of federally sponsored research in accordance with the DOE Public Access Plan (https://www.energy.gov/doe-public-access-plan).
}

\begin{abstract}
{Phase-field models are a popular choice in computational physics to describe complex dynamics of substances with multiple phases and are widely used in various applications.  We present nonlocal non-isothermal phase-field models of Cahn-Hilliard and Allen-Cahn types involving a nonsmooth double-well obstacle potential. Mathematically, in a weak form, the model translates to a system of variational inequalities coupled to a temperature evolution equation.  We demonstrate that under certain conditions and with a careful choice of the nonlocal operator one can obtain a model that allows for sharp interfaces in the solution that evolve in time, which is a desirable property in many applications. This can be contrasted to the diffuse-interface local models that can not resolve sharp interfaces.  We present the well-posedness analysis of the models,  discuss an appropriate numerical discretization scheme,  and supplement our findings with several numerical experiments.}
\end{abstract}



\maketitle

\section{Introduction}

Phase-field modeling is a powerful approach to describe various physical processes, such as, e.g., nucleation or solidification in material science~\cite{wang1993,kobayashi1993,xudu2023}
or image processing~\cite{bertozzi2007,garke2018}, just to mention a few. 
Most commonly, local phase-field models of Allen-Cahn or Cahn-Hilliard type, which are based on the local partial derivatives, are employed. However, those models always produce diffuse interfaces, which creates computational challenges in the setting when thin or sharp interfaces are desired. 

For example, in applications of solidification of materials the interface between two substances is described by a moving boundary problem that assumes sharp interfaces between pure phases. 
A practical realization of such models is challenging due to the coupling at a moving interface,  and instead a phase-field method can be adopted,  where the sharp interface is replaced by a diffuse phase-field model. This approach avoids an explicit tracking of the interface and simulates the whole evolution of a phase change.
However, to resolve very thin diffuse interfaces in order to ensure a good approximation to the sharp-interface solution, very fine meshes and/or more structurally complex models are needed.

In this paper, we propose a novel non-isothermal phase-field model based on a nonlocal phase-field system coupled to a temperature evolution equation. 
The nonlocality in the phase-field model allows to provide sharper interfaces in the solution independently of the mesh resolution and can help to alleviate the limitations of resolving thin diffuse interfaces in the local setting. 

More specifically, given a bounded domain $\D\subset\R^n$, $n\leq 3$ we consider the model
\begin{align}
\begin{dcases}
\partial_t\t -  D\upDelta\t - L\partial_t\p=0,\\
\mu\mathcal{G}\left(\partial_t \p\right) +B\p+ F_\p^{\prime}(\p,\t)=0,
\end{dcases}\quad\text{in}\quad (0,T)\times\D,\quad T>0,
\label{model_nonl_general}
\end{align}
where $\t(t,x)$ is the temperature and $\p(t,x)$ is the order parameter, which typically assumes values in $[0,1]$ with $\p=0$ and $\p=1$ corresponding to pure phases, e.g., liquid and solid, and $F(\p,\t)$ is a double-well potential that encourages the solution to admit pure phases.  Here, we assume it admits the following form
\begin{equation}
F(\p,\t)=F_0(\p)+m(\t)g(\p),
\label{potnetial}
\end{equation}
where $F_0(\p)$ is a double-well function that has minimum at $\p=0$ and $\p=1$ and a coupling term $m(\t)$ forces the solution to attain one phase over another depending on the value of the temperature $\t$ respective to some equilibrium temperature $\t_e$. 
The parameters $D$, $L$ and $\mu$ are the diffusivity, latent heat and relaxation time coefficients, respectively. The operator $B$ is a nonlocal diffusion operator 
\begin{equation*}
B\p =\int_{\R^n}(\p(\x)-\p(\xx))\kernel(\x-\xx)\d\xx, 
\end{equation*}
where $\kernel(\x-\xx)$ is a compactly supported integrable kernel that defines the nature and extent of nonlocal interactions.  Additionally,  $\G$ is the Green's function of  the operator $I-\beta\upDelta$, where $\beta\geq 0$ and $I$ is an identity operator. When $\beta=0$, we obtain a nonlocal Allen-Cahn equation, while $\beta>0$ corresponds to a non-mass conserving Cahn-Hilliard type equation.  We note that under appropriate conditions on the kernel we recover a local operator in the limit of vanishing nonlocal interactions, see, e.g.,~\cite{du2018CH},
\[
Bu\to -\varepsilon^2\upDelta u, \quad \text{for}\quad \varepsilon^2=\frac{1}{2n}\int_{\R^n}\kernel(\zeta)|\zeta|^2\d\zeta,
\]
where $\varepsilon$ is parameter that controls the width of the interface in the local model. Then, the local analogue of the model~\eqref{model_nonl_general} is obtained having $B\p=-\varepsilon^2\upDelta\p$,
and taking $\G=I$ (i.e, $\beta=0$), results in a prototypical isotropic model for solidification of pure materials, see, e.g.,~\cite{kobayashi1993,boettinger2002,wang1993,wheeler1993,steinbach2009} and references therein. For an overview of nonlocal and local phase-field models we refer to~\cite{bates1997overview,FifeReview2003,DuFeng2020,bates2006survey,steinbach2009}.

Our main contribution is to demonstrate that the proposed model~\eqref{model_nonl_general} can deliver solutions with time-evolving sharp interfaces and well-defined nonlocal interfacial energy,  which is,  to the best of our knowledge,  the first result in this direction. 
The following is essential for this:
First,  the fact that we employ nonlocal operators with integrable kernels and an obstacle potential is necessary to allow sharp transitions between pure phases.
Second,  the inclusion of the Green's operator is responsible of allowing evolution of discontinuous interfaces in time.  
For example,  considering the seemingly more natural $\G=I$, which corresponds to a nonlocal Allen-Cahn model,  in general does not permit discontinuities to evolve in time due to the fact that $Bu(t) \in L^2(\D)$ and the associated higher temporal regularity $\partial_t\p\in L^2(\D)$.  As an illustration,  consider the example of a moving discontinuity with a constant velocity $\nu>0$,  given with a Heaviside function ${H}$,
\[
\p(x,t)=H(x-\nu t),\quad x \in (0,1)\subset{\R^1},\quad t>0.
\]
It is clear that $\partial_t\p=-\delta_0(x-\nu t)$ is a distribution and it is not in $L^2(\D)$,  whereas $\G(\partial_t\p)$ is an element of $L^2(\D)$. 
Therefore,  the proposed model~\eqref{model_nonl_general}  with $\G(\partial_t\p)$ relaxes temporal regularity and allows discontinuities to evolve in time.  For this purpose,  we provide analysis of the time-discrete and time continuous formulations of the model and derive conditions under which sharp interfaces can be attained. 

This work will contribute to the growing literature dealing with the nonlocal phase-field models that has gained interest in recent years; see, e.g.,~\cite{BG2021CH,du2019,du2018CH,gal_nonlocal_2017,Foss2022,DuYang2016} and references therein. 
Many of those works consider an isothermal setting, i.e., assuming a constant temperature in the system. 
The non-isothermal case often has been studied for the nonlocal model of Caginalp type~\cite{caginalp1986analysis} that resembles a model structure~\eqref{model_nonl_general} with $\G=I$ and a linear coupling term ($m(\t)=\t$, $g(\p)=\p$), see,~\cite{armstrong2010,grasselli2013,bates1997global,feireisl2004,
bates2006,colli2019nonisothermal}.
In contrast, in this work we adopt a nonlinear coupling $m(\t)$ in~\eqref{potnetial} similar to the one used in the local Kobayashi model~\cite{kobayashi1993}.
Furthermore, we employ an obstacle-type potential $F_0(\p)$ involving an indicator function, which enforces the solution to always conform to the bounds $0\leq \p\leq 1$, which can be contrasted with the most commonly used regular potential that allows non-admissible solutions (see Section~\ref{sec:model}). Such structure of the potential introduces an additional non-smoothness into the system and a more care is required for the analysis and numerical treatment of the model.  {This work builds upon our previous work on the Cahn-Hilliard model with the obstacle potential~\cite{BG2021CH} and introduces a novel non-isothermal model with a non-mass conserving phase-field variable. While we adopt similar settings for the potential and nonlocal operator, the analysis of the model is more complicated due to the presence of the temperature and nonlinear coupling term.}

We comment on related works that adopt obstacle type potentials.
Local phase-field models with non-smooth potentials have been investigated in various works, see, e.g.,~\cite{BloweylElliott1994,colli2000} among others.
For the non-isothermal nonlocal case, a Caginalp type model with the obstacle potential has been analyzed in~\cite{colli2019nonisothermal} and extended to the optimal control problem in~\cite{colli2022OC}. 
The nonlocal operators considered in those works are of a fractional type defined via spectral theory, which is different from the bounded nonlocal operators adopted in this work. 
The study of different nonlocal non-isothermal models with a non-smooth potential has been also conduced in~\cite{krejci2004,krejci2007}.

The structure of the potential plays a pivotal role to guarantee sharp interfaces in the solution. Specifically,  adopting an obstacle potential in the nonlocal phase-field system allows to obtain sharp-interfaces with only pure phases, i.e., $\p\in\{0,1\}$, as it has been demonstrated in~\cite{BG2021CH} and will be established here for the non-isothermal case~\eqref{model_nonl_general}. 
As pointed out before,  in general those properties do not hold for nonlocal Allen-Cahn systems unless in a steady-state case.  For example, in~\cite{DuYang2016} it has been proved that a steady-state solution of the nonlocal Allen-Cahn equation with a regular potential can admit discontinuities, but due to the smooth nature of the potential the size of the jump is strictly smaller than one, confirming that a small transition layer occurs between pure phases.
Overall, the question of the sharpness of the interface in the solution of a nonlocal Allen-Cahn equation has been investigated in various works, see, e.g.,~\cite{DuYang2016,
DuYang2017,du2019,LWY2017,bates1997,fife1997,BatesChmaj1999,Chen1997}.

In addition to the analysis of model~\eqref{model_nonl_general}, 
we also present spatial and temporal discretization methods.
For the latter, we propose an implicit-explicit time stepping scheme that provides an efficient evaluation of the model,  and in the Allen-Cahn case ($\beta=0$)
allows to completely bypass a solution of a nonlinear phase-field system.

While the present work discusses the mathematical formulation of the model, in our forthcoming work~\cite{BDGR23} we will address several questions related to the performance of the model, asymptotic analysis and comparative study of the model in the context of solidification of pure materials. 

The structure of the paper is as follows.  In Section~\ref{sec:model} we introduce the nonlocal model and briefly discuss related local models.  Next,  in Section~\ref{sec:preliminaries} we formulate a model in a functional analytic framework.  In the subsequent Section~\ref{sec:time-discrete} we introduce a time-stepping scheme and derive the well-posedness result of the corresponding semi-discrete problem together with the regularity and the sharp-interface properties of the solution.  Analyzing the semi-discrete problem for vanishing temporal step size we derive the existence result in Section~\ref{sec:continuous_problem} for the continuous setting.  Finally, in Section~\ref{sec:discretization} we discuss the fully discrete scheme and illustrate our theoretical results with several one- and two-dimensional examples.  We present some concluding remarks in Section~\ref{sec:conclusion}.

\section{Model setting}\label{sec:model}

We introduce a nonlocal operator $B$, defined as follows
\begin{align}
Bu(\x)=\int_{\DD}(u(\x)-u(\xx))\kerneld\d\xx=\cker u(\x)- (\kernel*u)(\x), \quad\x\in\D,
\label{operatorB}
\end{align}
where a nonlocal kernel $\kernel:\R^n\to\R^+$ is radial integrable and compactly supported
\begin{equation}
\begin{aligned}
\kernel\in L^1(\R^n), \quad \kernel(\x)=\hat{\kernel}(|\x|)\;\;\text{and}\quad {(0,\sigma)\subset{\rm supp}(\hat{\kernel})\subset(0,\delta]},\;\delta>0, \;\sigma>0,
\label{kernel_cond1}
\end{aligned}
\end{equation}
where the nonlocal interaction radius $\delta>0$ defines an extent of nonlocal interactions and $\DI$ is a nonlocal interaction domain that incorporates interactions outside of $\D$:
\begin{align*}
\DI:=\{\xx\in\R^n\setminus\D\colon\kernel(\x,\xx)\neq 0,\quad\x\in\D\}.
\end{align*}
Here,  by $\kernel*\p$ we denote the convolution of $\kernel$ and $\p$ on $\DD$, and\footnote{For convenience of notation we suppose that $u$ is extended by zero outside of $\DD$.}
\begin{align}
\cker(\x):=\int_{\DD}\kerneld\d\xx,\quad 
\ckerr:=\int_{\R^n}\kerneld\d\xx,\label{cker}
\end{align}
where $\cker=\cker(\x)$ is constant for all $\x\in\D$ and 
$0\leq\cker(\x)\leq\ckerr<\infty$.  We consider 
\begin{align}
\begin{dcases}
{\partial_t\t} = \ D\upDelta\t+L{\partial_t\p},\\
\mu{\partial_t\p}+\A w=0,\\
 w= B\p +\partial_\p F(\p,\t), 
\end{dcases}\quad\text{in}\quad (0,T)\times\D,\quad T>0,
\label{eq:strong_solid_nonlocalCH}
\end{align}
where as before $D>0$,  $L$,  $\mu>0$ are a diffusivity,  latent heat and relaxation time parameters,  respectively.  A double well-potential $F(\p,\t)$ is of the form~\eqref{potnetial} and we use $\partial_\p F(\p,\t)$ to denote a subdifferential of a non-smooth potential $F$ which we introduce shortly. 
Here, $\A:=I-\beta\upDelta$, where $\beta\geq 0$ plays a role of a ``de-regularization'' parameter for a time derivative $\partial\p/\partial t$.
We complement the system~\eqref{eq:strong_solid_nonlocalCH} with the following initial and boundary conditions:
\begin{align*}
\p(0)=\p^0,\quad \t(0)=\t^0,\\
{\partial_n w}=0 \quad (\beta>0)\quad\text{and}\quad {{\partial_n\t}=0}\quad\text{on}\quad\partial\D,\quad \mathcal{N}\p = 0\quad\text{on}\quad\DI,
\end{align*}
where $n$ denotes an outward normal to $\partial\D$ and $\mathcal{N}u$ is a nonlocal flux condition on $\DI$, which is analogous to the local Neumann type boundary condition:
\begin{align*}
\mathcal{N}\p(\x):=\int_{\DD}(\p(\x)-\p(\xx))\kernel(\x-\xx)\d\xx=0,\quad\forall\x\in\DI.
\end{align*}
For $\beta=0$,~\eqref{eq:strong_solid_nonlocalCH} reduces to the nonsiothermal nonlocal Allen-Cahn equation:  
\begin{align}
\begin{dcases}
{\partial_t\t}= \ D\upDelta\t+L{\partial_t\p},\\
\mu{\partial_t\p}+ B\p +\partial_\p F(\p,\t)=0,
\end{dcases}\label{eq:strong_nonlocalAC}
\end{align}
which is a nonlocal analogue of an isotropic version of the local model for the solidification of pure materials, see, e.g.,~\cite{kobayashi1993,wang1993,boettinger2002} and references therein,
\begin{align}
\begin{dcases}
{\partial_t\t}=  D\upDelta\t+L{\partial_t\p}\\
\mu{\partial_t\p}-\varepsilon^2\upDelta\p + F_\p^{\prime}(\p,\t)=0,
\end{dcases}\label{model_localAC}
\end{align}
where $\varepsilon$ is an interface parameter.

\subsection*{Potential}
The most common choice of $F(\p)$ is a smooth double well potential
\begin{equation}
F(\p,\t) = \frac{1}{4}\p^2(1-\p)^2+m(\t)\left(\frac{1}{3}\p^3-\frac{1}{2}\p^2\right),\label{potential_regular}
\end{equation}
which can be considered as a smooth approximation of the more physically relevant 
logarithmic potential 
\begin{equation}
F(\p,\t)=\frac{\cpot}{2}\p(1-\p)+\frac{\sigma}{2}\left(\p\ln(\p)+(1-\p)\ln(1-\p)\right)-\cm m(\t)\p,\quad\p\in(0,1),
\end{equation}
for $ 0<\sigma<\cpot$,  or the obstacle potential
 \begin{equation}
F(\p,\t):=\frac{\cpot}{2}\p(1-\p)+\mathbb{I}_{[0,1]}(\p)-\cm m(\t)\p,\quad\cpot>0. 
\label{potential_obstacle}
 \end{equation}
 Here,  $\mathbb{I}_{[0,1]}(\p)$ is a convex indicator function of an admissible range $[0,1]$ and $\cpot>0$ is an appropriate scaling constant; see Figure~\ref{fig:potentials} for the illustrations of those potentials.
 
In contrast to the logarithmic potential,  the obstacle potential allows the solution to attain pure phases $\p=0$ and $\p=1$.
However, this potential is not differentiable in a classical sense and one must resort to the notion of subdifferentials.  More specifically,  we define the generalized differential of $F(\p,\t)$ with respect to $\p$ by
\begin{equation}
\partial_\p F(\p,\t)=-\cpot\p+\frac{\cpot}{2}-\cm m(\t)+\partial\mathbb{I}_{[0,1]}(\p),
\label{eq:subdiff_F}
\end{equation}
where $\partial\mathbb{I}_{[0,1]}(\p)$ is a subdifferential of the indicator function
\begin{equation*}
\partial\mathbb{I}_{[0,1]}(\p)=\begin{cases}
(-\infty,0] & \text{if } \p=0,\\
0 &\text{for } \p\in(0,1),\\
[0,+\infty) &\text{if } \p=1.
\end{cases}\label{subdiff}
\end{equation*} 
Furthermore, it is necessary that $F$ in~\eqref{potnetial} admits two local minima at $\p=0$ and $\p=1$ irrespective of the value of the temperature $\t$.  To ensure this in case of a smooth potential the following first and second order optimality conditions should hold:
\begin{align*}
\frac{ \partial F}{\partial\p}|_{\p=0,1}=F_0^\prime(\p)+m(\t)g^\prime(\p)|_{\p=0,1}=0,\\
\frac{\partial^2 F}{\partial\p^2}|_{\p=0,1}=F_0^{\prime\prime}(\p)+m(\t)g^{\prime\prime}(\p)|_{\p=0,1}=0.
\end{align*}
In case of the regular potential~\eqref{potential_regular} the above holds true if $m(\t)$ is chosen such that $|m(\t)|<1/2$.  For a non-smooth potential with the bound constraints such as~\eqref{potential_obstacle} the necessary conditions are
\begin{align*}
0\in\partial_u F(\p,\t)|_{\p=0,1},
\end{align*}
which is equivalent to $\cpot/2-\cpot m(\t)\geq 0$ ($\p=0$) and $-\cpot/2-\cpot m(\t)\leq 0$ ($\p=1$).  In case of a strict inequality it is also a sufficient condition,  which reduces to the same requirement $|m(\t)|<1/2$.
 \begin{figure}[t!]
\begin{subfigure}{.3\textwidth}
  \centering
  \includegraphics[width=\textwidth]{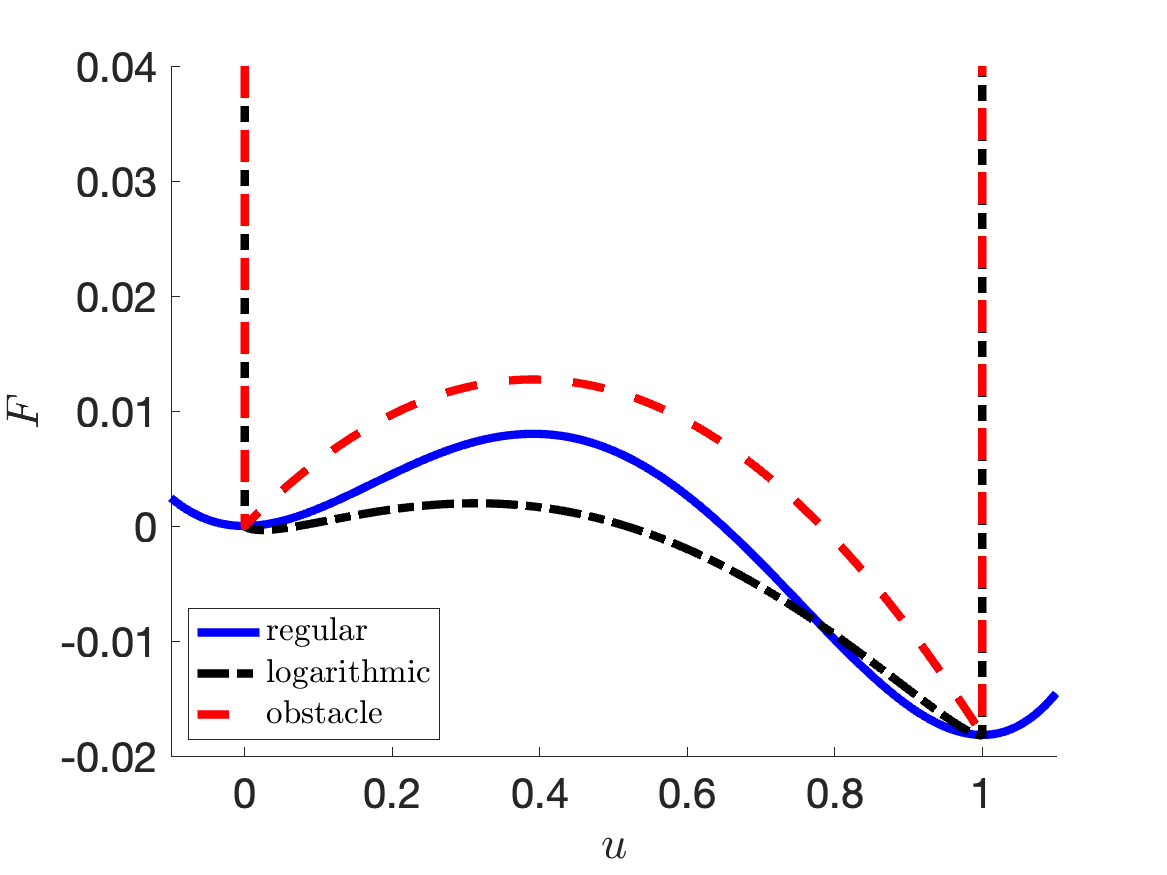}
  \caption{$m>0$ }
 \end{subfigure}
 \begin{subfigure}{.3\textwidth}
  \centering
   \includegraphics[width=\textwidth]{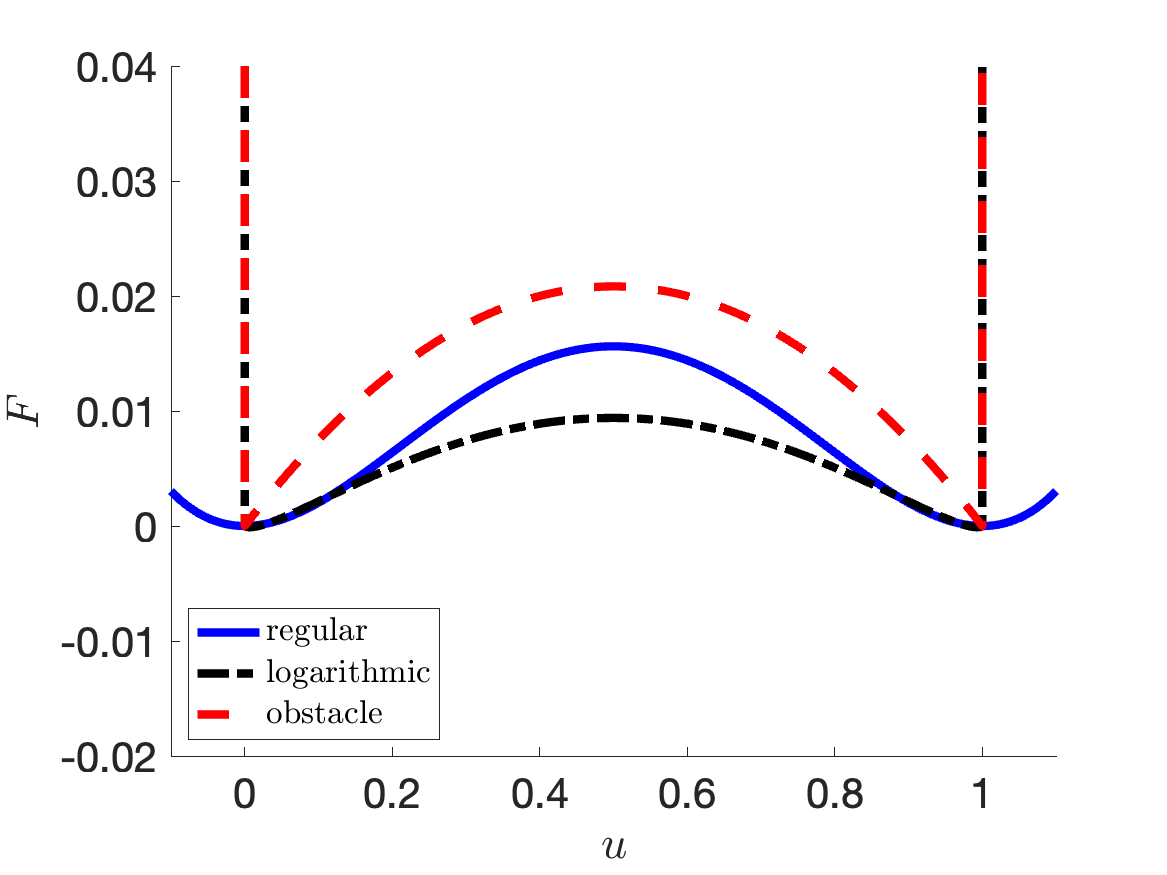}
  \caption{$m=0$}
 \end{subfigure}
 \begin{subfigure}{.3\textwidth}
  \centering
    \includegraphics[width=\textwidth]{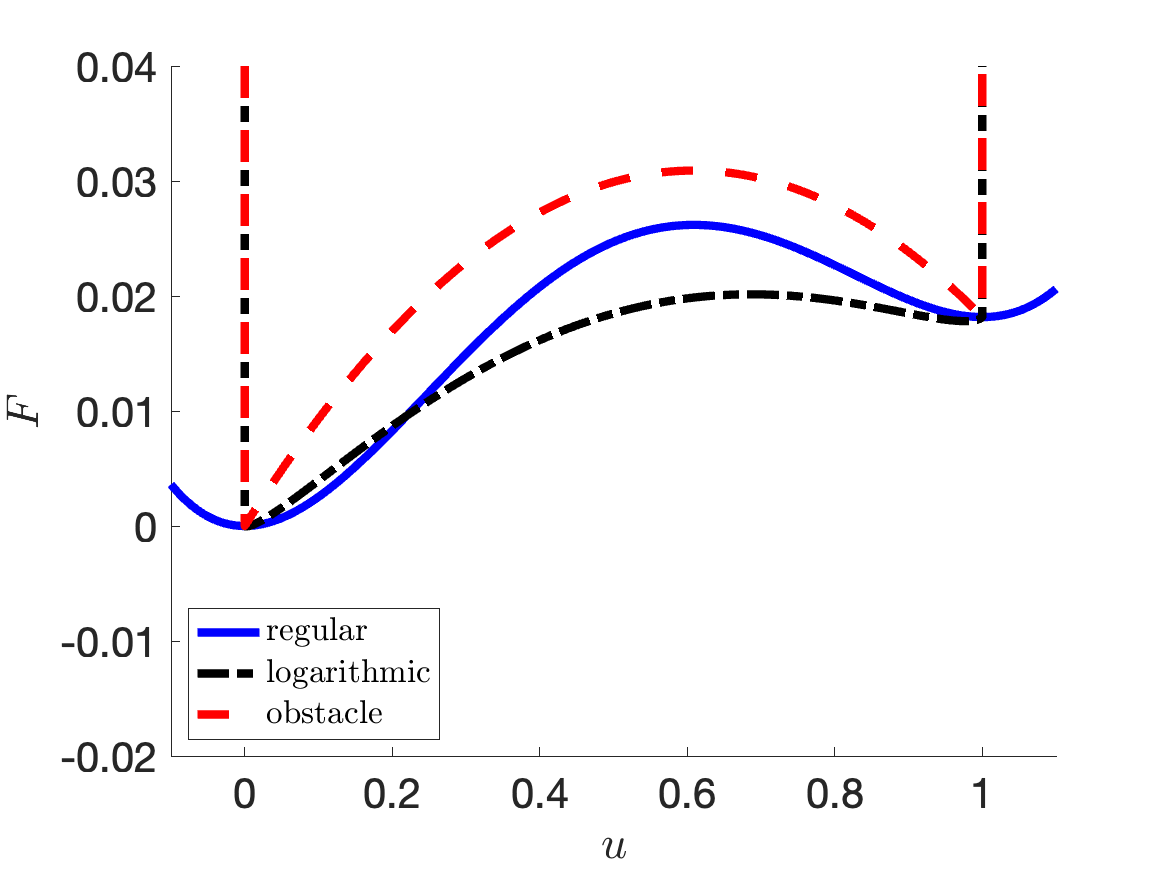}
  \caption{$m<0$}
 \end{subfigure}
 \caption{Illustration for the regular (blue), logarithmic (black) and obstacle potentials (red).}\label{fig:potentials}\end{figure} 
 
The coupling term $m(\t)$,  is usually chosen  as a function of a temperature,  which is almost linear around $\t=\t_{e}$.  In particular, in the Caginalp model~\cite{caginalp1986analysis} it is set to be linear, whereas in the Kobayashi model~\cite{kobayashi1993} it is chosen as a nonlinear function of~$\t$:
\begin{equation}
m(\t)=\left(\frac{\alpha}{\pi}\right)\tan^{-1}\left(\rho(\t_e-\t)\right),\label{mt_kobayashi}
\end{equation}
with an assumption that $0<\alpha<1$, which ensures that $|m(\t)|<1/2$. 
In the present exposition we assume that the function $m$ in~\eqref{potential_obstacle} is uniformly bounded and uniformly Lipschitz continuous in $\R$,  i.e.,  
\begin{equation}
|m(v)|\leq {1}/{2}\qquad\text{and}\qquad |m^\prime(v)|\leq C_m, \quad\forall v\in\R.
\label{mt_condition}
\end{equation}
For example,  using $m$ from~\cite{kobayashi1993} as defined in~\eqref{mt_kobayashi} fulfills the above requirements.
 
Now, invoking the definition of the nonlocal operator~\eqref{operatorB} and $\partial_u F$~\eqref{eq:subdiff_F}, the system~\eqref{eq:strong_solid_nonlocalCH} can be stated as follows
\begin{align}
\begin{dcases}
{\partial_t\t} = \ D\upDelta\t+L{\partial_t\p},\\
\mu{\partial_t\p}+\A w=0\\
w= \param\p -\kernel*\p+\frac{\cpot}{2}-\cm m(\t)+\lambda,\quad\lambda\in \partial\mathbb{I}_{[0,1]}(\p),
\end{dcases}\label{eq:model_subdif_lagrange}
  \end{align}
where $\param=\int_{\DD}\kernel(\x-\xx)\d\xx-\cpot$ is the nonlocal interface parameter that plays a similar role to the interface parameter $\varepsilon^2$ in the local settings in~\eqref{model_localAC}, meaning that the larger values of $\xi$ lead to a more diffuse interface, cf.~\cite{BG2021CH}.  However, in contrast to the local settings where a thin interface is obtained only in the limiting case $\varepsilon\to 0$, here we show that the sharp interface with only pure phases can appear for $\xi=0$.

\section{Preliminaries}\label{sec:preliminaries}
We denote by $L^p(\D)$ and by $W^{k,p}(\D)$, $p\in[1,\infty]$, $k\in\mathbb{N}$, the usual Lebesgue and Sobolev spaces  endowed with the norms $\norm{\cdot}_{L^p(\D)}$ and $\norm{\cdot}_{W^{1,1}(\D)}$. 
We also denote by $\inner{\cdot,\cdot}$ and $\norml{\cdot}$ an inner product and a norm on $L^2(\D)$ and 
set $H^1(\D)=W^{1,2}(\D)$ with $\norm{\cdot}^2_{H^1(\D)}=\norml{\nabla v}
^2+\norml{v}^2$, and recall the Poincar\'e inequality
\begin{equation}
\norml{v-v^{\D}}\leq C_P\norml{\nabla v},\quad\forall v\in H^1(\D),\label{eq:poincare_neumann}
\end{equation}
where $v^{\D}:=\frac{1}{|\D|}\int_{\D}v(\x)\d\x$. 
For a reflexive Banach space $X$ we denote its dual by $X^\prime$ and by $\dual{\cdot,\cdot}_{X}$ the corresponding duality pairing.  
For $T>0$ we introduce space-time cylinders $\Q:=(0,\tend)\times\D$ and $\QC:=(0,\tend)\times(\DD)$.
For notational convenience, we set $\Vt:=H^1(\D)$ with $\norm{v}_{\Vt}:=\norm{v}_{H^1(\D)}$, and let
\begin{equation*}
\Va:=\{v\in L^2(\D)\colon \norm{v}_{\Va}<\infty\},\quad\text{where}\quad \norm{v}_{\Va}^2:=\beta\norml{\nabla v}^2+\norml{v}^2, \quad\beta\geq 0,
\end{equation*}
with $(u,v)_{\Va}:=\beta(\nabla u,\nabla v)+(u,v)$.  It is clear, that for $\beta=0$, $\Va=L^2(\D)$, while for $\beta>0$, $\Va\cong H^1(\D)$ and the following norm equivalence holds
\begin{equation}
\min(1,\beta)\norm{v}^2_{H^1(\D)}\leq \norm{v}^2_{\Va}\leq\max(1,\beta)\norm{v}^2_{H^1(\D)}.
\label{eq:Va_H1_relation}
\end{equation}
We introduce a nonlocal space $\Vb$ according to~\cite{BG2021CH}:
\begin{equation*}
\Vb:=
\left\{v\in L^2(\DD)\colon\NM v=0\;\text{on }\DI\right\}\cong L^2(\D),
\end{equation*}
and recall the following nonlocal Green's first identity~\cite{Dunonlocal2012}:
\begin{equation*}
(Bu,v) 
=\frac{1}{2}\int_{\DD}\int_{\DD}(u(\x)-u(\xx))(v(\x)-v(\xx))\kerneld\d\xx\d\x+(v,\NM u)_{L^2(\DI)}.
\end{equation*}
We note the condition $\mathcal{N}v=0$ on $\DI$ can be also understood as an ``exterior" nonlocal problem posed on $\DI$ with volume constraints $v=g$ on $\D$, for which it holds~\cite{BG2021CH}
\begin{equation}
\norm{v}_{L^2(\DI)}\leq C_I\norm{g},
\label{eq:bound_exterior}
\end{equation}
where the constant $C_I$ depends only on $\kernel$,  and $\D$, $\DI$. 
Next,  we introduce a Green's operator $\G:\Va^\prime\to\Va$ of the operator $\A=(I-\beta\upDelta)$, $\beta>0$, with the Neumann boundary conditions. Then,  for a given $u\in\Va^\prime$ it solves
\begin{equation}
\inner{\A\G u,v}=\duala{u,v},\quad\forall v\in\Va, 
\label{greensVa_eq}
\end{equation}
i.e., $\G=\A^{-1}$,  and the following holds true
\begin{align}
\norm{u}^2_{\Va^\prime}=\norm{\G u}^2_{\Va}=\duala{u,\G u}.
\label{greens_Va}
\end{align}
For kernels satisfying~\eqref{kernel_cond1} we recall the following properties~\cite[Proposition 3.3]{BG2021CH}.
\begin{proposition}\label{prop:kernel_prop}
For $\kernel$ satisfying~\eqref{kernel_cond1} and for $\eta>0$ there exists a function $\kernel_\eta:\R^n\to\R^+$ satisfying~\eqref{kernel_cond1} and $\kernel_\eta\in W^{1,1}(\R^n)$ and a constant $C_{\eta}>0$,  such that
\begin{equation}
\norm{\nabla\kernel_\eta}_{L^{1}(\R^n)}\leq C_\eta,\quad\text{and}\quad\norm{\kernel-\kernel_\eta}_{L^1(\R^n)}\leq\eta,\quad C_\eta>0.\label{eq:kernel_molifiers}
\end{equation}
Furthermore, the sequence $\kernel_\eta*u\to\kernel*u$ converges uniformly in $L^{\infty}(\DD)$ for any $u\in L^{\infty}(\DD)$, and the limiting function is continuous
\begin{equation}
\kernel*u\in C(\DD),\quad\forall u\in L^{\infty}(\DD).
\end{equation}
\end{proposition}

\subsection{Variational formulation}
We introduce a set of admissible solutions
\begin{equation}
\mathcal{K}:=\{v\in\Vb\colon 0\leq v\leq 1\;\; \text{a.\ e. in}\;\; \D\}.
\label{eq:set_K}
\end{equation}
Then,  a weak form of~\eqref{eq:model_subdif_lagrange} is to find $(\t(t),\p(t),w(t))\in\Vt\times\mathcal{K}\times\Va$ such that 
\begin{equation}
\begin{aligned}
\dualt{\partial_t\t(t),\phi}+D\inner{\nabla\t(t),\nabla\phi} - L{\dualt{\partial_t\p(t),\phi}}&=0,\quad\forall\phi\in\Vt,\\
\mu\duala{\partial_t\p(t),\psi}+{\duala{\A w(t),\psi}}&=0,\quad\forall\psi\in\Va,\\
 \inner{B\p(t)-\p(t)+{\cpot}/{2}-m(\t(t))-w(t) ,\zeta-\p(t)}&\geq 0,\quad\forall\zeta\in\mathcal{K},
\end{aligned}
 \label{CH_VI}
\end{equation}
subject to $u(0)=u^0\in\mathcal{K}$ and $\t(0)=\t^0\in\Vt$.  We define a positive cone 
\[
\M
=\{v\in L^2(\D)\colon v\geq 0\;\text{a.\ e. on}\;\D\},
\] 
and by introducing a Lagrange multiplier $\lambda(t)=\lambda_+(t)-\lambda_{-}(t)$ with $\lambda_{\pm}(t)\in M$,  
we can restate the above variational inequality as a system of complementarity conditions: Find $(\t(t),\p(t),w(t),\lambda_{\pm}(t))\in\Vt\times\Vb\times\Va\times M$ such that
\begin{equation}
\begin{aligned}
 \inner{B\p(t)-\p(t)+{\cpot}/{2}-\cm m(\t(t))-w(t)+\lambda(t),\zeta}&= 0,\quad\forall\zeta\in\Vb,\\
 \inner{\eta-\lambda_{+}(t),1-\p(t)}\geq 0,\quad \inner{\eta-\lambda_{-}(t),\p(t)}&\geq 0,\quad\forall\eta\in\M.
\end{aligned}
 \label{CH_VI_SP_inequality}
\end{equation}
Furthermore,  by taking $\eta=\pm\lambda_{\pm}(t)$ and $\eta=0$ in the above inequalities it is easy to show that the following property holds true for all $\eta\in\M$:
\begin{equation}
(\eta,1-\p(t))\geq 0,\;\; (\lambda_{+}(t),1-\p(t))=0\;\;\text{and}\;\; (\eta,\p(t))\geq 0,\;\; (\lambda_{-}(t),\p(t))=0.
\label{CH_VI_compl_split}
\end{equation}
An equivalent time-integrated version of~\eqref{CH_VI},\eqref{CH_VI_SP_inequality} is to find $\t\in L^2(0,\tend;\Vt)$,  $\p\in L^{\infty}(0,\tend;\Vb)$,  $w\in L^2(0,\tend;\Va)$, $\lambda_{\pm}\in L_{+}^2(\Q)=\{v\in L^2(\Q)\colon v\geq 0\;\text{a.e.}\;\Q\}$ such that $\partial_t\p\in L^2(0,T;\Va^\prime)$, $\partial_t\t\in L^2(0,T;\Vt^\prime)$, $u(0)=u^0$,  $\t(0)=\t^0$ and
\begin{equation}
\begin{aligned}
\int_{0}^{\tend}\left(\dualt{\partial_t\t,\phi}+D\inner{\nabla\t,\nabla\phi} - L{\dualt{\partial_t\p,\phi}}\right)\d t&=0,\quad\forall\phi\in L^2(0,\tend;\Vt),\\
\int_{0}^{\tend}\left(\mu\duala{\partial_t\p,\psi}+{\duala{\A w,\psi}}\right)\d t&=0,\quad\forall\psi\in L^2(0,\tend;\Va),\\
  \inner{B\p-\p+{\cpot}/{2}-\cm m(\t)-w(t)+\lambda,\zeta}_{L^2(\Q)}&= 0,\quad\forall\zeta\in L^2(0,\tend;\Vb),\\
 \inner{\eta-\lambda_{+},1-\p}_{L^2(\Q)}\geq 0,\quad \inner{\eta-\lambda_{-},\p}_{L^2(\Q)}&\geq 0,\quad\forall\eta\in L^2_+(\Q).
\end{aligned}\label{CH_VI_time}
\end{equation}

\section{Well-posedness of a time-discrete formulation}\label{sec:time-discrete}
{Next, we discretize~\eqref{CH_VI} in time and analyze the corresponding problem using an optimization approach.} 

For $T>0$ and $K\in\mathbb{N}$, we define $\tt=T/K$, $t^k=k\tt$, $k=1,\dots,K$.
Then, given $\t^0\in\Vt$, $\p^0\in\mathcal{K}$, we seek $(\t^k,\p^k,w^k)\in\Vt\times\mathcal{K}\times\Va$ such that
\begin{subequations}
\begin{align}
\inner{\t^k-\t^{k-1},\phi}+D\tt\inner{\nabla\t^k,\nabla\phi} -{L}\inner{\p^{k}-\p^{k-1},\phi}&=0,\quad\forall\phi\in\Vt \label{CH_VI_tk1}\\
{\mu}\inner{\p^k-\p^{k-1},\psi}+{\tt}\inner{w^k,\psi}+\beta{\tt}\inner{\nabla w^k,\nabla\psi}&=0,\quad\forall\psi\in\Va,\label{CH_VI_tk2}\\
\inner{B\p^k-\cpot\p^k+{\cpot}/{2}-\cm m(\t^{k-1}) - w^k,\zeta-\p^k}&\geq 0,\quad\forall\zeta\in\mathcal{K}\label{CH_VI_tk3}.
\end{align}\label{CH_VI_tk}\end{subequations}
Due to an explicit discretization of the coupling term $m(\t^{k-1})$, ~\eqref{CH_VI_tk} forms a decoupled system of equations. Furthermore, introducing $\lambda^k:=\lambda^k_{+}-\lambda^k_{-}$, $\lambda^k_{\pm}\in\M$ we can reformulate~\eqref{CH_VI_tk3} as a system of complementarity conditions:
\begin{subequations}\begin{align}
\inner{B\p^k-\cpot\p^k+{\cpot}/{2}-\cm m(\t^{k-1}) - w^k+\lambda^k,\zeta}&=0,\quad\forall\zeta\in\Vb,\label{CH_VI_tk_SP3}\\
\inner{\eta-\lambda^k_{+},1-\p^k}\geq 0,\quad \inner{\eta-\lambda^k_{-},\p^k}&\geq 0,\quad\forall\eta\in\M.\label{CH_VI_tk_SP4}
\end{align}\label{CH_VI_tk_SP}\end{subequations}
Then, this together with~\eqref{CH_VI_tk2} at each time step is the system of the Karush–Kuhn–\ Tucker (KKT) optimality conditions for the following minimization problem:
\begin{equation}
\min_{\p\in\mathcal{K}} \J(\p):=
\frac{\param}{2}\norml{\p}^2-\frac{1}{2}\inner{\kernel*\p,\p}+\frac{\mu}{2\tt}\norm{\p-\p^{k-1}}^2_{\Va^\prime}
+\frac12\inner{{\cpot} -2\cm m(\t^{k-1}),\p},
\label{min_problem}
\end{equation}
with $\xi:=\cker-\cpot$. 
Thus, for $k=1,\dots,K$, under appropriate conditions on the time step $\tt$ (as stated in Theorem~\ref{thm:wellpos_tk}) the problem~\eqref{CH_VI_tk_SP} can be equivalently expressed~as 
\begin{equation}
\begin{aligned}
\inner{\t^k-\t^{k-1},\phi}+D{\tt}\inner{\nabla\t^k,\nabla\phi} -{L}\inner{\p^{k}-\p^{k-1},\phi}&=0,\quad\forall\phi\in\Vt\\
u^{k}=\argmin_{\p\in\mathcal{K}} \J(\p).
\end{aligned}\label{min_problem_t}
\end{equation}

\begin{theorem}[Existence and uniqueness of a time-discrete solution]\label{thm:wellpos_tk}
Let $\kernel$ satisfy~\eqref{kernel_cond1},  then for $\xi\geq 0$ there exists a solution of~\eqref{min_problem_t}.  Furthermore,  for $\beta>0$,  $\xi>0$ and $\tt<{2\xi\mu}/((C_\eta^2+\beta\hat{C}_\eta^2)(1+C^2_I))$ or for $\beta=0$,  $\xi\geq 0$ and $\tt<\mu/(\ckerr(1+C_I^2)-\xi)$ the solution is unique.  Here,  the constants $\ckerr$, $C_\eta$, $\hat{C}_\eta$ and $C_I$ are as in~\eqref{cker},~\eqref{eq:kernel_molifiers} and~\eqref{eq:bound_exterior}, respectively, {and $\eta=\xi/(2+2C_I^2)$.  }
\end{theorem}
\begin{proof}
Since the system~\eqref{min_problem_t} is a decoupled system of equations we can analyze existence and uniqueness of each variable separately.  For the temperature $\t$ existence and uniqueness follows a standard Lax-Milgram argument. 
Invoking a uniform boundedness on $m(\t^{k-1})$~\eqref{mt_condition} the proof of existence of the minimizer in~\eqref{min_problem_t} follows similar lines as in~\cite[Theorem 3.1]{BG2021CH} and we show only the proof for the uniqueness of the corresponding minimizer. 
First, let $\beta>0$ and setting $w^k=-(\mu/\tau)\G(\p^k-\p^{k-1})$, where $\G$ is defined as in~\eqref{greensVa_eq}, we can express~\eqref{CH_VI_tk2}--\eqref{CH_VI_tk3} as follows:
\begin{equation}
\left(\xi\p^k-\kernel*\p^k+\frac{\cpot}{2}-c_m m(\t^{k-1})+\frac{\mu}{\tt}\G(\p^{k}-\p^{k-1}),\zeta-\p^k\right)\geq 0. 
\end{equation}
Using a contradiction argument,  we assume $\p^k_1$ and $\p^{k}_2$ are two solutions, and $\pU^k=\p^k_1-\p^k_2$.  Now, taking $\zeta=\p_2^k$ for $\p_1^k$ and $\zeta=\p_1^k$ for $\p_2^k$ in the above variaitonal inequality,  adding both together and using~\eqref{greens_Va} leads to
\begin{align*}
\xi\norml{\pU^k}^2-(\kernel*\pU^k,\pU^k)+\frac{\mu}{\tt}(\G(\pU^k),\pU^k)=  \xi\norml{\pU^k}^2-(\kernel*\pU^k,\pU^k)+{\frac{\mu}{\tt}\norml{\pU^k}_{\Va^\prime}^2}\leq 0. 
\end{align*}
Using Young's and Cauchy inequalities we can bound the convolution by
\begin{multline*}
\norms{(\kernel*\pU^k,\pU^k)}\leq\norms{\duala{\kernel*\pU^k,\pU^k}}\\
\leq\norm{\kernel_{\eta}*\pU^k}_{\Va}\norm{\pU^k}_{\Va^\prime}+\norm{\kernel-\kernel_{\eta}}_{L^1(\R^n)}\norm{\pU^k}^2_{L^2(\D\cup\DI)}\\
\leq
\frac{\tt}{4\mu}\norm{\kernel_{\eta}*\pU^k}^2_{\Va}+\frac{\mu}{\tt}\norm{\pU^k}^2_{\Va^\prime}+ \norm{\kernel-\kernel_\eta}_{L^1(\R^n)}\norm{\pU^k}^2_{L^2(\D\cup\DI)}\\
=\left(\frac{\tt}{4\mu}\norm{\kernel_\eta}^2_{L^1(\R^n)}+\frac{\tt\beta}{4\mu}\norm{\nabla\kernel_{\eta}}^2_{L^1(\R^n)}+\norm{\kernel-\kernel_{\eta}}_{L^1(\R^n)}\right)\norm{\pU^k}^2_{L^2(\D\cup\DI)}+\frac{\mu}{\tt}\norm{\pU^k}^2_{\Va^\prime}.
\end{multline*}
Now, combining this with the previous estimate and using~\eqref{eq:bound_exterior} and~\eqref{eq:kernel_molifiers} we obtain
\begin{align*}
\left(\xi-\left(1+C_I^2\right)\left({\tt C_\eta^2}/{4\mu}+{\tt\beta\hat{C}_\eta^2}/{4\mu}+\eta\right)\right)\norm{U^k}^2\leq 0.
\end{align*}
Then,  choosing $\eta=\xi/(2+2C_I^2)$ and
$\tt<{2\xi\mu}/((C_\eta^2+\beta\hat{C}_\eta^2)(1+C^2_I))$
we obtain that $U^k=0$ and, hence, $\p^k_1=\p^k_2$. For $\beta=0$, $0\leq\xi<\ckerr(1+C_I^2)$ and $\tt<\mu/(\ckerr(1+C_I^2)-\xi)$ we obtain a uniqueness of the solution following similar arguments as above. 
\end{proof}

\subsection{Sharp interfaces}
In this section we study the conditions under which the phase-field variable can attain sharp interfaces. 

\begin{theorem}[Sharp interfaces, $\beta>0$]
Let $(\t^k,\p^k,w^k)$ be a solution of~\eqref{CH_VI_tk} and $\kernel$ satisfies~\eqref{kernel_cond1}. Then for $\xi>0$ and $\beta> 0$ at each time step the following projection formula holds true for the phase-field variable $\p^k$:
\begin{align}
&\p^k=P_{[0,1]}\left(\frac{1}{\xi}g^k\right)=
\begin{cases}
1&\text{if}\quad g^k\geq\xi,\\
0&\text{if}\quad g^k\leq 0,\\
g^k/\xi&\text{if}\quad g^k\in (0,\xi),
\end{cases}\label{proj_tk_CH}\\
\text{with} \quad &g^k:=w^k+\kernel*\p^k+\cm m(\t^{k-1})-\cpot/2.\label{eq:g_proj_regularity}
\end{align}
Moreover,  for $\xi=0$ and $\beta>0$ we have
\begin{align*}
\p^k\in \frac{1}{2}\left(1+{\rm sign}(g^k)\right)=
\begin{cases}
\{1\}\quad\text{if}\quad g^k>0,\\
\{0\}\quad\text{if}\quad g^k<0,\\
[0,1]\quad\text{if}\quad g^k=0,
\end{cases}
\end{align*}
where ${\rm sign}(0)=[0,1]$.  Furthermore,  if $|\{g^k\}|=0$ a.e.  in $\D$,  i.e., the set when $g^k=0$ has measure zero, then $\p^k$ can be discontinuous and attains only $\p^k\in\{0,1\}$.
\end{theorem}
\begin{proof}
The proof follows similar strategy as in~\cite{BG2021CH} and we skip it for brevity. 
\end{proof}

\begin{corollary}[Projection formula, $\beta=0$]\label{corr:sharp_interfacesAC}
For the Allen-Cahn case ($\beta=0$) and for $\xi\geq 0$ a phase-field variable $\p^k$ also admits a representation formula:
\begin{equation}
\p^k:=P_{[0,1]}\left(\frac{1}{\mu/\tt+\xi}g^k\right)\quad\text{with}\quad g^k:=\frac{\mu}{\tt}\p^{k-1}+\kernel*\p^k+\cm m(\t^{k-1})-\frac{\cpot}{2}.\label{proj_tk_AC}
\end{equation}
{However,  since $\mu/\tt+\xi> 0$,  no sharp interfaces with a jump-discontinuity can occur in this case, unless the solution at a previous time step already has a jump discontinuity. 
Furthermore,  in the steady-state the corresponding solution admits:
\begin{align*}
\p= P_{[0,1]}\left(\frac{1}{\xi}\left(\kernel*\p-{\cpot}/{2}+\cpot m(\t)\right)\right),\quad\xi>0.
\end{align*}
Then,  if $\xi=0$ and $|\left\{\kernel*\p-{\cpot}/{2}+\cpot m(\t)\right\}|=0$ almost everywhere,  the solution $\p$ can admit only pure phases, $\p\in\{0,1\}$.}
\end{corollary}
\begin{remark}
We note that the projection formula~\eqref{proj_tk_CH} or~\eqref{proj_tk_AC} provides a useful insight into stability and regularity properties of the solution. Furthermore, for $\beta=0$ (the Allen-Cahn case), 
using explicit discretization of the convolution and coupling terms, the formula~\eqref{proj_tk_AC} can be used to directly evaluate the solution $\p^k$, provided $\kernel*\p^{k-1}$, without the need to solve the corresponding nonlocal system (see Section~\ref{sec:discretization}).  
\end{remark}

\begin{theorem}[Higher regularity]
Let $(\t^k,\p^k,w^k)$ be the solution of~\eqref{CH_VI_tk} and $\kernel$ satisfies~\eqref{kernel_cond1} and $\kernel\in W^{1,1}(\R^n)$. Then if $\xi>0$ and $u_0\in H^1(\D)$ (for $\beta = 0$), we obtain that $\p^k\in H^1(\D)$ for all $k=1,\dots,K$, $\beta\geq 0$, and the following holds
\begin{equation}
\norml{\nabla\p^k}\leq \frac{C}{\xi}\norm{g^k}_{H^1(\D)},
\label{reg_u_tk}
\end{equation}
where $C = 1$ for $\beta>0$ and $C=\tt\xi/(\mu+\tt\xi)<1$ for $\beta=0$, and $g^k$ is defined as in~\eqref{eq:g_proj_regularity} and~\eqref{proj_tk_AC} for $\beta>0$ and $\beta=0$, respectively. 
Additionally, for $\xi\geq 0$ and $\beta>0$ we also have that $\lambda_{\pm}^k$ in~\eqref{CH_VI_tk_SP} is in $H^1(\D)$ for $k=1,\dots,K$. 
\end{theorem}
\begin{proof}
Since $\p^k\in L^{\infty}(\R^n)$ and $\kernel\in W^{1,1}(\R^n)$ by Young's inequality we have that $\kernel*\p^k\in W^{1,\infty}(\R^n)$. Furthermore, since $\t^{k-1}\in H^1(\D)$,  from~\eqref{mt_condition} it follows that
\[
\norml{\nabla m(\t^{k-1})}\leq\norm{m^\prime(\t^{k-1})}_{L^\infty(\D)}\norml{\nabla\t^{k-1}}\leq C,
\]
and, hence, $m(\t^{k-1})\in H^1(\D)$. Now, let $\beta=0$ and $k=1$ then $g^1\in H^1(\D)$ and by the stability of the $L^2-$projection in $H^1$ it directly follows that 
\[
\norml{\nabla\p^1}\leq (1/\tt+\xi)^{-1}\norml{\nabla g^1}\leq C,\quad C>0,
\] 
and hence $\p^1\in H^1(\D)$. Applying an induction argument it is easy to see that $\p^{k}\in H^1(\D)$. The proof for $\beta>0$ follows similar lines. 
The regularity of the Lagrange multiplier $\lambda^k\in H^1(\D)$ for $\beta>0$ follows directly by invoking the regularities of $\p^k, w^k,\t^{k-1}\in H^1(\D)$ and the fact that 
$\lambda^k=-\xi\p^k+\kernel*\p^k-\cpot/2+\cm m(\t^{k-1})+w^k$.
\end{proof}

\section{Continuous problem}\label{sec:continuous_problem}
In this section, we establish the existence of the solution of problem~\eqref{CH_VI} by using a Rothe method and studying the limit $\tt\to 0$ in~\eqref{CH_VI_tk}. We also discuss the conditions when sharp interfaces occur. 

\subsection{Interpolants}
Let ${X}$ be either $\Vb$, $L^2(\D)$, $\Va$ or $\Vt$ and  
for a given sequence of functions $\{z^k\}_{k=1}^K\subset X$ we introduce 
piecewise constant and piecewise linear interpolants:  
\begin{equation}
\begin{aligned}
\inp{z}(t):=&\ z^k,\quad \inpl{z}(t):=z^{k-1}, &t\in(t_{k-1},t_k],\quad k=1,\dots,K,\\
\inpt{z}(t):=&\ \frac{t-t_{k-1}}{\tt} z^{k}+\frac{t_k-t}{\tt}z^{k-1}, &t\in[t_{k-1},t_k],\quad k=1,\dots,K,
\end{aligned}\label{eq:t_interpolant}
\end{equation}
where $\tt=\tend/K$. Then, $\partial_t\inpt{z}=(z^{k}-z^{k-1})/\tt$, $t\in[t_{k-1},t_k]$, $k=1,\dots,K$ and we have
\begin{align*}
\norm{\inp{z}}^2_{L^2(0,\tend;X)}=\tt\sum_{k=1}^K\norm{z^k}_{X}^2,\quad
\norm{\inpl{z}}^2_{L^2(0,\tend;X)}=\tt\sum_{k=0}^{K-1}\norm{z^k}_{X}^2,\\
\norm{\partial_t\inpt{z}}^2_{L^2(0,\tend;X)}=\tt\sum_{k=1}^K\norm{(z^k-z^{k-1})/\tt}_{X}^2=\frac{1}{\tt}\sum_{k=1}^K\norm{z^k-z^{k-1}}_{X}^2.
\end{align*}
The following holds true for the above interpolants~\eqref{eq:t_interpolant}, (see Proposition~3.9 in~\cite{colli2019well}):
\begin{align}
&\norm{\inpt{z}}^2_{L^2(0,\tend;Z)}\leq \tt\norm{z_0}^2_{Z}+2\norm{\inp{z}}^2_{L^2(0,\tend;Z)},\label{prop_interp1}\\
&\norm{\inp{z}-\inpt{z}}^2_{L^2(0,\tend;Z)}=\frac{\tt^2}{3}\norm{\partial_t\inpt{z}}^2_{L^2(0,\tend;Z)},\label{prop_interp2}
\end{align}
where $Z=H^1(\D)$ or $Z=L^2(\D)$.  We also make use of the following property
\begin{multline}
\norm{\inpl{z}-\inpt{z}}^2_{L^2(\Q)}\leq \tt\norm{\partial_t\inpt{z}}^2_{L^2(0,\tend;(H^1(\D))^\prime)}
+\tt\norm{\inp{z}}^2_{L^2(0,\tend;H^1(\D))}+\tt^2\norm{z^0}_{H^1(\D)}^2,\label{prop_interp3}
\end{multline}
which follows directly by realizing that
\begin{multline*}
\norm{\inpl{z}-\inpt{z}}^2_{L^2(\Q)}=\sum_{k=1}^K\int_{t_{k-1}}^{t_k}\frac{(t-t_{k-1})^2}{\tt^2}\norml{z^{k-1}-z^k}^2\d t=\frac{\tt}{3}\sum_{k=1}^K\norml{z^{k-1}-z^k}^2\\
=\frac{\tt^2}{3}\sum_{k=1}^K\dualh{\frac{z^{k}-z^{k-1}}{\tt},z^k-z^{k-1}}
\leq \frac{\tt^2}{3}\sum_{k=1}^K\norm{\partial_t\inpt{z}}_{(H^1(\D))^\prime}\norm{z^{k}-z^{k-1}}_{H^1(\D)}\\
\leq \tt\norm{\partial_t\inpt{z}}^2_{L^2(0,\tend;(H^1(\D))^\prime)}+\tt\norm{\inp{z}}^2_{L^2(0,\tend;H^1(\D))}+\tt^2\norm{z^0}_{H^1(\D)}^2.
\end{multline*}
\subsection{Existence of a solution}
Now, we can rewrite the system\eqref{CH_VI_tk1}--\eqref{CH_VI_tk2},~\eqref{CH_VI_tk_SP} in terms of the interpolants as follows
\begin{subequations}
\begin{align}
(\partial_t\inpt{\t}(t),\phi)+D\inner{\nabla\inp{\t}(t),\nabla\phi} -L\inner{\partial_t\inpt{\p}(t),\phi}&=0,\quad\forall\phi\in\Vt\\
{\mu}\inner{\partial_t \inpt{\p}(t),\psi}+\inner{\inp{w}(t),\psi}+\beta\inner{\nabla\inp{w}(t),\nabla\psi}&=0,\quad\forall\psi\in\Va,\label{CH_VI_tk_int_b}\\
\inner{B\inp{\p}(t)-\cpot\inp{\p}(t)+{\cpot}/{2}-{\cm m(\inpl{\t}(t))} - \inp{w}(t)+\inp{\lambda}(t),\zeta}&=0,\quad\forall\zeta\in\Vb,\label{CH_VI_tk_int_c}\\
\inner{\eta-\inpp{\lambda}(t),1-\inp{\p}(t)}\geq 0,\quad \inner{\eta-\inpm{\lambda}(t),\inp{\p}(t)}&\geq 0,\quad\forall\eta\in\M.
\label{CH_VI_tk_int_d}
\end{align}\label{CH_VI_tk_int}\end{subequations}
By analyzing the system above for $\tt\to 0$ we derive the following existence result.

\begin{theorem}
Let $\kernel\in W^{1,1}(\R^n)$ and satisfies~\eqref{kernel_cond1},  and let $\t^0\in\Vt$,  $\p^0\in\mathcal{K}$. Then for $\xi\geq 0$, 
$0\leq \beta\leq\min\left\{\frac{\mu D}{2\cpot C_m L},\frac{\mu^2 D^2}{4\cpot^2C_P^2C_m^2L^2}\right\}$ there exist a solution quadruple $(\t,\p, w,\lambda)$ solving~\eqref{CH_VI_time} and fulfilling the regularity 
\begin{align*}
\t\in  L^2(0,\tend;\Vt),\quad \p \in L^\infty(\QC),\quad w\in L^2(0,\tend;\Va),\quad \lambda\in L^2(0,\tend;\Va),\\
\text{and}\qquad \partial_t\t\in L^2(0,\tend; \Vt^\prime),\quad \partial_t\p\in L^2(0,\tend;\Va^\prime),
\end{align*}
where we recall $L^2(0,\tend;\Va)\cong L^2(0,\tend;H^1(\D))$ for $\beta>0$.
Additionally,  if i) $\xi>0$,  $\beta>0$ or  ii) $\xi\geq 0$,  $\beta=0$,  and $\p^0 \in H^1(\D)$,  then $\p$ admits an improved regularity
\begin{equation}
\p\in L^\infty(0,\tend;H^1(\D))\cap L^\infty(\QC).
\label{eq:higher_regularity_p}
\end{equation}
\end{theorem}
\begin{proof}
For compactness of the presentation in some cases we will use a generic constant $C>0$ in the estimates,  which is independent of $\tt$. The inclusion $\p\in L^\infty(\QC)$ follows immediately from the fact that $0\leq \inp{\p}(t)\leq 1$ for a.e. $t\in(0,\tend)$. 

\begin{itemize}
\item \emph{Estimates for the time derivative of the phase-field variable $\partial_t\inpt{\p}$.}
\end{itemize}

First, we consider the following estimate for the discrete time-derivative,
\begin{multline*}
\frac{\mu}{\tt}\norm{\p^k-\p^{k-1}}^2_{\Va^\prime}\leq
{\xi}\norm{\p^{k-1}}^2-{\xi}\norm{\p^{k}}^2+\inner{\kernel*\p^{k},\p^k}-\inner{\kernel*\p^{k-1},\p^{k-1}}\\
+
\inner{\cpot-2\cm m(\t^{k-1}),\p^{k-1}-\p^k},
\end{multline*}
which is obtained from the fact that $\p^k$ is a solution of~\eqref{min_problem} and $\J(\p^k)\leq \J(\p^{k-1})$, $k=1,\dots,K$.
From the above expression using Young's inequality we further obtain
\begin{multline}
\mu\norm{\partial_t\inpt{\p}}^2_{L^2(0,T;\Va^\prime)}=
\frac{\mu}{\tt}\sum_{k=1}^K\norm{\p^{k}-\p^{k-1}}_{\Va^\prime}^2
\leq \xi\norm{\p^0}^2-\xi\norm{\p^K}^2\\
+\norm{\p^0}\norm{\gamma*\p^0}+\norm{\p^K}\norm{\gamma*\p^K}+\sum_{k=1}^K\norms{\duala{\cpot-2\cm m(\t^{k-1}),\p^{k-1}-\p^k}}\\
\leq(\xi+\ckerr)\norm{\p^0}_{L^2(\DD)}^2+\ckerr\norm{\p^K}^2_{L^2(\DD)}\\
+{\frac{\mu}{2}}\norm{\partial_t\inpt{\p}}^2_{L^2(0,T;\Va^\prime)}
+{\frac{\cpot^2}{2\mu}}\norm{1-2 m(\inpl{\t})}_{L^2(0,\tend;\Va)}^2,
\label{eq:proof_dt_estimate}
\end{multline}
where $\ckerr$ as in~\eqref{cker}. Next, we estimate the last term in~\eqref{eq:proof_dt_estimate}:
\begin{multline}
\norm{1-2 m(\inpl{\t})}^2_{L^2(0,\tend;\Va)}
=
\norm{1-2 m(\inpl{\t}(t))}^2_{L^2(\Q)}+4\beta\norm{\nabla m(\inpl{\t})}_{L^2(\Q)}^2\\
\leq C+4\beta C_m^2\left(\tt\norml{\nabla\t^0}^2+\norm{\nabla\inp{\t}}^2_{L^2(\Q)}\right),\label{eq:proof_est1}
\end{multline}
where in the last estimate we have used~\eqref{mt_condition} and the relationship
\begin{align}
\norm{\nabla\inpl{\t}}^2_{L^2(\Q)}\leq\tt\sum_{k=1}^{K}\norml{\nabla\t^k}^2+\tt\norml{\nabla\t^0}^2= \norm{\nabla\inp{\t}}^2_{L^2(\Q)}+\tt\norml{\nabla\t^0}^2.\label{eq:t_interp_rel}
\end{align}
To bound the last term in~\eqref{eq:proof_est1} we introduce the orthogonal projection $\pt:L^2(\D)\to L^2(\D)$, $\pt(v):=v-v^{\D}$, where $v^{\D}:=\frac{1}{|\D|}\int_{\D} v\d\x$. 
We also have stability estimates:
\begin{equation}
\norm{\pt(v)}_{\mathcal{X}}\leq\norm{v}_{\mathcal{X}},\quad\forall v\in\mathcal{X},\label{eq:proj_stability}
\end{equation}
where $\mathcal{X}$ is either $L^2$, $\Va$ or $\Va^\prime$.  For $\mathcal{X}\in\{L^2, \Va\}$ the property follows directly,  while for $\Va^\prime$ this is obtained by a duality argument. Then, by the linearity of the projection operator and the fact that $\nabla P(\t^k)=\nabla\t^k$,  from~\eqref{CH_VI_tk} we obtain the following
\begin{equation*}
\inner{\pt(\t^k)-\pt(\t^{k-1}),\phi}+D\tt\inner{\nabla\t^k,\nabla\phi} -{L}\inner{\pt(\p^{k}-\p^{k-1}),\phi}=0,\quad\forall\phi\in\Vt. 
\end{equation*}
Taking $\phi=2\tt\pt(\t^k)$ and using the relation $2a(a-b)=a^2+(a-b)^2-b^2$ leads to
\begin{multline*}
\norml{\pt(\t^k)}^2+2\tt D\norml{\nabla\t^k}^2
=\norml{\pt(\t^{k-1})}^2-\norml{\pt(\t^k)-\pt(\t^{k-1})}^2\\+2L\inner{\pt(\p^k-\p^{k-1}),\pt(\t^k)}
\leq \norml{\pt(\t^{k-1})}^2+2L\norm{\pt(\p^k-\p^{k-1})}_{\Va^\prime}\norm{\pt(\t^k)}_{\Va}\\
\leq \norml{\pt(\t^{k-1})}^2+{\frac{L}{q}}\norm{\pt(\p^k-\p^{k-1})}_{\Va^\prime}^2+{Lq}\left(\beta\norml{\nabla P(\t^{k})}^2+\norml{P(\t^k)}^2\right)\\
\leq
\norml{\pt(\t^{k-1})}^2+{\frac{L}{q}}\norm{\p^k-\p^{k-1}}_{\Va^\prime}^2+{Lq}(\beta+C_P^2)\norml{\nabla\t^{k}}^2,
\end{multline*}
where in the last two estimates we have used Young's inequality with the constant $q>0$,~\eqref{eq:proj_stability}, and~\eqref{eq:poincare_neumann}. Now, taking a summation from $k=1\dots K$ we obtain
\begin{equation*}
(2\tt D-Lq(\beta+C_P^2))\sum_{k=1}^K\norml{\nabla\t^k}^2
\leq
\norml{\t^{0}}^2+{\frac{L}{q}}\sum_{k=1}^K\norm{\p^k-\p^{k-1}}_{\Va^\prime}^2,
\end{equation*}
and now choosing $q=\tt D/(L(\beta+C_P^2))$ leads to
\begin{equation}
\norm{\nabla\inp{\t}}^2_{L^2(\Q)}
\leq D^{-1}\norml{\t^{0}}^2+{D^{-2}L^2(\beta+C_P^2)}\norm{\partial_t\inpt{\p}}_{L^2(0,\tend;\Va^\prime)}^2.\label{eq:proof1_t}
\end{equation}
Using the above estimate and~\eqref{eq:proof_est1} in~\eqref{eq:proof_dt_estimate} we obtain
\begin{multline}
\frac{\mu}{2}\norm{\partial_t\inpt{\p}}^2_{L^2(0,T;\Va^\prime)}
\leq 
C\left(1+ \norm{\p^0}^2+\norml{\p^K}^2
+\norml{\t^0}^2+\tt\norml{\nabla{\t}^0}^2\right)\\
+\frac{2\beta C_m^2\cpot^2L^2(\beta+C_P^2)}{\mu D^2}\norm{\partial_t\inpt{\p}}^2_{L^2(0,\tend;\Va^\prime)}.
 \label{eq:proof_der_u}
\end{multline}
For 
$
\beta(\beta+C_P^2)<{\mu^2D^2}/({2\cpot^2C_m^2 L^2}),
$
or alternatively for $\beta\leq\min\left\{\frac{\mu D}{2\cpot C_m L},\frac{\mu^2 D^2}{4\cpot^2C_P^2C_m^2L^2}\right\}$ 
\begin{equation*}
\left(\frac{\mu}{2}-\frac{2\beta C_m^2\cpot^2L^2(\beta+C_P^2)}{\mu D^2}\right)\norm{\partial_t\inpt{\p}}^2_{L^2(0,T;\Va^\prime)}\leq  C,
\end{equation*}
where $C$ depends on $\p^0$, $\p^K$, $\t^0$ and, hence, we have $\partial_t\inpt{\p}\in L^2(0,T;\Va^\prime)$.

\begin{itemize}
\item \emph{Estimates for the temperature $\inp{\t}$ and $\partial_t\inp{\t}$.}
\end{itemize}
Using the regularity $\partial_t\inpt{\p}\in L^2(0,T;\Va^\prime)$, from~\eqref{eq:proof1_t} it immediately follows that 
\begin{equation}
\norm{\nabla\inp{\t}}^2_{L^2(\Q)}
\leq C\left(\norml{\t^{0}}^2+\norm{\partial_t\inpt{\p}}_{L^2(0,\tend;\Va^\prime)}^2\right)\leq C,\label{eq:proof_t_bound}
\end{equation}
and, hence, $\inp{\t}\in L^2(0,\tend;\Vt)$,  and from~\eqref{eq:t_interp_rel} it is also $\inpl{\t}\in L^2(0,\tend;\Vt)$. {Then from~\eqref{mt_condition} we have $m(\inp{\t}),m(\inpl{\t})\in L^2(0,\tend;\Vt)\cap L^\infty(\QC)$.} 
Next, from~\eqref{CH_VI_tk_int} we obtain
\begin{multline*}
\norms{\inner{\partial_t\inpt{\t}(t),\phi}}\leq D\norms{\inner{\nabla\inp{\t}(t),\nabla\phi}}+L\norms{\duala{\partial_t\inpt{\p}(t),\phi}}\\
\leq \left(D\norm{\nabla\inp{\t}(t)}+L\max\{1,\beta\}\norm{\partial_t\inpt{\p}(t)}_{\Va^\prime}\right)\norm{\phi}_{\Vt},
\end{multline*}
where we have used repeatedly Cauchy inequality and $\norm{\phi}_{\Va}\leq\max\{1,\beta\}\norm{\phi}_{\Vt}$. 
Now, dividing by $\norm{\phi}_{\Vt}\neq 0$, taking a supremum over $\norm{\phi}_{\Vt}\neq 0$, and then taking the square of the above estimate and integrating over $(0,T)$ leads to
\begin{equation}
\|\partial_t\inpt{\t}\|^2_{L^2(0,\tend;\Vt^\prime)}\leq C\left(\norm{\inp{\t}}_{L^2(0,\tend;\Vt)}^2
+\norm{\partial_t\inpt{\p}}_{L^2(0,\tend;\Va^\prime)}^2\right)\leq C\label{eq:proof_dtt_bound}
\end{equation}
and, hence, $\partial_t\inpt{\t}\in L^2(0,\tend;\Vt^\prime)$.
\begin{itemize}
\item \emph{Estimates for the chemical potential $\inp{w}$.}
\end{itemize}
Using the fact that $\inp{w}(t)=-{\mu}\G (\partial_t\inpt{\p}(t))$, where $\G:\Va^\prime\to\Va$ is defined as in~\eqref{greensVa_eq}, together with~\eqref{greens_Va} we obtain
\begin{equation*}
\norm{\inp{w}}^2_{L^2(0,\tend;\Va)}=\mu^2\norm{\G(\partial_t\inpt{\p})}^2_{L^2(0,\tend;\Va)}=\mu^2\norm{\partial_t\inpt{\p}}^2_{L^2(0,\tend;\Va^\prime)}\leq C,
\end{equation*}
and hence $\inp{w}\in L^2(0,\tend; \Va)\cong L^2(0,\tend;H^1(\D))$ (for $\beta>0$).

\begin{itemize}
\item \emph{Estimates for the Lagrange multiplier $\inp{\lambda}$.}
\end{itemize}

Consider~\eqref{CH_VI_tk_int_c} and recall that $B\inp{u}-\cpot\inp{u}=\xi\inp{u}-\kernel*\inp{u}$,  then 
\begin{align}
\inp{\lambda}(t) 
=-\xi\inp{\p}(t)+\kernel*\inp{\p}(t)-{\cpot}/{2}+{\cm m(\inpl{\t}(t))} + \inp{w}(t).\label{eq:proof1_lagrange_mult}
\end{align}
Using Youngs inequality,~\eqref{mt_condition} and the fact that $\p\in L^\infty(\QC)$ we can estimate
\begin{multline*}
\norml{\inp{\lambda}(t)}
\leq \norm{\kernel*\inp{\p}(t)}+\norm{\xi\inp{\p}(t)} + \norm{{\cm m(\inpl{\t}(t))}-{\cpot}/{2}}+\norm{\inp{w}(t)}
\leq C+\norm{\inp{w}(t)}_{\Va}.
\end{multline*}
Taking a square and integrating over $(0,\tend)$ leads to $\|\inp{\lambda}\|^2_{L^2(\Q)}\leq C$,
hence $\lambda\in L^2(Q)$.

\begin{itemize}
\item \emph{Higher regularity estimates for the phase-field variable $\inp{\p}$.}
\end{itemize}
We show that for $\beta>0$,  $\xi>0$ or $\beta=0$, $\xi\geq 0$, $\inp{\p}\in L^2(0,\tend;H^1(\D))$. For $\beta>0$,  $\xi>0$ the proof is similar as in~\cite{BG2021CH} and we consider the case $\beta=0$,  $\xi\geq 0$. Then, from~\eqref{proj_tk_AC} and~\eqref{reg_u_tk} it follows that for a.e. $t\in (0,1)$ the following holds
\begin{equation*}
\inp{\p}(t)=P_{[0,1]}\left(\frac{1}{\mu/\tt+\xi}\inp{g}(t)\right)\quad\text{and}\quad \norml{\nabla\inp{\p}(t)}\leq \frac{1}{\mu/\tt+\xi}\norml{\nabla\inp{g}(t)},
\end{equation*}
where $\inp{g}:=({\mu}/{\tt})\inpl{\p}+\kernel*\inp{\p}+{\cm m(\inpl{\t})}-\cpot/2$.
Then, using Young's inequality,~\eqref{mt_condition} and $\inpl{u}(t)\in L^{\infty}(\QC)$ we estimate
\begin{multline*}
\norml{\nabla\inp{\p}(t)}
\leq \frac{1}{\mu/\tt+\xi}\left(
\frac{\mu}{\tt}\norml{\nabla\inpl{\p}(t)}+\norml{\nabla(\kernel*\inp{\p}(t))}
+\norml{\nabla\left(\cm m(\inpl{\t}(t))-\cpot/2\right)}\right)\\
\leq \frac{1}{\mu/\tt+\xi}\left(
\frac{\mu}{\tt}\norml{\nabla\inpl{\p}(t)}+\norml{\nabla\kernel}_{L^1(\R^n)}\norml{\inp{\p}(t))}_{L^{2}(\DD)}
+\norm{\cm m^\prime(\inpl{\t}(t))}_{L^\infty(\D)}{\norml{\nabla\inpl{\t}(t)}}\right)\\
\leq \frac{1}{1+\xi\tt/\mu}\norml{\nabla\inpl{\p}(t)}+\frac{C}{\mu/\tt+\xi}\left(1+\norml{\nabla\inpl{\t}(t)}\right)
\leq \norml{\nabla\inpl{\p}(t)}+C\tt\left(1+\norml{\nabla\inpl{\t}(t)}\right),
\end{multline*}
where in the last estimate we have used that $1/(1+\xi\tt/\mu)\leq1$ and $ C/(\mu/\tt+\xi)\leq C\tt$, and $C>0$ is independent of $\tt$. From the above estimate for $k=1$ we deduce that
\begin{align*}
\norml{\nabla{\p}^1}\leq \norml{\nabla\p^0}+C\tt(1+\norml{\nabla\t^0})\leq C+C\tt(1+\norml{\nabla\t^0}),
\end{align*}
and by an induction argument we obtain that for $k=1,\dots,K$ it holds
\begin{align*}
\norml{\nabla{\p}^k}\leq \norml{\nabla\p^{k-1}}+C\tt\left(1+\norm{\nabla\t^{k-1}}\right)\leq C+C\tt\sum_{i=1}^k(1+\norml{\nabla\t^{i}}).
\end{align*}
Taking the supremum over $k$ in the above estimate leads to
\begin{multline*}
\esssup_{t\in(0,T)} \norm{\nabla\inp{\p}(t)}=\sup_{k}\norm{\nabla\p^k}\leq C+C\sup_k\tt\sum_{i=1}^k\left(1+\norm{\nabla\t^i}\right)\\
= C+C\tt K+\tt\sum_{i=1}^K\norm{\nabla\t^i}\leq C\left(1+\norm{\inp{\t}}_{L^1(0,T;H^1(\D))}\right),
\end{multline*}
which implies that $\inp{\p}\in L^{\infty}(0,\tend;H^1(\D))$ for $\beta\geq 0$ and $\xi\geq 0$.

\begin{itemize}
\item \emph{Higher regularity estimates for the Lagrange multiplier $\inp{\lambda}$.}
\end{itemize}
For $\beta>0$ and $\xi=0$, in general,  we can not expect higher regularity of the phase-field variable $\inp{u}$.  However,  we can demonstrate that the corresponding Lagrange multiplier $\inp{\lambda}(t)\in\Va$ for $\xi\geq 0$ and $\beta>0$.  
Indeed,  let $\xi\geq 0$ then from~\eqref{CH_VI_tk_int_c},~\eqref{eq:proof1_lagrange_mult}, using the regularity of $\inp{w}\in L^2(0,\tend;\Va)\cong L^2(0,\tend;H^1(\D))$, $\inp{\p}\in L^2(0,\tend; H^1(\D))$ (for $\xi>0$),  $m(\inp{\t})\in L^2(0,\tend;H^1(\D))$ and the fact that 
\[
\norm{\kernel*u}^2_{L^2(0,\tend;\Va)}\leq\left(\beta\norm{\nabla\kernel}^2_{L^1(\R^n)}+\norm{\kernel}^2_{L^1(\R^n)}\right)\norm{\inp{\p}}^2_{L^\infty(\QC)}\leq C,
\]
it follows that $\inp{\lambda}\in L^2(0,\tend;H^1(\D))$.
For $\beta=0$ we have $\inp{\lambda}\in L^2(\Q)$ and,  in general,  no higher regularity can be expected,  since in~\eqref{eq:proof1_lagrange_mult} $\inp{w}=-\partial_t\inpt{\p}\in L^2(\Q)$.

Summarizing the above we derive the following energy estimate for $\xi\geq 0$, $\beta\geq 0$:
\begin{multline}
\norm{\partial_t\inpt{\p}}^2_{L^2(0,\tend;\Va^\prime)}+\|\partial_t\inpt{\t}\|_{L^2(0,\tend; \Vt^\prime)}^2+\norm{\inp{\p}}^2_{L^\infty(\QC)}+\norm{\inp{\t}}^2_{L^2(0,\tend; \Vt)}\\
+ \norm{\inp{w}}^2_{L^2(0,\tend;\Va)}+\norm{\inp{\lambda}}^2_{L^2(\Q)}\leq C,\label{eq:aux_energy_est}
\end{multline}
where the underlying constant depends on the initial data $\p^0$ and $\t^0$.

\begin{itemize}
\item \emph{Passing to the limit $\tt\to 0$.}
\end{itemize}

From the energy estimate~\eqref{eq:aux_energy_est} using Banach-Alaoglu theorem we can extract weakly convergent subsequences.  That is,  there exist functions $\p$,  $\t$,  $\lambda$ and $w$ ($\beta>0$) such that as $\tt\to 0$ (equivalently $K\to\infty$) the following holds
\begin{align*}
\partial_t\inpt{\p}\rightharpoonup\partial_t\p\quad &\text{weakly in}\; L^2(0,\tend;\Va^\prime),\\
\partial_t\inpt{\t}\rightharpoonup\partial_t\t\quad &\text{weakly in}\; L^2(0,\tend;\Vt^\prime),\\
\inp{\t}\rightharpoonup \t\quad &\text{weakly in}\; L^2(0,\tend; \Vt),\\
\inp{\p}\rightharpoonup\p\quad &\text{weakly-* in}\; L^{\infty}(0,\tend;H^1(\D)),&&\text{for}\; \xi>0,\\
\inp{\p}\rightharpoonup\p\quad &\text{weakly-* in}\; {L^\infty(Q),}&&\text{for}\; \xi\geq 0,\\
\inp{\lambda}\rightharpoonup\lambda\quad &\text{weakly in}\; {L^2(0,\tend,H^1(\D)),} &&\text{for}\; \beta>0,\\
\inp{\lambda}\rightharpoonup\lambda\quad &\text{weakly in}\; {L^2(Q),}&&\text{for}\;\beta = 0,\\
\inp{w}\rightharpoonup w\quad &\text{weakly in}\; L^2(0,\tend;\Va),
\end{align*}
where by a slight abuse of notation we drop the subsequent index.   
To show that the limiting functions fulfill the variational formulation~\eqref{CH_VI},   we pass to the limit when $\tt\to 0$ (or $K\to\infty$) in the time-integrated variant of the variational system~\eqref{CH_VI_tk_int}:
\begin{multline}
\int_0^\tend\inner{\partial_t\left(\inpt{\t}-L\inpt{\p}\right),\phi}\d t+\int_{0}^{\tend}D\inner{\nabla\inp{\t},\nabla\phi}\d t =0,\quad\forall\phi\in L^2(0,\tend;\Vt)\\
\int_0^{\tend}\inner{{\mu}\partial_t \inpt{\p}+\inp{w},\psi}\d t+\int_0^{\tend}\beta\inner{\nabla\inp{w},\nabla\psi}\d t=0,\quad\forall\psi\in L^2(0,\tend;\Va),\\
\inner{\xi\inp{\p}-\kernel*\inp{\p}+{\cpot}/{2}-{\cm m(\inpl{\t})} - \inp{w}+\inp{\lambda},\zeta}_{L^2(\Q)}=0,\quad\forall\zeta\in L^2(0,\tend;\Vb),\\
\inner{\eta-\inpp{\lambda},1-\inp{\p}}_{L^2(\Q)}\geq 0,\quad \inner{\eta-\inpm{\lambda},\inp{\p}}_{L^2(\Q)}\geq 0,\quad\forall\eta\in L^2(0,\tend;\M).
\label{CH_VI_tk_integrated_time}
\end{multline}
Passing to the limit in the first two equations,  due to linearity,  leads us to the fact that the limiting functions $\t$ and $\p$ satisfy the first two equations in~\eqref{CH_VI_time}.
Due to nonlinearity induced by $m(\t)$ and inequality constraints,  respectively,  we need to achieve strong convergences in the corresponding sequences before passing to the limit in the remaining last two equations in~\eqref{CH_VI_tk_integrated_time}.

Since $\inpt{\t}$ is uniformly bounded in $L^2(0,T;H^1(\D))\cap W^{1,2}(0,T;(H^1(\D))^\prime)$, where 
for $p\in[1,\infty)$, $W^{1,p}(0,\tend;X)=\{u\in L^p(0,\tend;X)\colon \partial_t u\in L^p(0,\tend;X)\}$, which follows from~\eqref{eq:proof_t_bound},~\eqref{prop_interp1}, and~\eqref{eq:proof_dtt_bound},
by the Aubin-Lion's lemma we obtain that $\inpt{\t}\to\t$ strongly in $L^2(\Q)$.
Then,  using~\eqref{prop_interp3} it follows
\begin{align*}
\norm{\inpl{\t}-\t}^2_{L^2(\Q)}\leq\norm{\inpl{\t}-\inpt{\t}}^2_{L^2(\Q)}
+\norm{\inpt{\t}-\t}^2_{L^2(\Q)}\\
\leq \tt\norm{\partial_t\inpt{\t}}^2_{L^2(0,\tend;\Vt^\prime)}+\tt\norm{\inp{\t}}^2_{L^2(0,\tend;\Vt)}+\tt^2\norm{\t^0}^2_{\Vt}+\norm{\inpt{\t}-\t}^2_{L^2(\Q)},
\end{align*}
and passing to the limit for $\tt\to 0$ it follows that $\inpl{\t}\to\t$ strongly in $L^2(\Q)$.  Now,  using a uniform Lipschitz continuity of $m(\inpl{\t})$~\eqref{mt_condition} we deduce that $m(\inpl{\t})\to m(\t)$.  Then, passing to the limit $\tt\to 0$ in the third equation in~\eqref{CH_VI_tk_integrated_time} we obtain that the limiting functions $\p$, $\t$, $w$, and $\lambda$ satisfy the corresponding equation~in~\eqref{CH_VI_time}.

Analogously,  to pass to the limit in the inequality constraints we upgrade a weak convergence to strong in $\inp{\p}$.  We consider the case $\beta=0$ and $\xi\geq 0$,  since the remaining case $\beta>0$,  follows similar steps as in~\cite{BG2021CH}.  
First,  we note that $\inpt{\p}$ is uniformly bounded in $L^2(0,\tend;H^1(\D))$,  which directly follows the bound on $\inp{\p}$ in $L^2(0,\tend;H^1(\D))$ and~\eqref{prop_interp1}.   
Then, invoking the Aubin-Lion's lemma we obtain that  $\inpt{\p}\to \p$ strongly in $L^2(\Q)$, and using~\eqref{prop_interp2} we have that
\begin{multline*}
\norm{\inp{\p}-\p}_{L^2(\Q)}\leq \norm{\inp{\p}-\inpt{\p}}_{L^2(\Q)}+\norm{\inpt{\p}-\p}_{L^2(\Q)} 
\leq\frac{\tt}{3}\norm{\partial_t\inpt{\p}}_{L^2(\Q)}+\norm{\inpt{\p}-\p}_{L^2(\Q)}.
\end{multline*}
Then, passing to the limit when $\tt\to 0$ we obtain that $\inp{\p}\to\p$ strongly in $L^2(\Q)$.  Invoking the properties~\eqref{CH_VI_compl_split} and the fact that the inner product of a strong and weak convergences converges we pass to the limit in~\eqref{CH_VI_tk_integrated_time}
\begin{align*}
0\leq (\eta,1-\inp{\p})_{L^2(\Q)}\to(\eta,1-\p)_{L^2(\Q)}\;\;\text{and}\;\;
\inner{\inpp{\lambda},1-\inp{\p}}_{L^2(\Q)}\to \inner{\lambda_+,1-\p}_{L^2(\Q)}.
\end{align*}
Proceeding similarly for $\inpm{\lambda}$ we obtain that the limiting pair $(\p,\lambda_{\pm})$ satisfy the complementarity conditions  in~\eqref{CH_VI_time} and this concludes the proof. 
\end{proof}

\subsection{Sharp interfaces}
In the next theorem we extend the projection formula~\eqref{proj_tk_CH} to the time continuous case and provide conditions when sharp interfaces occur.
\begin{theorem}[Sharp interfaces]
Let $(\t,\p,w,\lambda)$ be a solution of~\eqref{CH_VI_time} and $\kernel\in W^{1,1}(\R^n)$ satisfies~\eqref{kernel_cond1}, then the following holds a.e.  in $\D\times (0,T)$:
\[
\p(t)=P_{[0,1]}\left(\xi^{-1}{g(t)}\right),\quad\text{where}\; g:=w+\kernel*\p+\cm m(\t)-{\cpot}/{2}.
\]
Furthermore,  if $\xi>0$,  $\beta> 0$ and  $|\{g(t)\}|=0$ a.e.  $t\in(0,T)$,  then $\p(t)$ is discontinuous and attains only pure phases $\p(t)\in\{0,1\}$ a.e.  in $\D\times(0,T)$.
\end{theorem}

\section{Discretization and numerical examples}\label{sec:discretization}
We discuss discretization of~\eqref{CH_VI_tk_SP} using piece-wise linear finite elements and present several numerical examples. 

\subsection{Discretization}
Let $\{\mathcal{T}_h\}_h$ be a shape-regular triangulation of $\DD$ with $h$ denoting a maximum diameter of the elements $\mathcal{T}\in\{\mathcal{T}_h\}_h$,  We denote $\mathcal{J}_h^k$, $k\in\{i,c,ic\}$ a set of nodes corresponding to the triangulation of $\overline{\D}$ ($k=i$), $\DD\setminus\overline{\D}$ ($k=c$) and $\overline{\DD}$ ($k=ic$).
Associated with $\{\mathcal{T}_h\}_h$ we define piece-wise linear finite element spaces: ${S}^h_\omega=\{v_h\in C^{0}(\overline{\omega})\colon v_h|_{\mathcal{T}}\in P_1(\mathcal{T}),\;\forall \mathcal{T}\in\mathcal{T}_h\}$,  with $\omega\in\{\D,\DD\}$, see~\cite{BG2021CH} for more details.
As detailed there we employ a trapezoidal quadrature rule that provides mass-lumping property for the local and nonlocal terms:
\begin{align*}
\inner{B\phi,\psi_j}=b_h(\phi,\psi_j):=(\cker^h\phi,\psi_j)_h-(\kernel\circledast\phi,\psi_j)_h+(\mathcal{N}_h\phi,\psi_j)_{\bar{h}},\quad \forall\phi,\psi_j\in\Sh,
\end{align*}
where the last term vanishes for $j\in\mathcal{J}_h^{i}$ and $b_h(\phi,\psi_j)=(\mathcal{N}_h\phi,\psi_j)_{\bar{h}}$ for $j\in\mathcal{J}_h^{c}$,  where $\mathcal{N}_h\phi=\cker^h\phi-\kernel\circledast\phi$.
Here,  $(\cdot,\cdot)_{h,\bar{h}}$ denotes a mass-lumped $L^2(\D)$ and $L^2(\DI)$ inner product, respectively,  and 
\[
(\kernel\circledast\phi)(\x)=\int_{\DD}I_h^{\xx}\left[\kernel(\x,\xx)\phi(\xx))\right]\d\xx,\quad \cker^{h}(\x)=\int_{\DD}I_h^{\xx}\left[\kernel(\x,\xx)\right]\d\xx,
\]
where $I_h^\xx\left[\cdot\right]$ denotes a nodal interpolant with respect to $\xx$.

Then,  given $\t^0_h\in\Sho$,  $\p^0_h\in\Sh$ we seek $\t_h^k,w^k,\lambda^k\in\Sho$ and $\p_h^k\in\Sh$, such that for $k=1,\dots,K$ it holds
\begin{align*}
&(\th^k,\phi)_{h}+\tt D(\nabla \th^k,\nabla\phi)-L(\ph^k,\phi)_h=(\th^{k-1}+L\ph^{k-1},\phi)_{h},&\forall\phi\in\Sho,\\
&\mu(\ph^k,\psi)_{h}+\tt (\wh^k,\psi)_h+\beta\tt(\nabla \wh^k,\nabla\psi)=\mu(\ph^{k-1},\psi)_{h},&\forall\psi\in\Sho,\\
&b_h(\p_h^k,\zeta)-(\cpot\ph^k+\wh^k+{\cm m(\th^{k-1})}-\lh^k-\cpot/2,\zeta)_{h}=0,  &\forall\zeta\in\Sh,\\
&\lhp^k(\x_j)(\ph^k(\x_j) - 1)=0,\quad \lhm^k(\x_j)\ph^k(\x_j)=0, &\forall \x_j\in\mathcal{J}_h^{i},\\
&\lh^k=\lhp^k-\lhm^k,\quad \lhpm^k\geq 0,\quad 0\leq \ph^k\leq 1.
\end{align*}
For the solution algorithm we adapt a primal-dual-active set strategy~\cite{kunisch2002} that has been already successfully applied in local and nonlocal settings~\cite{blank2011,BurkovskaGunzburger2019,BG2021CH}. 
Additionally,  employing an explicit discretization of the convolution term, 
$b_h(\ph^k,\zeta)\approx(\cker^h\ph^k,\zeta)_h-(\kernel\circledast\ph^{k-1},\zeta)_h$ can further speed-up computations,  and in the Allen-Cahn case ($\beta=0$)
can completely avoid a solution of a nonlinear phase-field system,  since $\p_h^k$ can be directly evaluated from a discrete version of the projection formula~\eqref{proj_tk_AC}.  That is,   for a given $\th^{k-1}$ and $\ph^{k-1}$, $K=1,\dots,K$,  $x_j\in \mathcal{J}_h^{i}$ it holds
 \begin{equation*}
\ph^k(x_j):=P_{[0,1]}\left(\frac{g_h^k(x_j)}{\mu/\tt+\cker^h(x_j)-\cpot}\right),\quad g_h^k:=\frac{\mu}{\tt}\ph^{k-1}+\kernel\circledast\ph^{k-1}+\cm m(\th^{k-1})-\frac{\cpot}{2}.
\end{equation*}

\subsection{Numerical examples}
We set $\D=(0,1)^n$, $n\in\{1,2\}$, which is discretized with a uniform mesh of a mesh size $h$.  
We consider a polynomial type nonlocal kernel:
\begin{equation*}
\kernel(\x-\xx)=\begin{cases}
\varepsilon^2 C(\delta)\max\left(0,1-\frac{|\x-\xx|^2}{\delta^2}\right),\quad\text{if}\;|\x-\xx|\leq \delta,\quad\delta>0,\\
0,\quad\text{otherwise},
\end{cases}
\end{equation*}
where $C(\delta)$ is chosen such that $\int_{\R^n}|\zeta|^2\kernel(|\zeta|)\d\zeta=2n\varepsilon^2$,  thus,  $C(\delta)=15/(2\delta^3)$ for $n=1$ and $C(\delta)=24/(\pi\delta^4)$ for $n=2$. {The kernel is additionally scaled by $\varepsilon>0$ as in~\eqref{eq:strong_solid_nonlocalCH} to keep resemblance to the corresponding local model for vanishing nonlocal interactions. }
Then,  for all $\x\in\D$ the constant $\cker$ in~\eqref{cker} can be computed exactly using $\cker=\frac{2\pi^{n/2}}{\Gamma(n/2)}\int_{0}^\delta|\xi|^{n-1}\hat{\kernel}(|\xi|)\d\xi$,  which leads to 
$\cker=10\varepsilon^2\delta^{-2}$ ($n=1$) and $\cker=12\varepsilon^2\delta^{-2}$ ($n=2$). 
The coupling term $m(\t)$ is chosen as in~\eqref{mt_kobayashi}
together with $\cpot=1/6$ in order to match the condition $F(0,\t)-F(1,\t)=m(\t)/6$; see~\cite{kobayashi1993}.  
\subsection*{Example~1}
We set the model parameters to 
$\mu=0.0012$, 
$\t_e=1$, 
$\alpha =0.9$, 
$\rho=20$, 
$L=0.5$, 
$D=1$, 
$\e=0.02$.  The final time is set to $\tend=0.05$,  $\tt=0.0003$ and $h=0.0024$.  We chose $\beta=0.02$ and set a nonlocal interaction radius to
$\delta=0.1540$, which corresponds to 
$\xi=\cker-\cpot=0.002$. 
The initial conditions are set to:
\begin{equation*}
\p^0(\x)=\begin{cases}
1,\quad\x\leq 0.2\\
0,\quad\text{otherwise},
\end{cases}\qquad
\t^0(\x)=0.
\end{equation*}
The corresponding snapshots of the solutions at different time instances are presented on Figure~\ref{fig:ex1}. We can observe that the nonlocal model delivers sharp interfaces (up to two grid points per interface) compared to the corresponding local model, where the interface is diffuse.  Furthermore, the sharpness of the interface in the nonlocal solution remains during a whole time evolution. 
We also investigate the effect of the parameter $\beta>0$ on the solution. Using the same settings as above on Figure~\ref{fig:ex2} we plot the snapshots of the nonlocal solutions for different values of $\beta$. As expected, we observe that the thickness of the interface decreases with increasing~$\beta$.  
\begin{remark}\label{rem:1}
We notice that the speed of the interface differs with respect to $\beta$. 
This suggest that in order to correctly match the speed of the interface with, e.g., nonlocal or local Allen-Cahn type models the remaining model parameters must be scaled appropriately with respect to $\beta$. The question of the interplay of the model parameters and nonlocal parameters (such as those $\delta$, $\beta$ or $\cker$) to match certain physical properties of the solution, such as e.g., a speed of the interface or mass, is more complex and will be detailed in our forthcoming work~\cite{BDGR23}. 
\end{remark}
\begin{figure}[ht!]
  \includegraphics[width=0.3\textwidth]{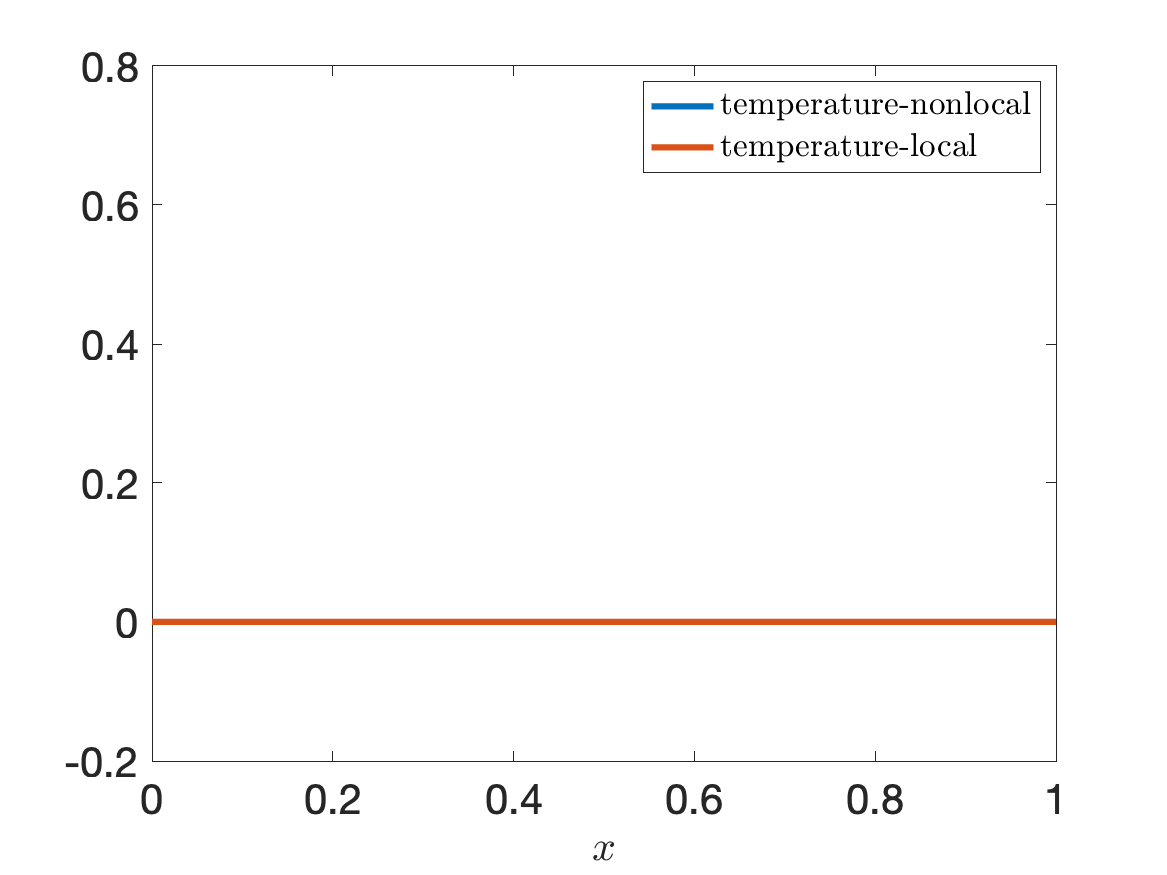}
   \includegraphics[width=0.3\textwidth]{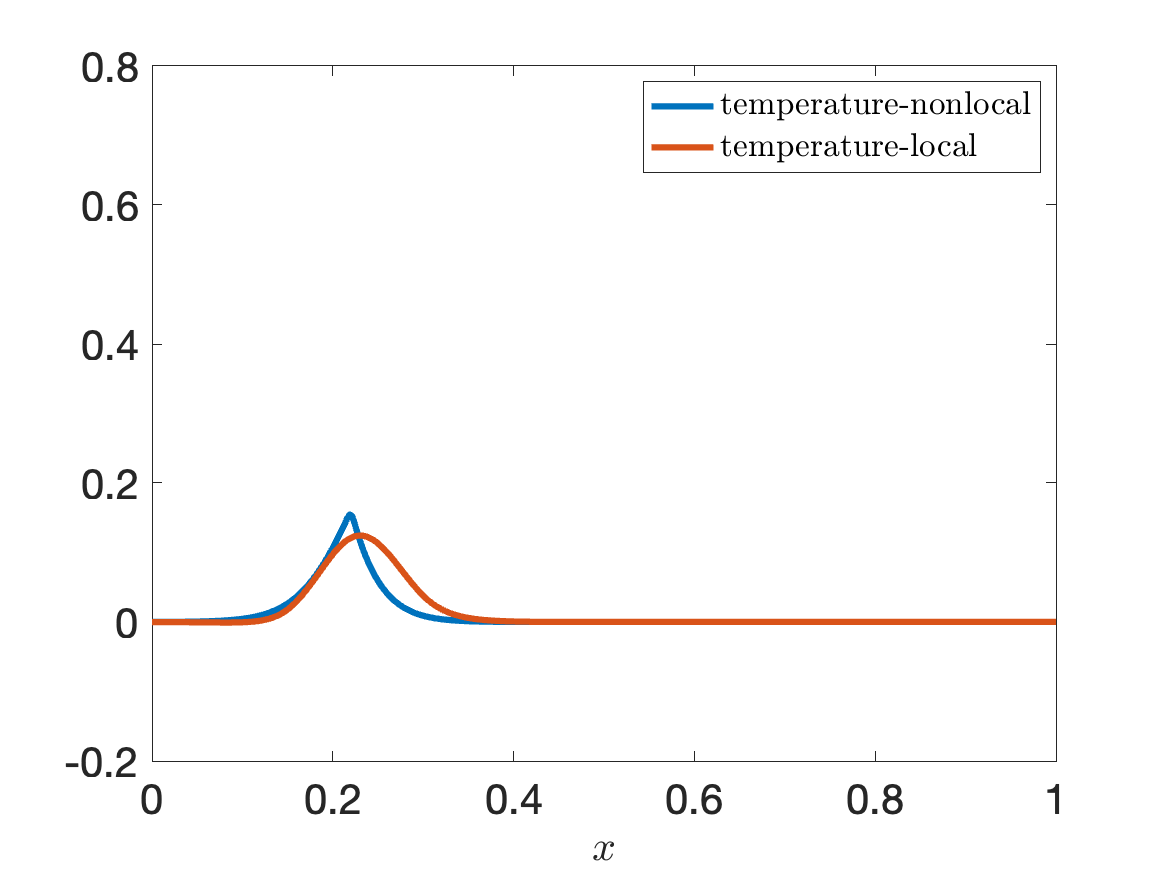}
    \includegraphics[width=0.3\textwidth]{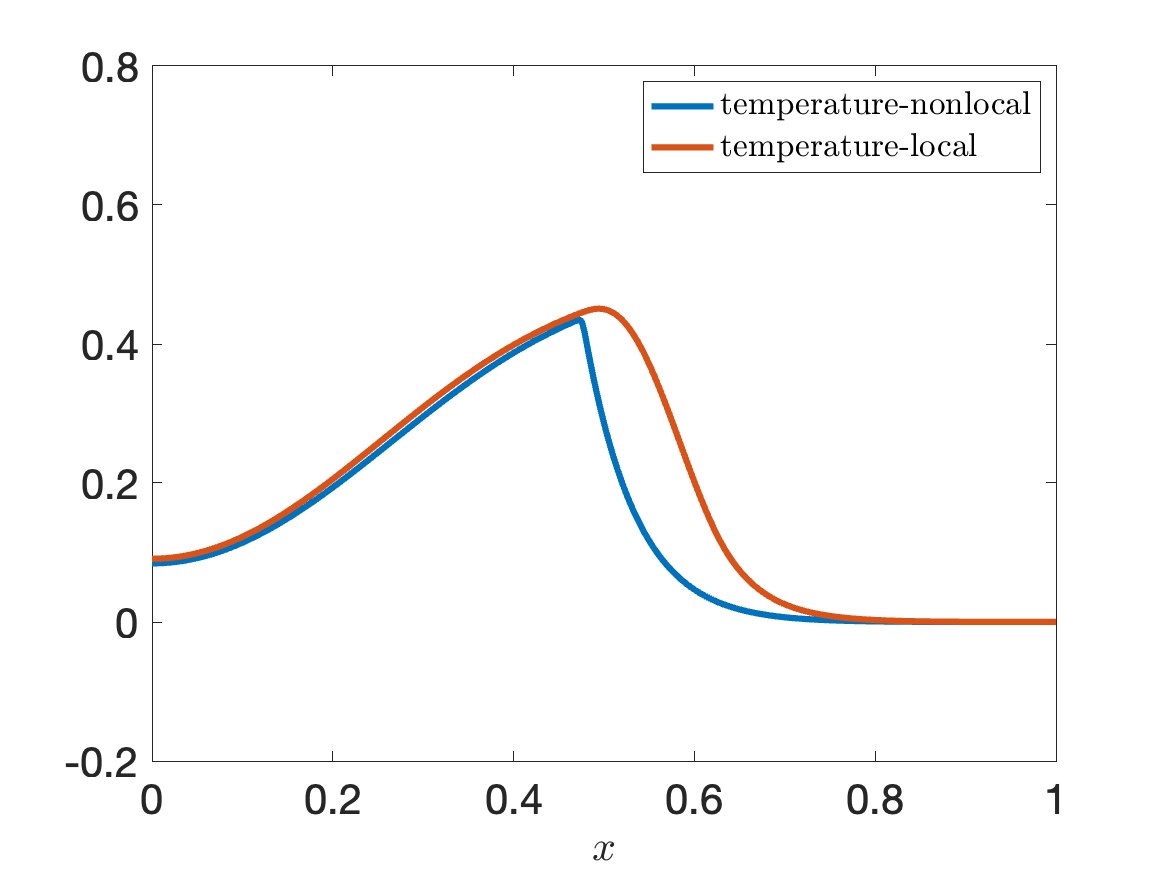}\\
        \includegraphics[width=0.3\textwidth]{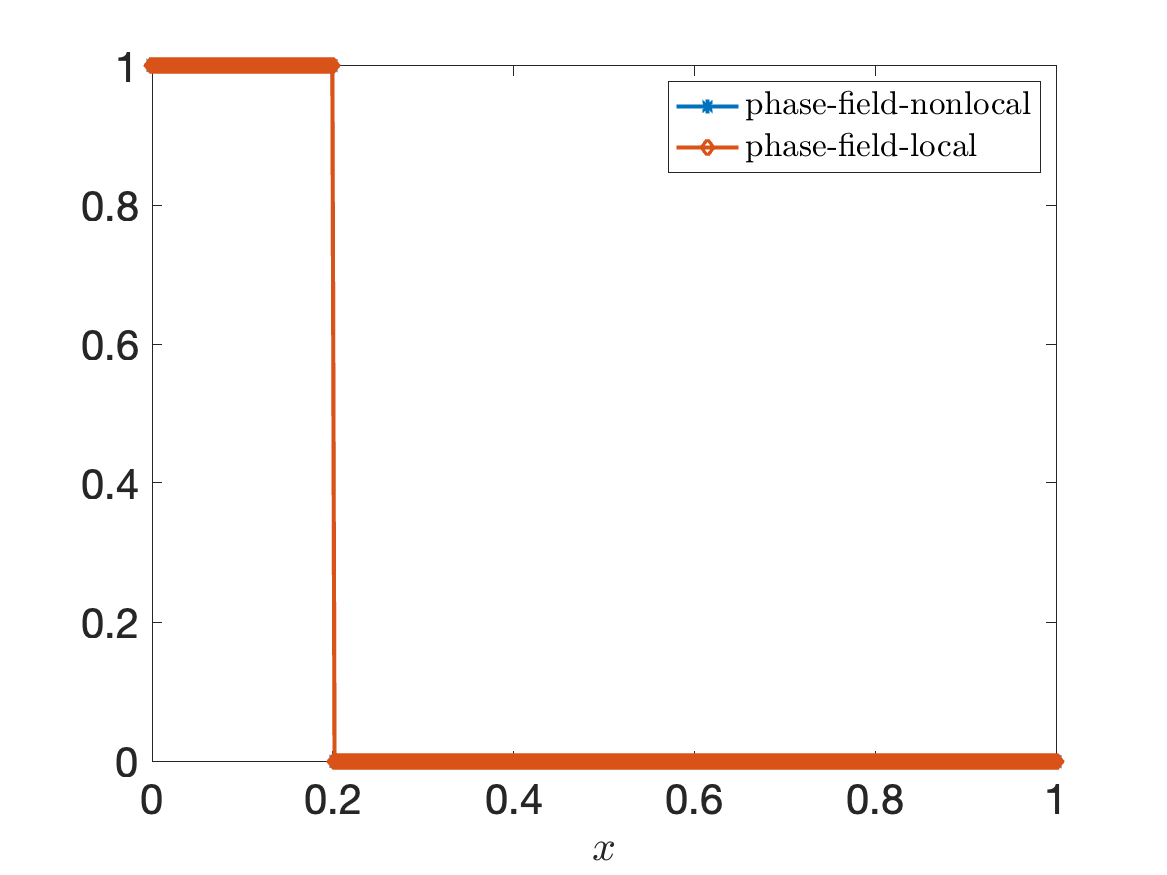}
   \includegraphics[width=0.3\textwidth]{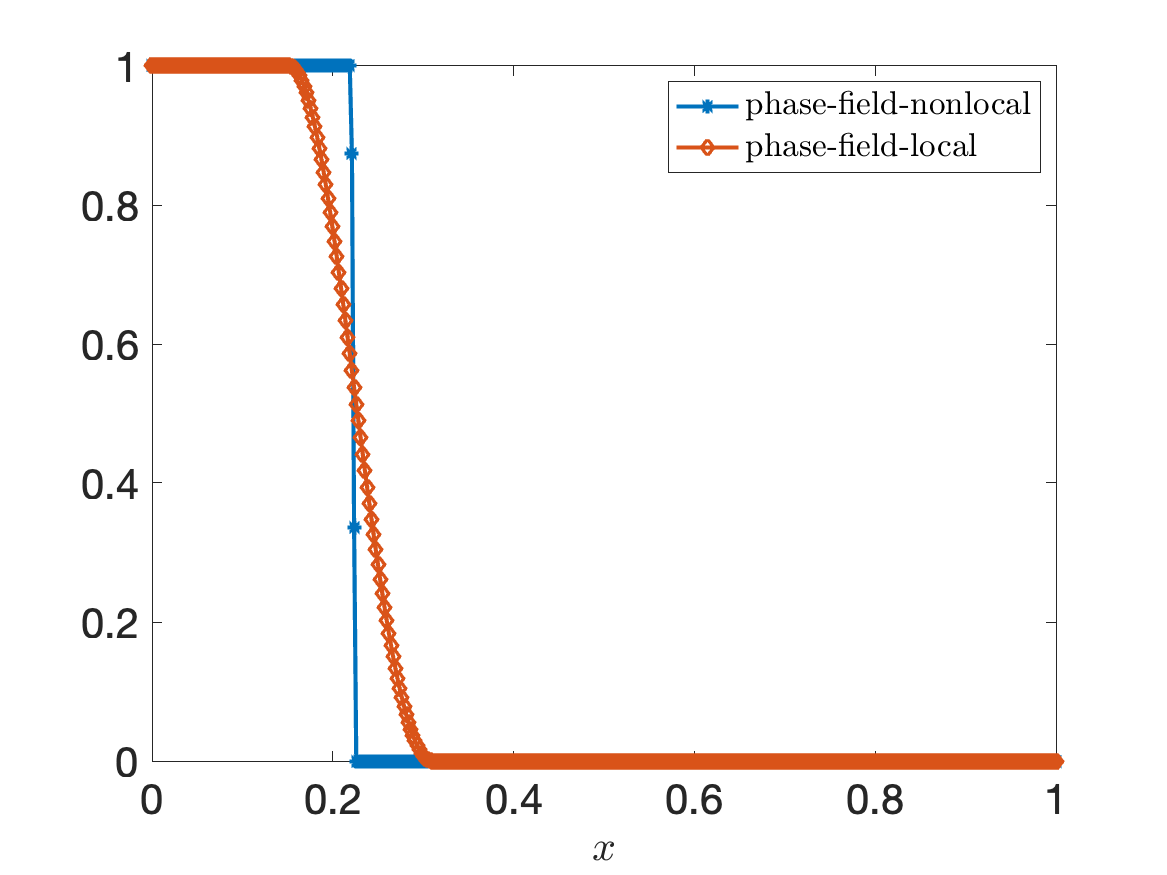}
    \includegraphics[width=0.3\textwidth]{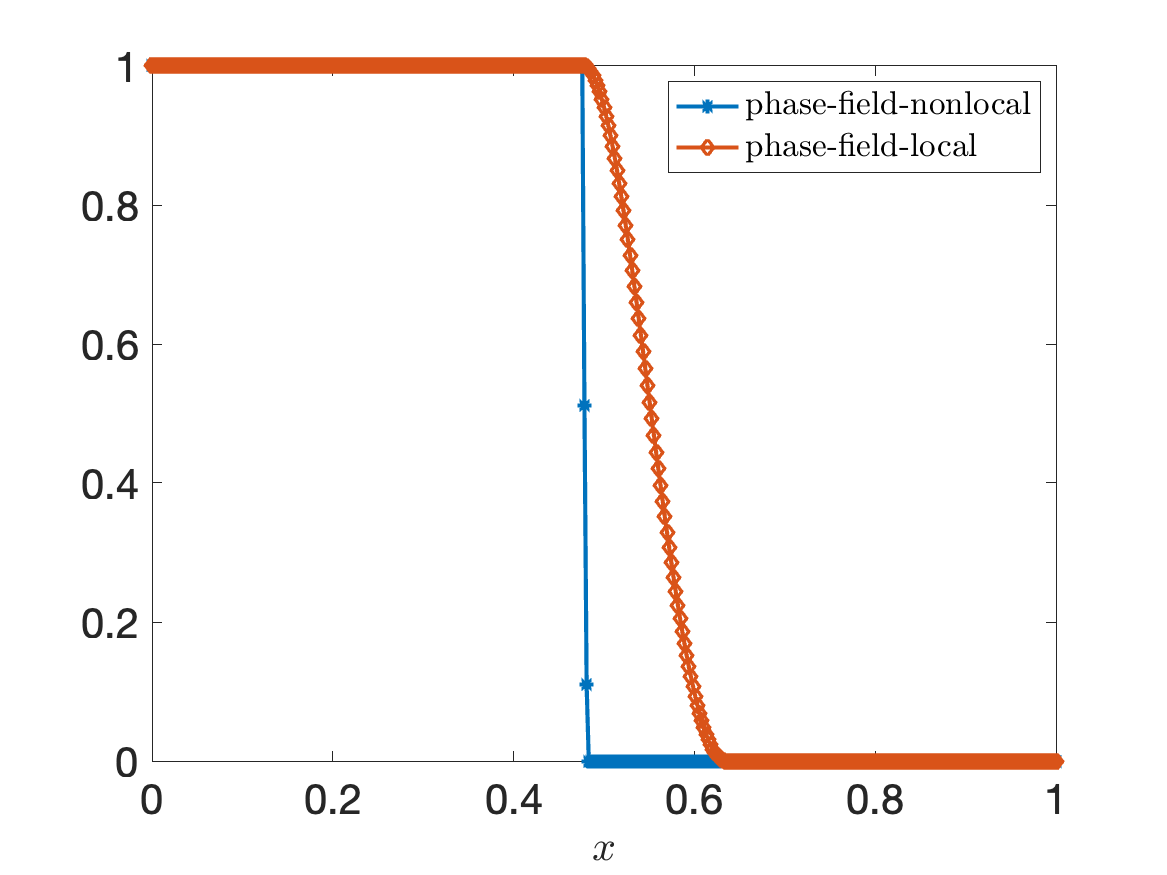}
 \caption{Snapshots of the local (red) and nonlocal (blue) solutions of the temperature (top) and phase-field variable (bottom) at $t_k=[0,0.0013,0.0163]$ (from left to right).}\label{fig:ex1}
 \end{figure} 
 \subsection*{Example~2}
We investigate numerically the effect of the nonlocal parameter $\delta$ on the solution.  We set $h=0.0012$,  $\beta=0.08$ and keep the remaining settings apart from $\delta$ as in Example~1.  On Figure~\ref{fig:ex2} (right) we plot the snapshots of the phase-field solution for different values of $\delta$.  As expected, we observe for decreasing $\delta$ or equivalently increasing $\xi$ the interface becomes more diffuse and the nonlocal solution converges to the corresponding local solution. 
 \begin{figure}[ht!]
  \includegraphics[width=0.32\textwidth]{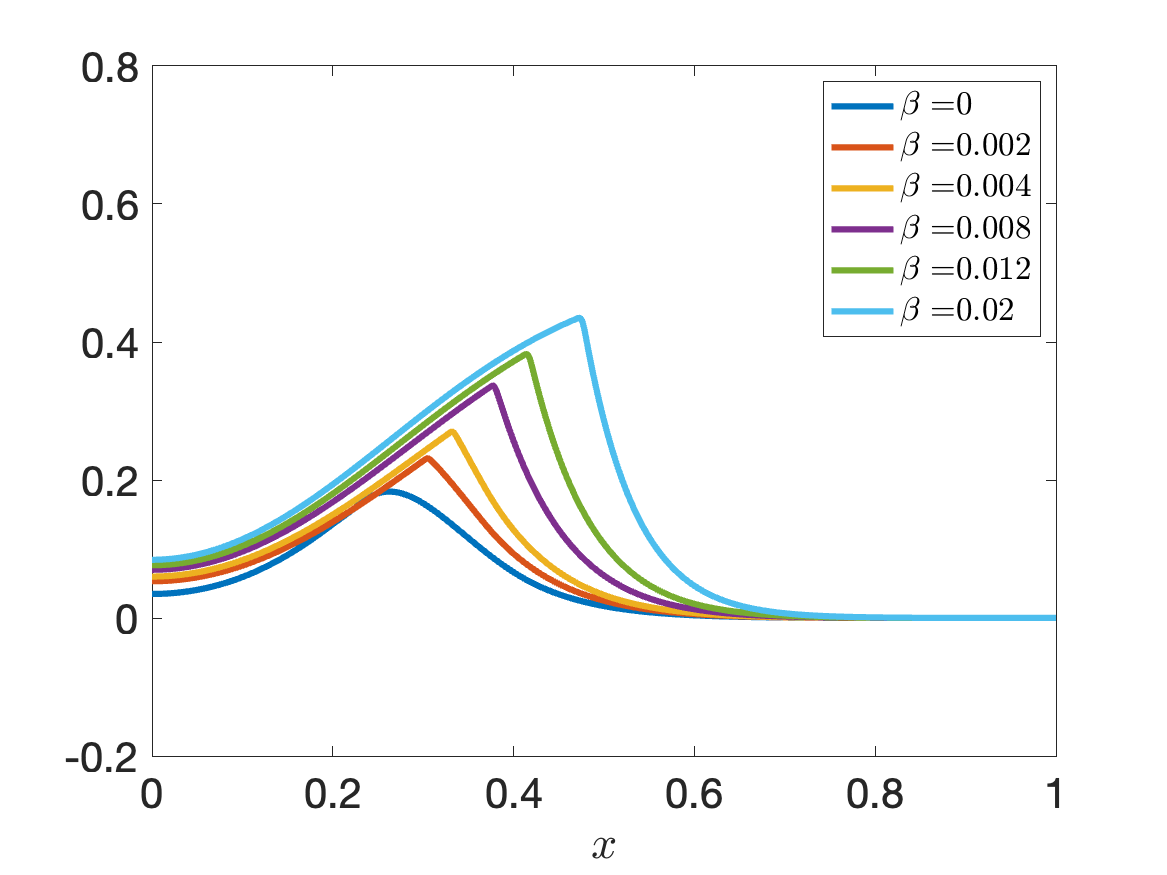}
   \includegraphics[width=0.32\textwidth]{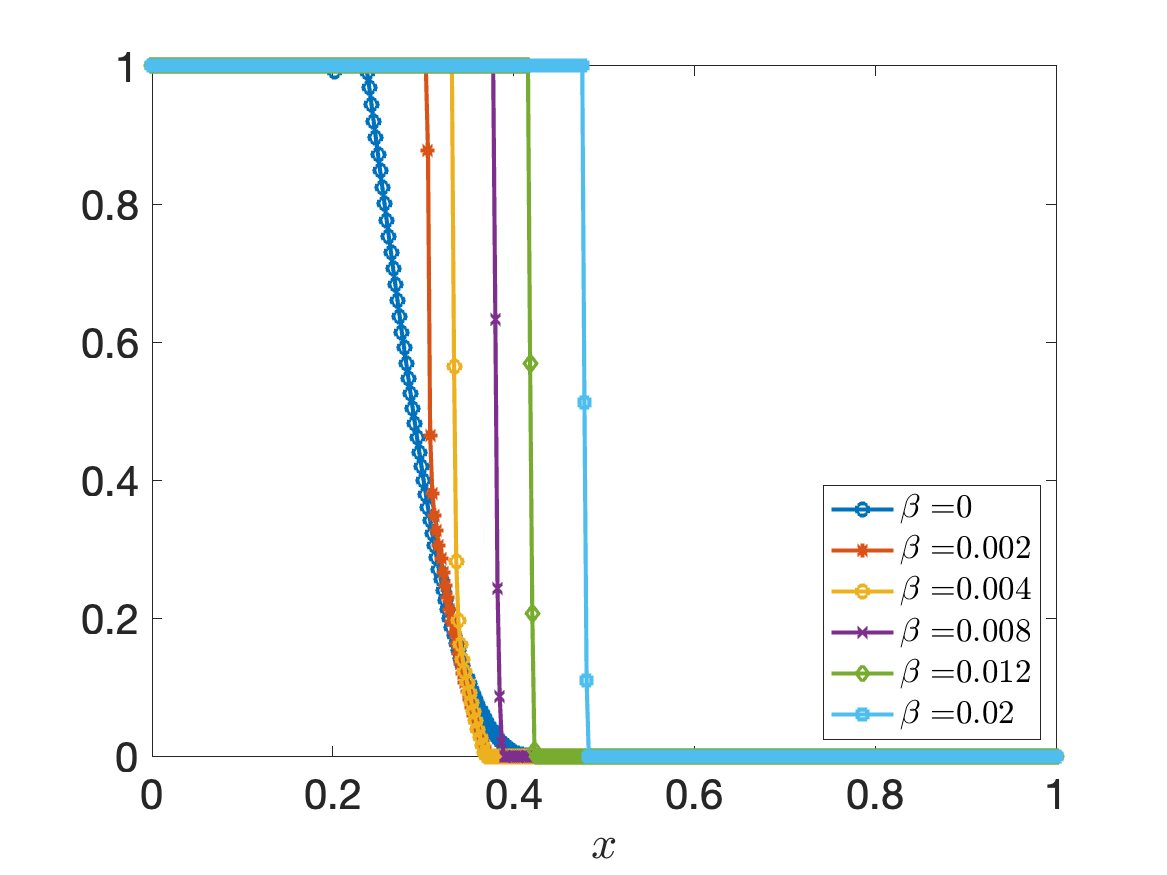}
    \includegraphics[width=0.32\textwidth]{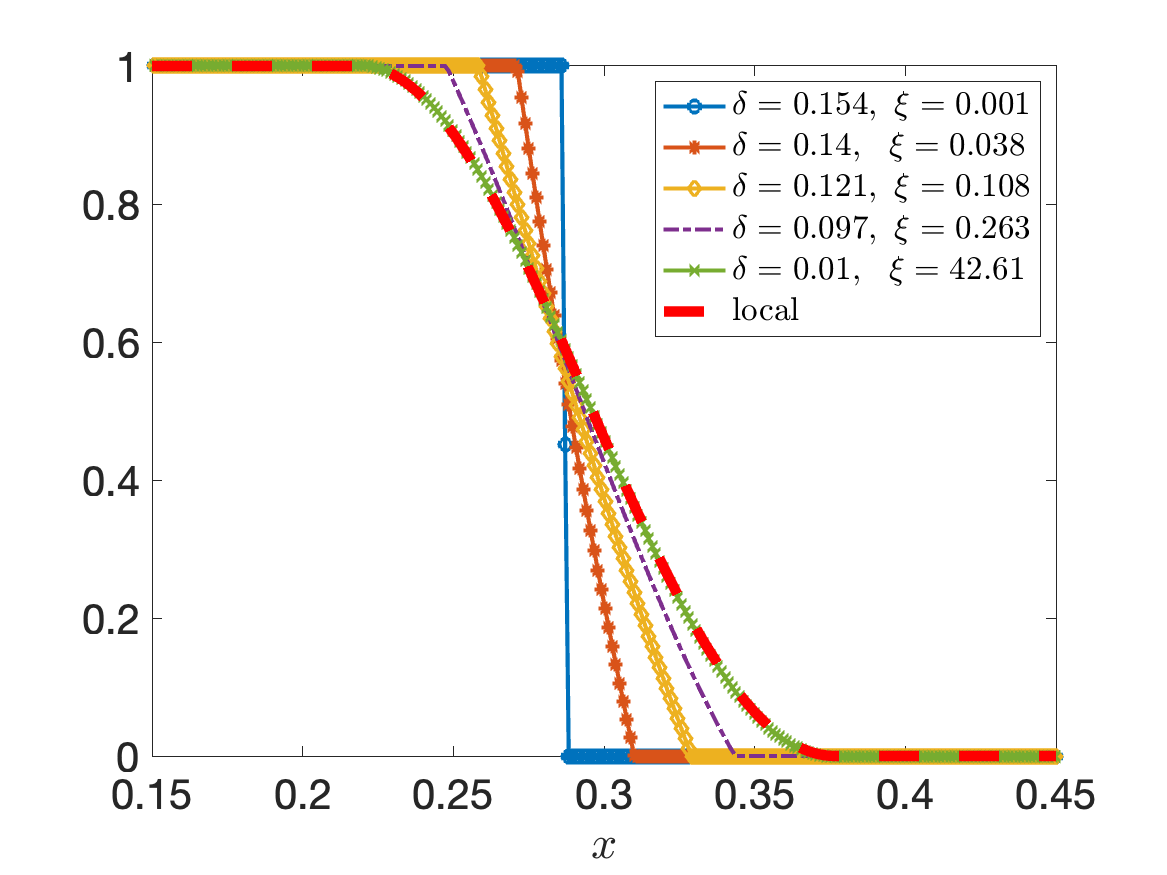}
 \caption{Snapshots of the temperature (left) and phase-field (middle) solutions of the nonlocal model at $t=0.017$ for different values of $\beta$ (Example~1). Right: Snapshots of the nonlocal and the local solutions at $t=0.0037$, $\beta=0.08$ (zoomed-in) (Example~2).}\label{fig:ex2}\end{figure} 
 \subsection*{Example~3}
Lastly,  we consider a two-dimensional example. This example is inspired by an example in~\cite{kobayashi1993},  which corresponds to solidification of pure materials, where we have a solid region around the boundaries of the domain and a pool of liquid in the interior. As the time evolves,  the solidification that occurs from the walls propagates inward until all liquid solidifies. We set the model parameters to
$\mu=0.0003$, 
$\t_e=1$, 
$\alpha =0.9$, 
$\rho=10$, 
$L=0.5$, 
$D=1$, 
$\e=0.01$. For discretization we use $\tend=0.03$,
$\tt=0.0001$,
$h\approx 0.0048$, and we also set
$\beta=0.002$ and $\delta=0.0826$,  which corresponds to
$\xi=0.0093$. 
The initial conditions are 
\begin{equation}
\p^0(\x,\xx)=\begin{cases}
1,\quad 0.1<\x<0.9,\quad 0.1<\xx<0.9\\
0,\quad\text{otherwise},
\end{cases}\qquad
\t^0(\x,\xx)=0.
\end{equation}
Phase-field and temperature solutions are presented on Figure~\ref{fig:ex5}--\ref{fig:ex5a}, where we also include solutions of the local model with the regular potential~\eqref{potential_regular}.  We observe that the nonlocal model with $\beta>0$ delivers the solution with the sharpest interface, whereas the most diffuse interface occurs in the local model with the regular potential. 
The corresponding width of the interface region is depicted on Figure~\ref{fig:ex5b}. We observe that the interface width in the nonlocal Cahn-Hilliard case spans approximately $1-2$ grid cells, while in the nonlocal and local Allen-Cahn solutions this varies between $16-18$ and $18-20$ grid cells, respectively.
 We also notice that the interface thickness affects the speed of the interface which is different for all solutions, and is explained by the fact that the same set of parameters is used for all models (cf.  Remark~\ref{rem:1}).
 \begin{figure}[ht!]\centering
   \includegraphics[width=0.2\textwidth]{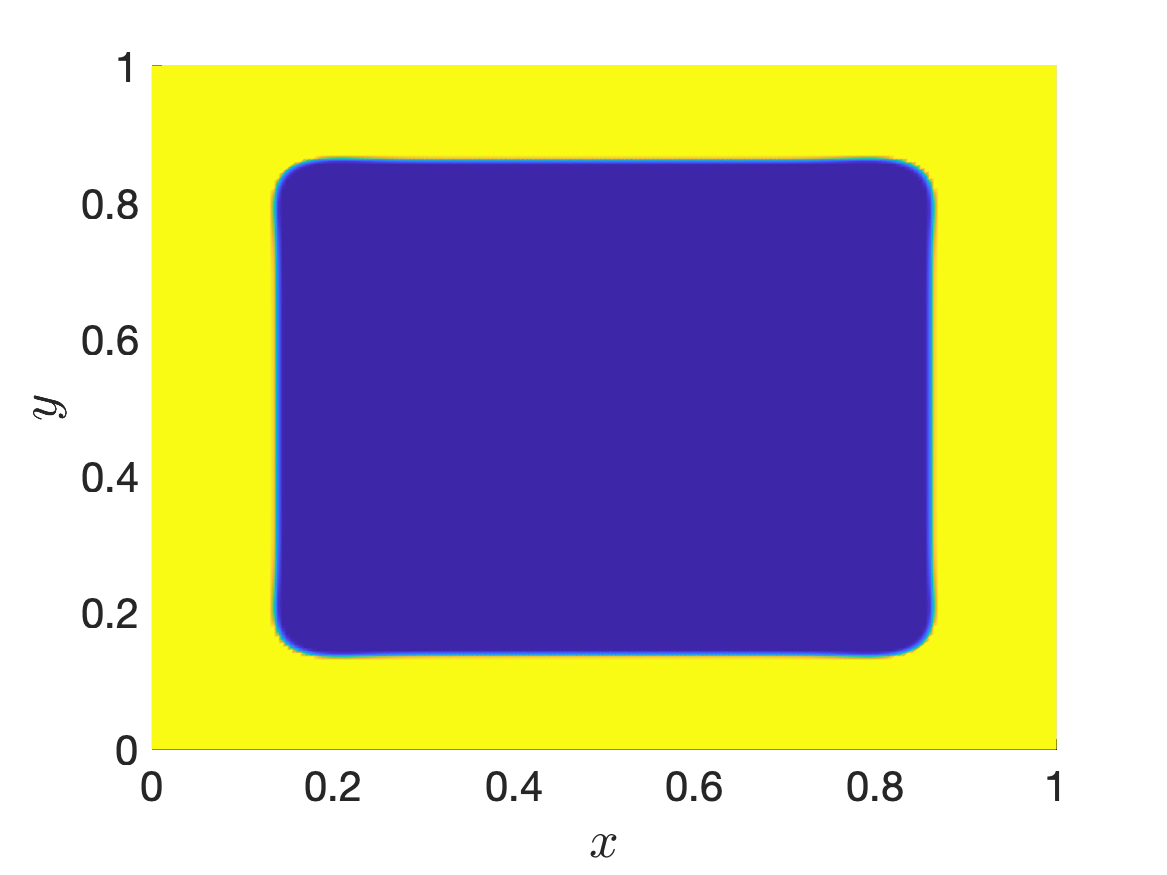}
    \includegraphics[width=0.2\textwidth]{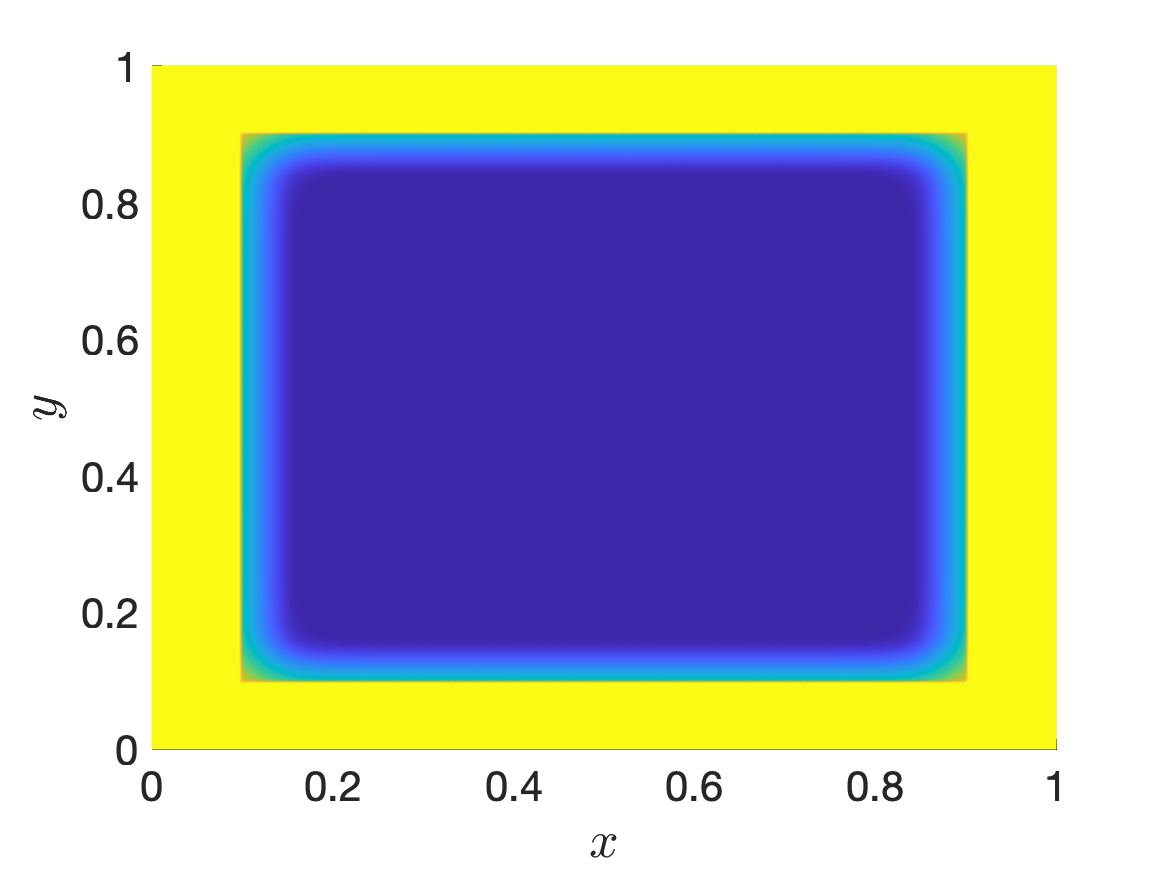}
   \includegraphics[width=0.2\textwidth]{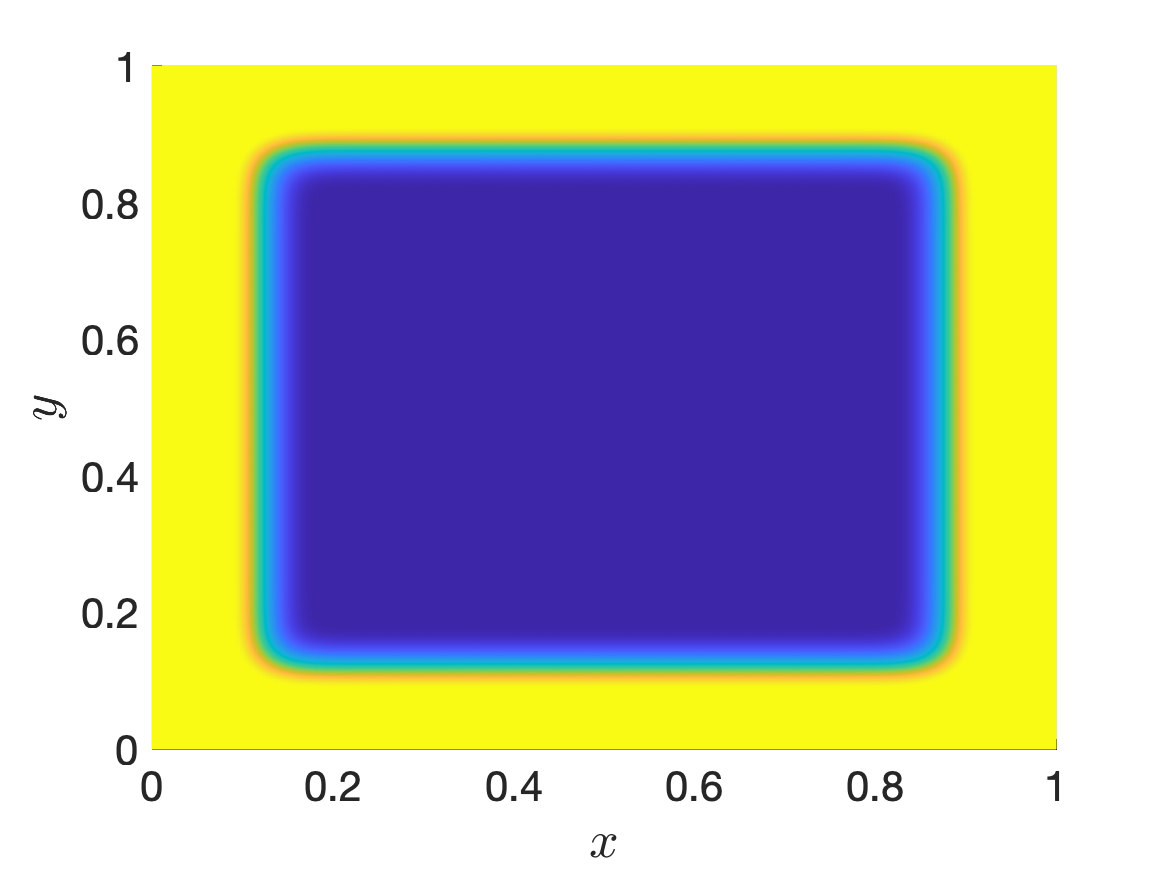}
   \includegraphics[width=0.2\textwidth]{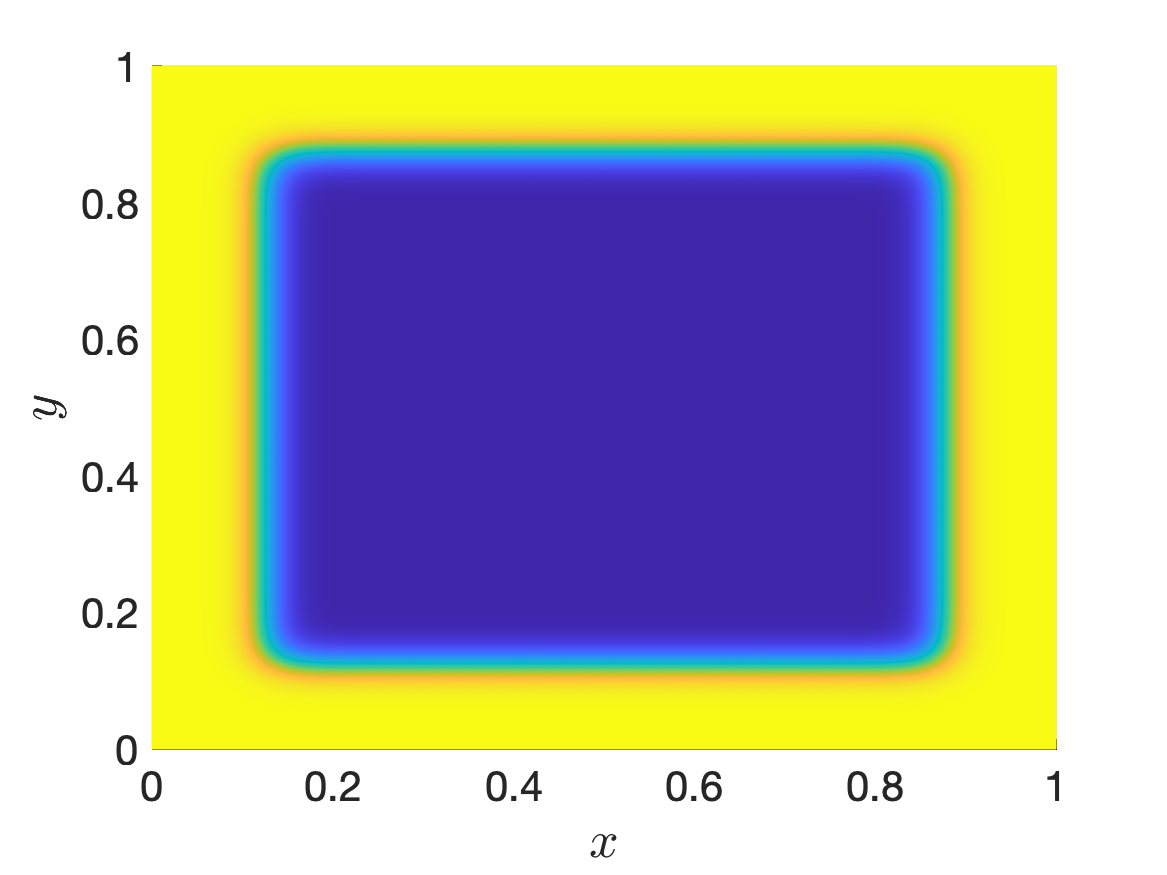}\\
       \includegraphics[width=0.2\textwidth]{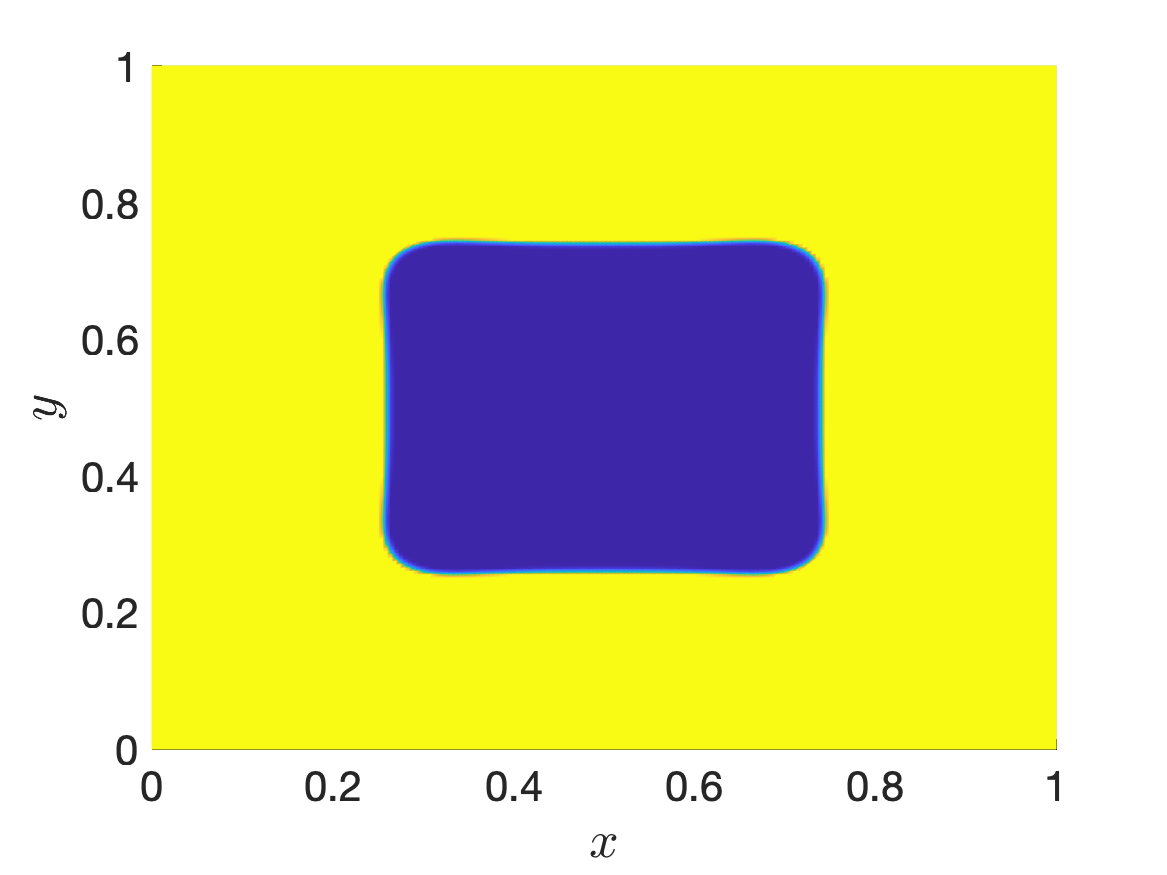}
   \includegraphics[width=0.2\textwidth]{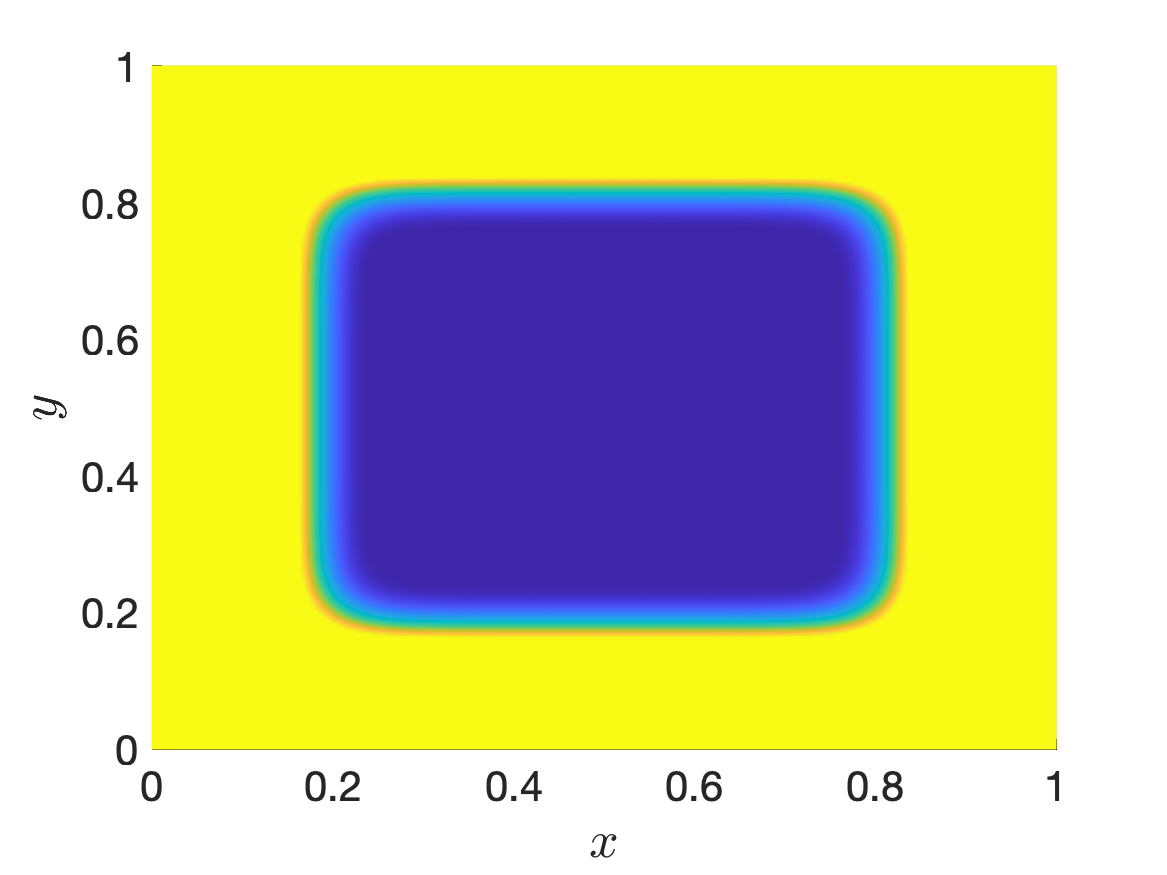}
    \includegraphics[width=0.2\textwidth]{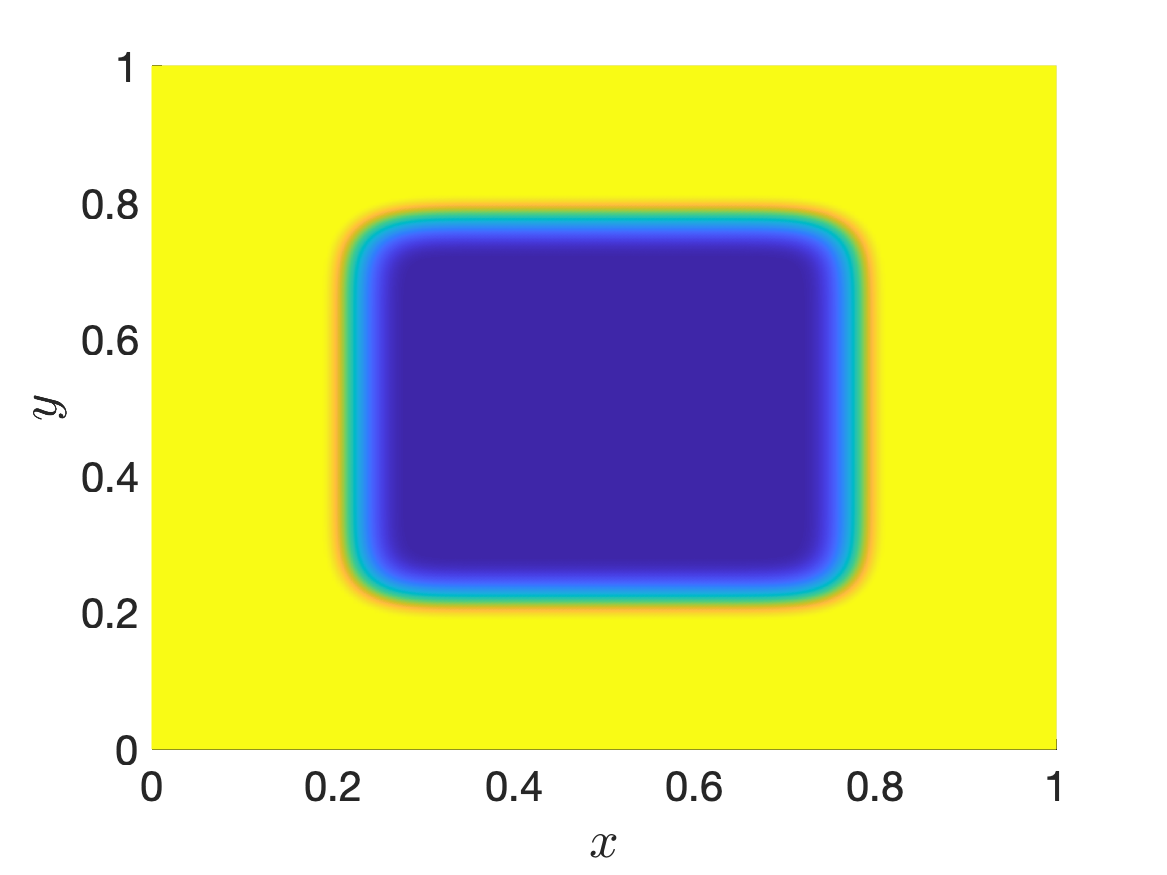}
     \includegraphics[width=0.2\textwidth]{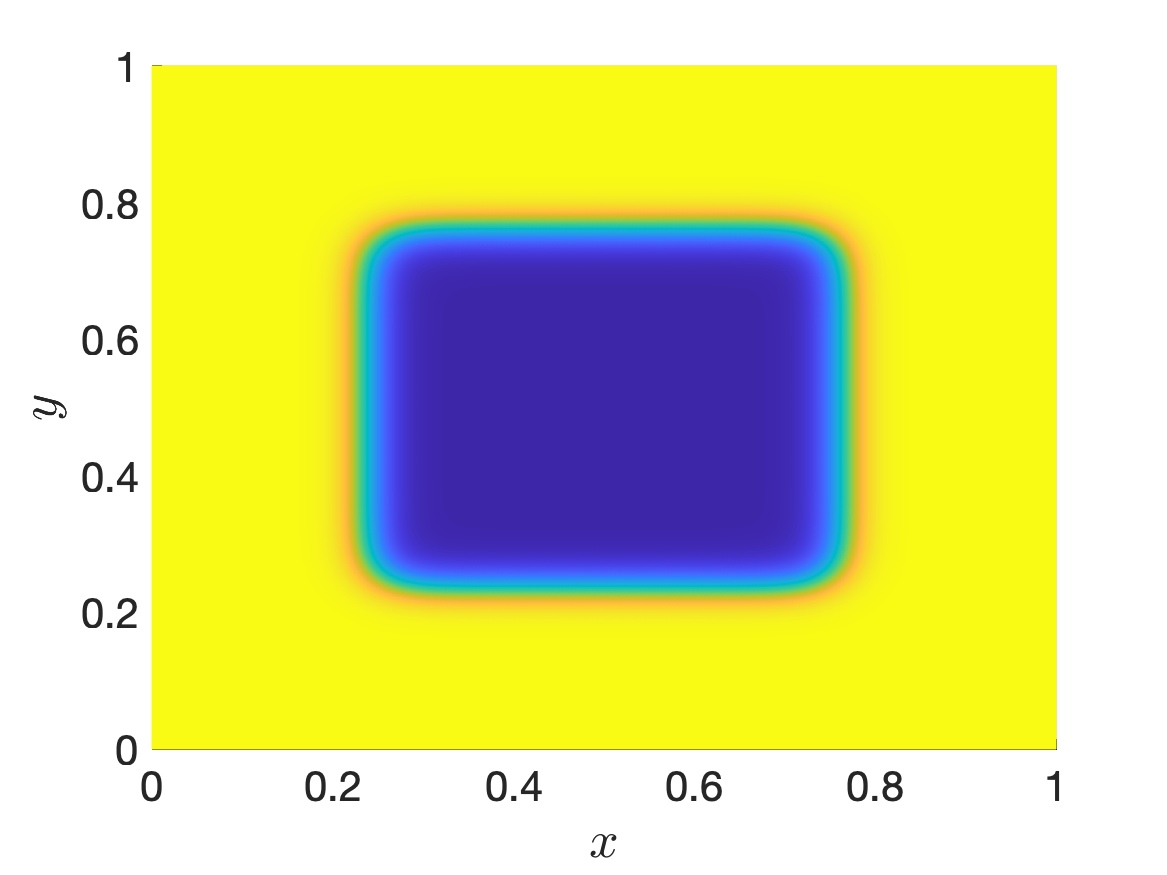}\\
   \includegraphics[width=0.2\textwidth]{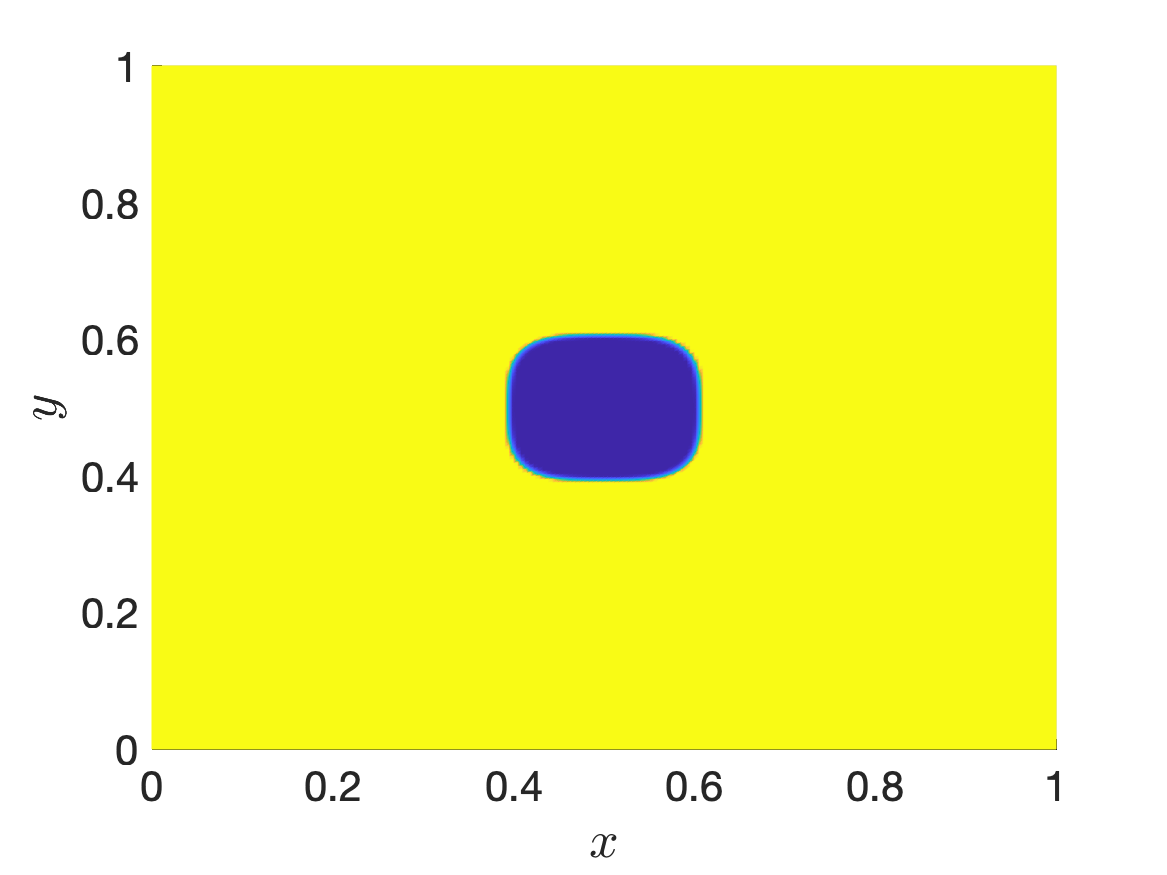}
    \includegraphics[width=0.2\textwidth]{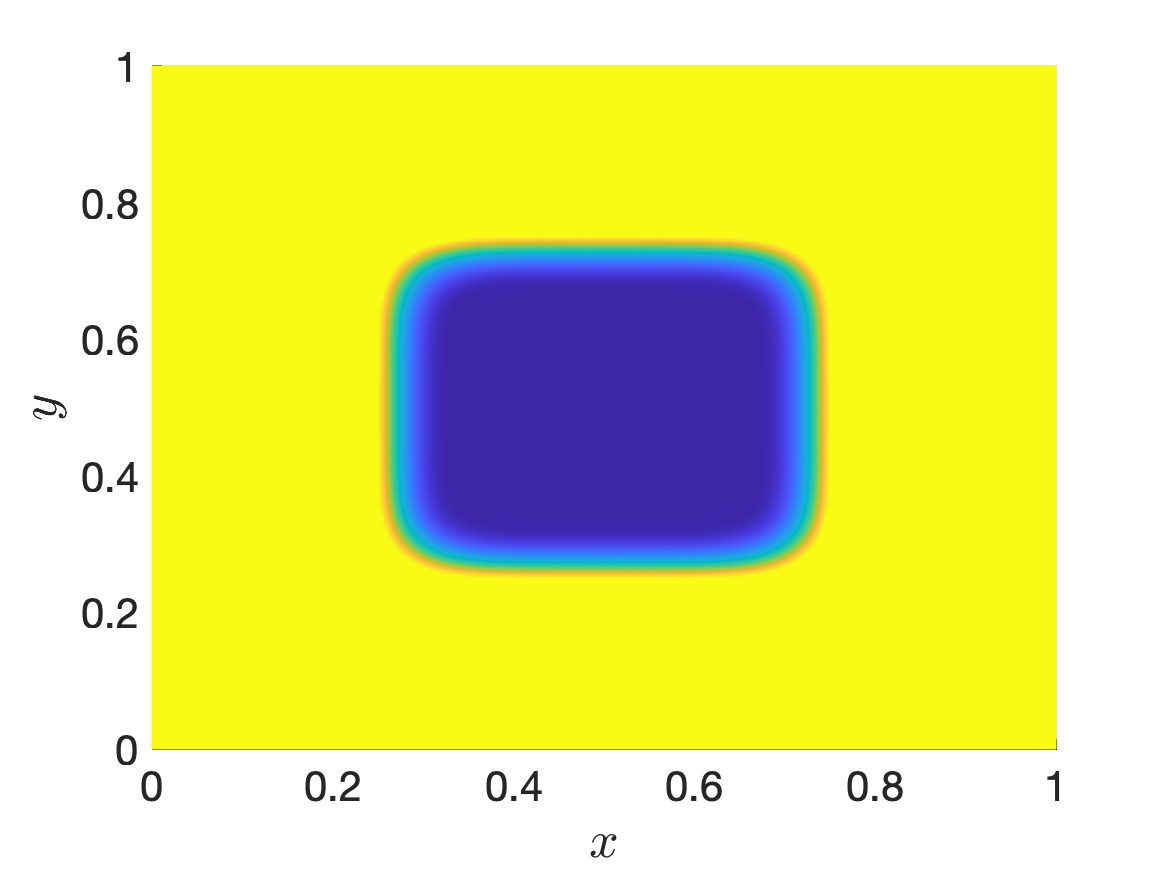}
    \includegraphics[width=0.2\textwidth]{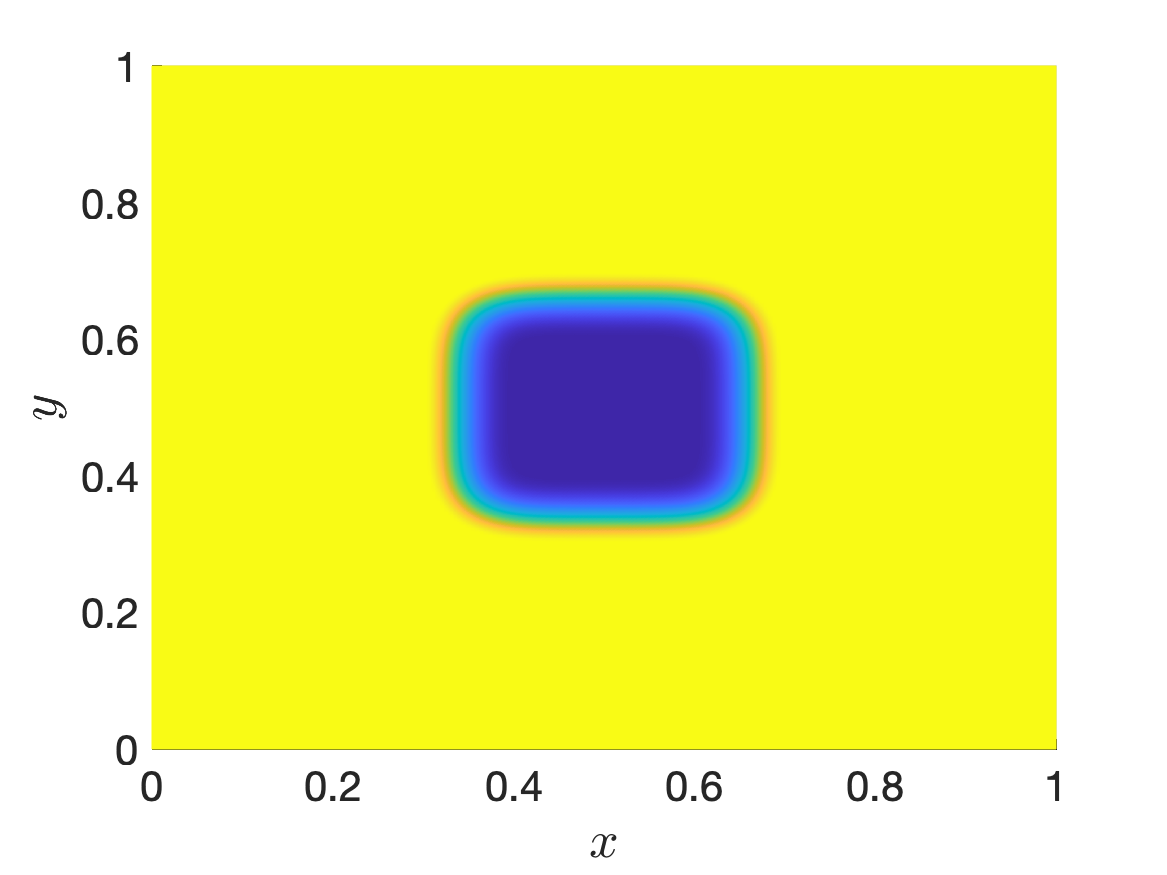}
    \includegraphics[width=0.2\textwidth] {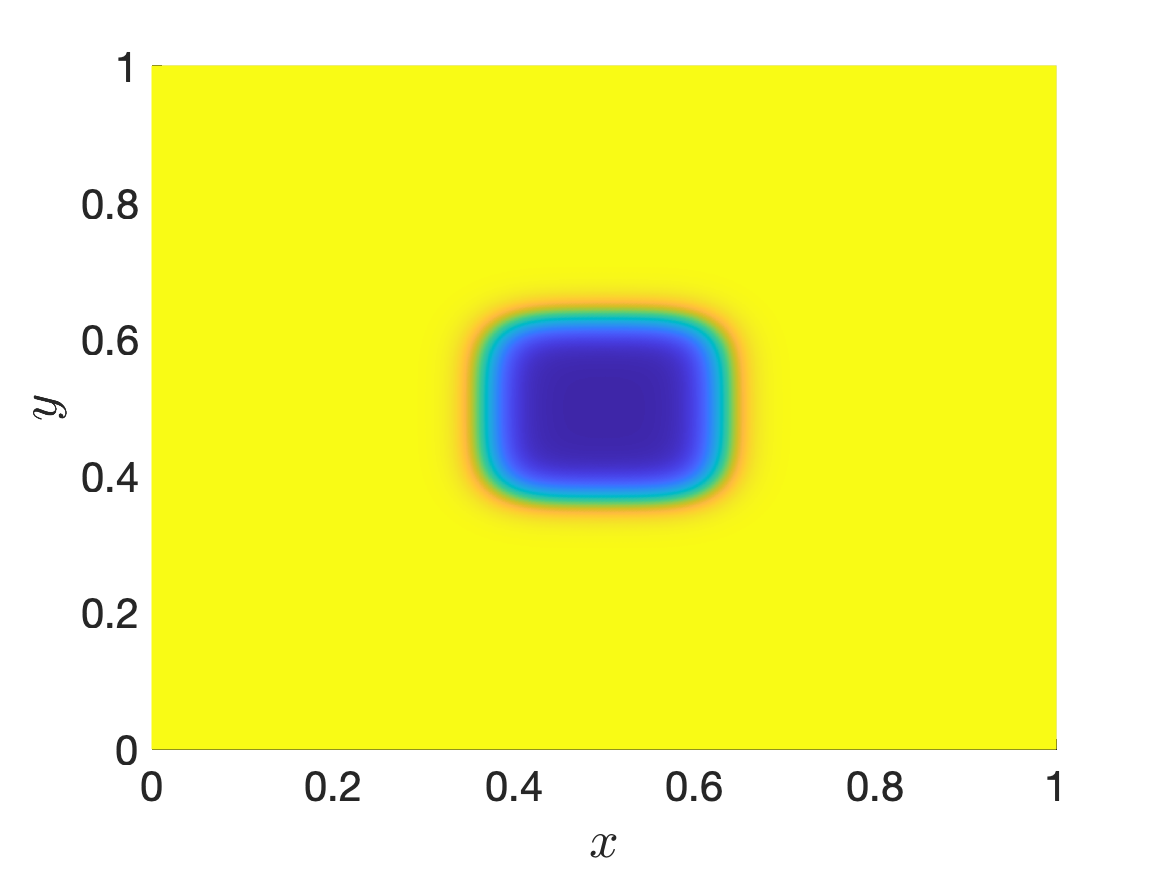}
 \caption{Snapshots of the phase-field solutions at $t=[0.002,    0.008,    0.015]$ (from top to bottom). From left to right: nonlocal model with an obstacle potential for $\beta=0.002$ (first column) and $\beta=0$ (second column),  local $(\beta=0)$ model with an obstacle potential (third column) and regular potential (fourth column).}\label{fig:ex5}
 \end{figure} 
  \begin{figure}[ht!]\centering
   \includegraphics[width=0.2\textwidth]{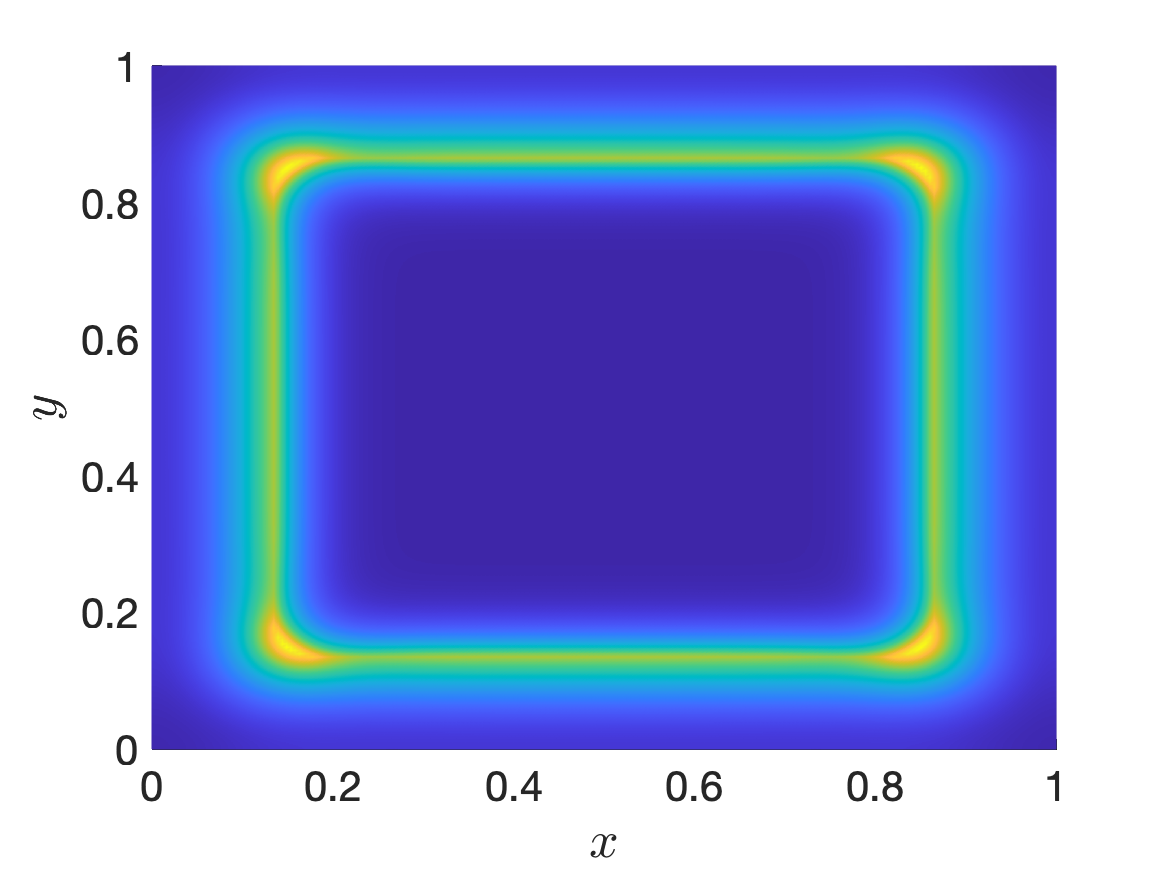}
    \includegraphics[width=0.2\textwidth]{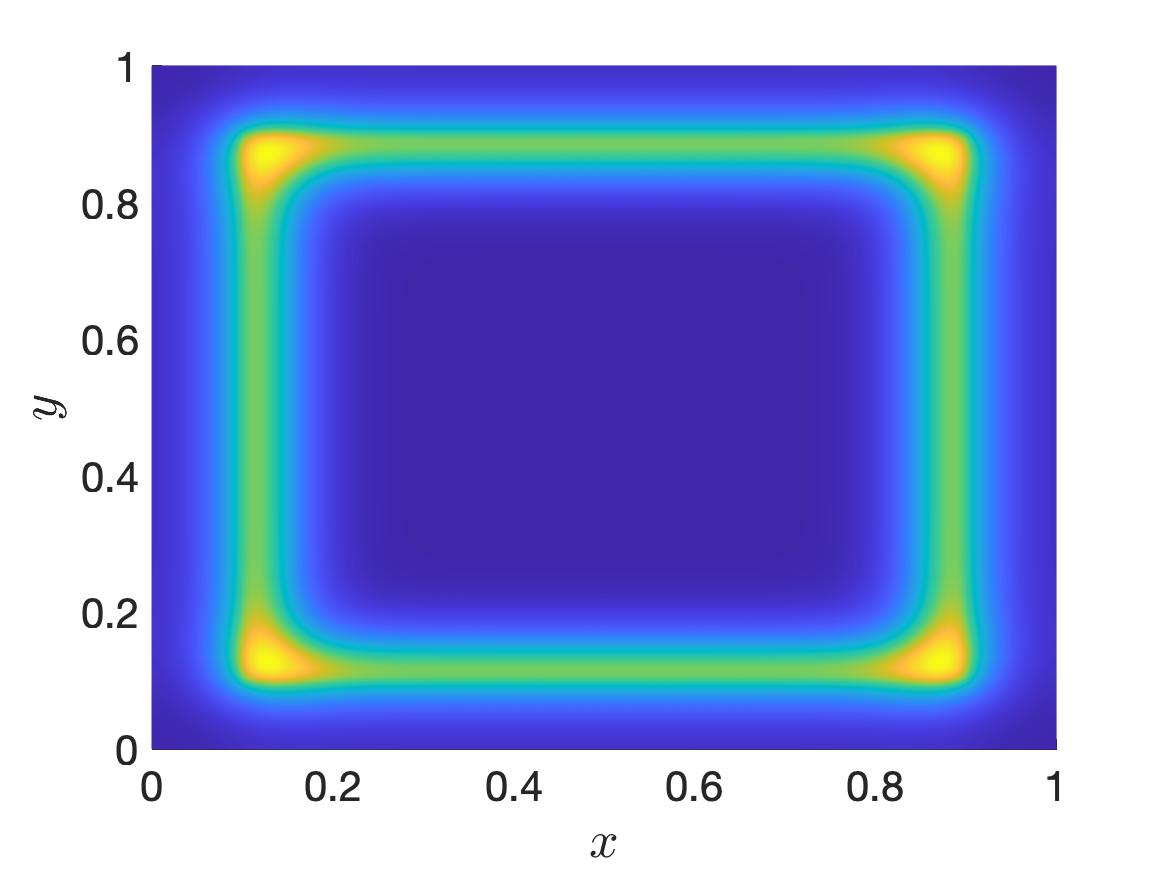}
   \includegraphics[width=0.2\textwidth]{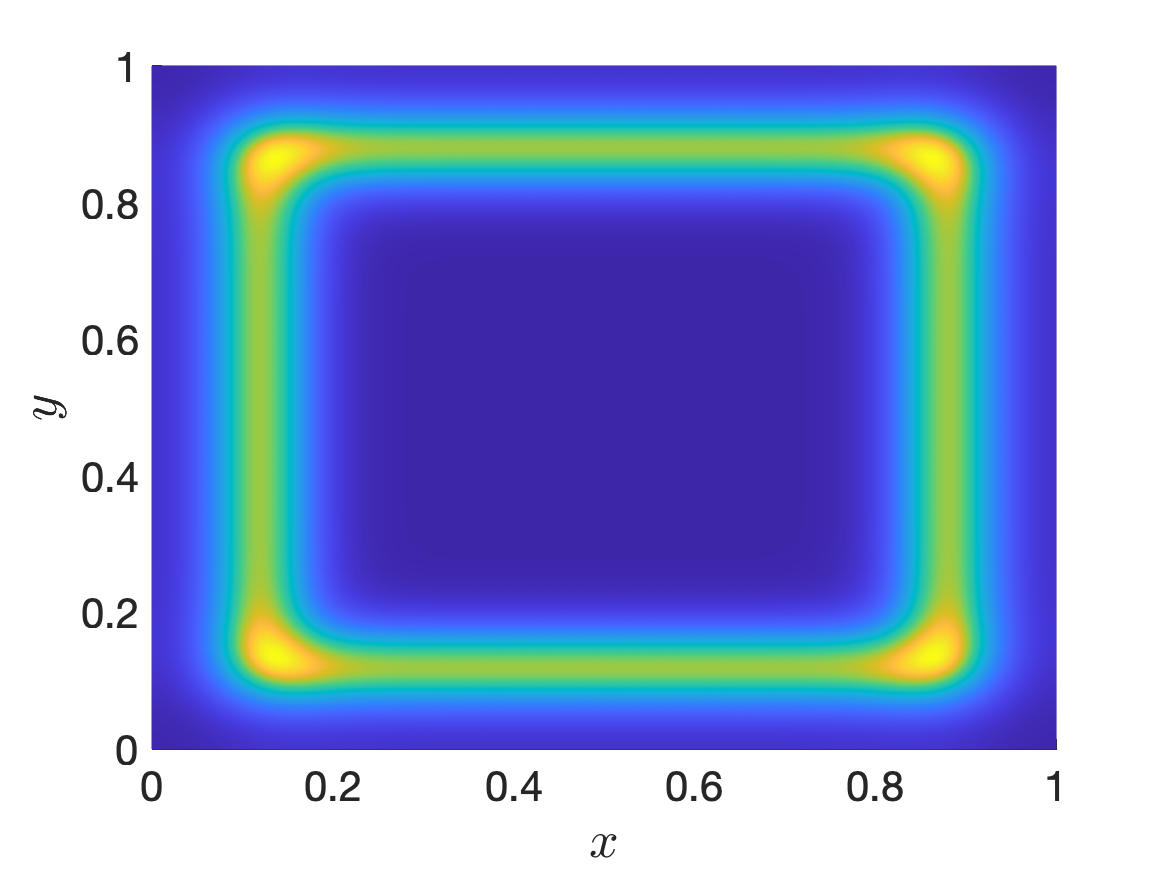}
   \includegraphics[width=0.2\textwidth]{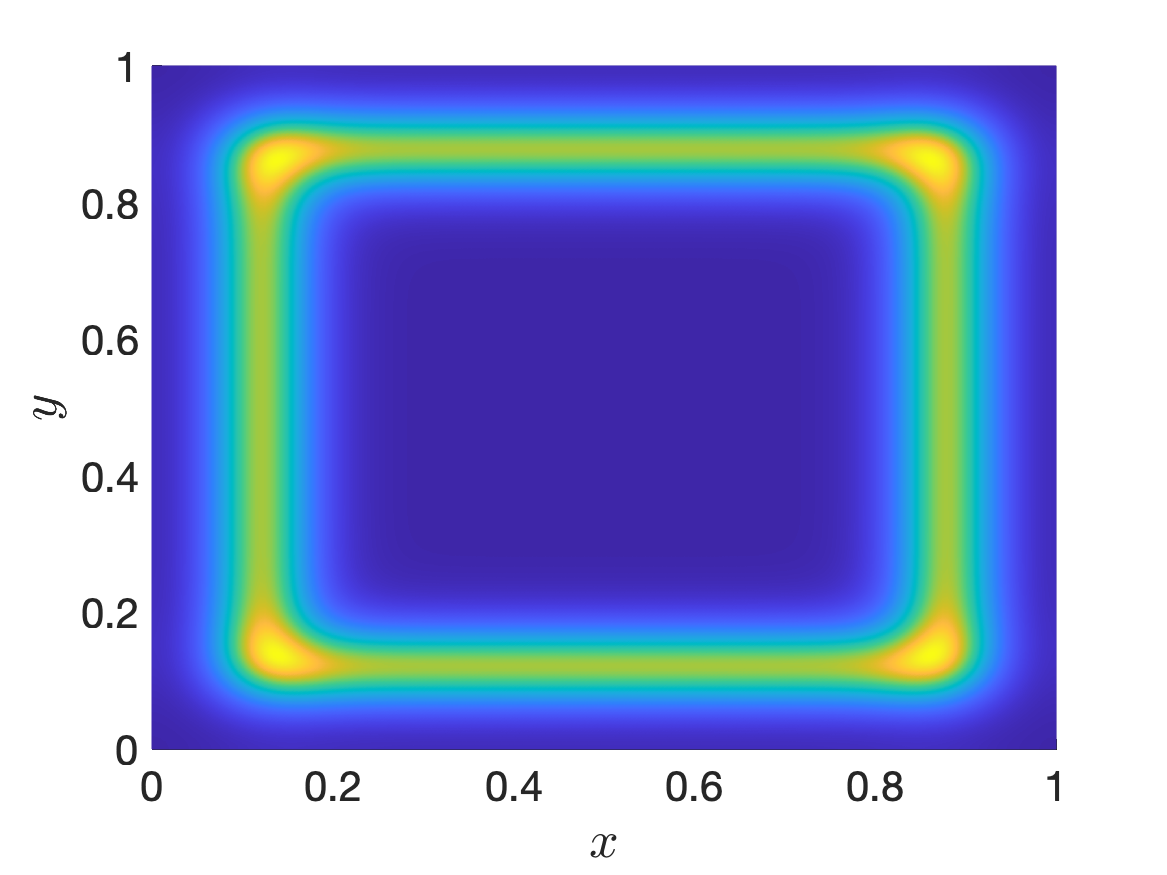}\\
       \includegraphics[width=0.2\textwidth]{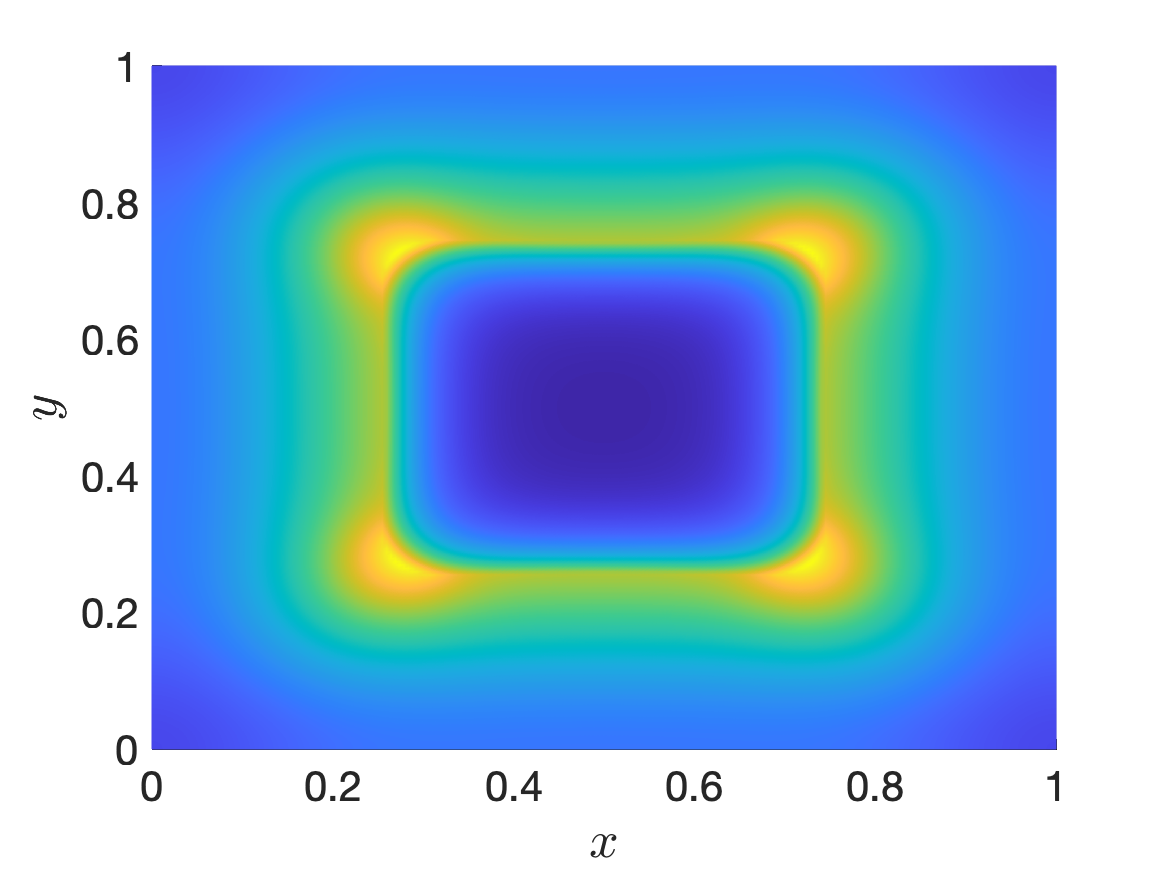}
   \includegraphics[width=0.2\textwidth]{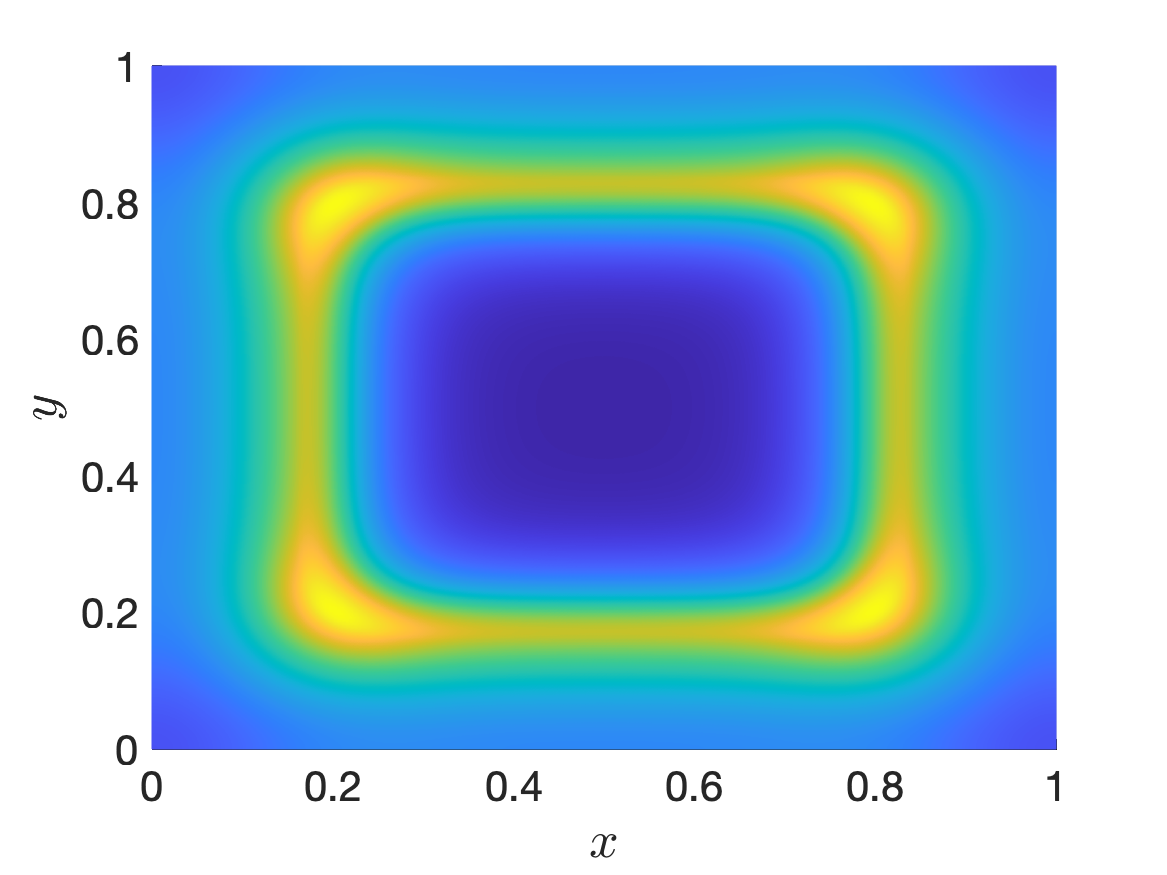}
    \includegraphics[width=0.2\textwidth]{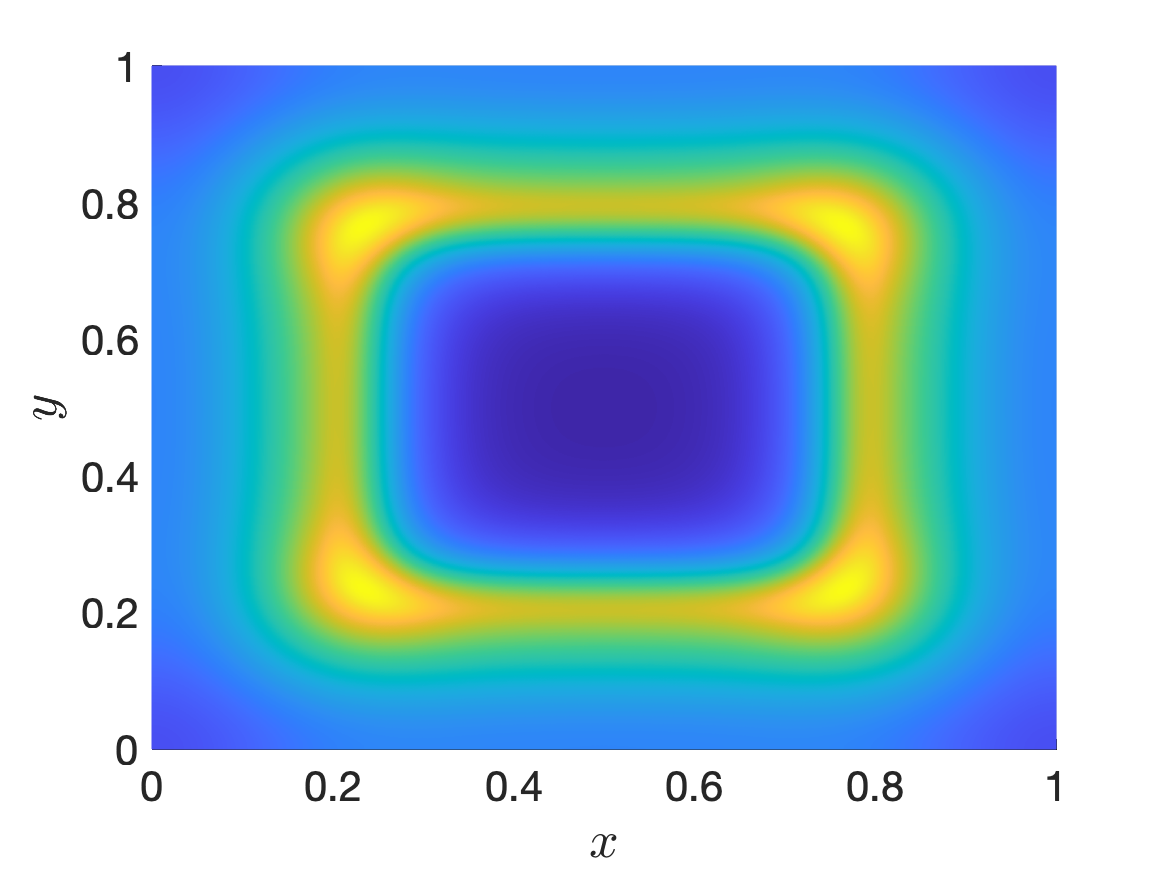}
     \includegraphics[width=0.2\textwidth]{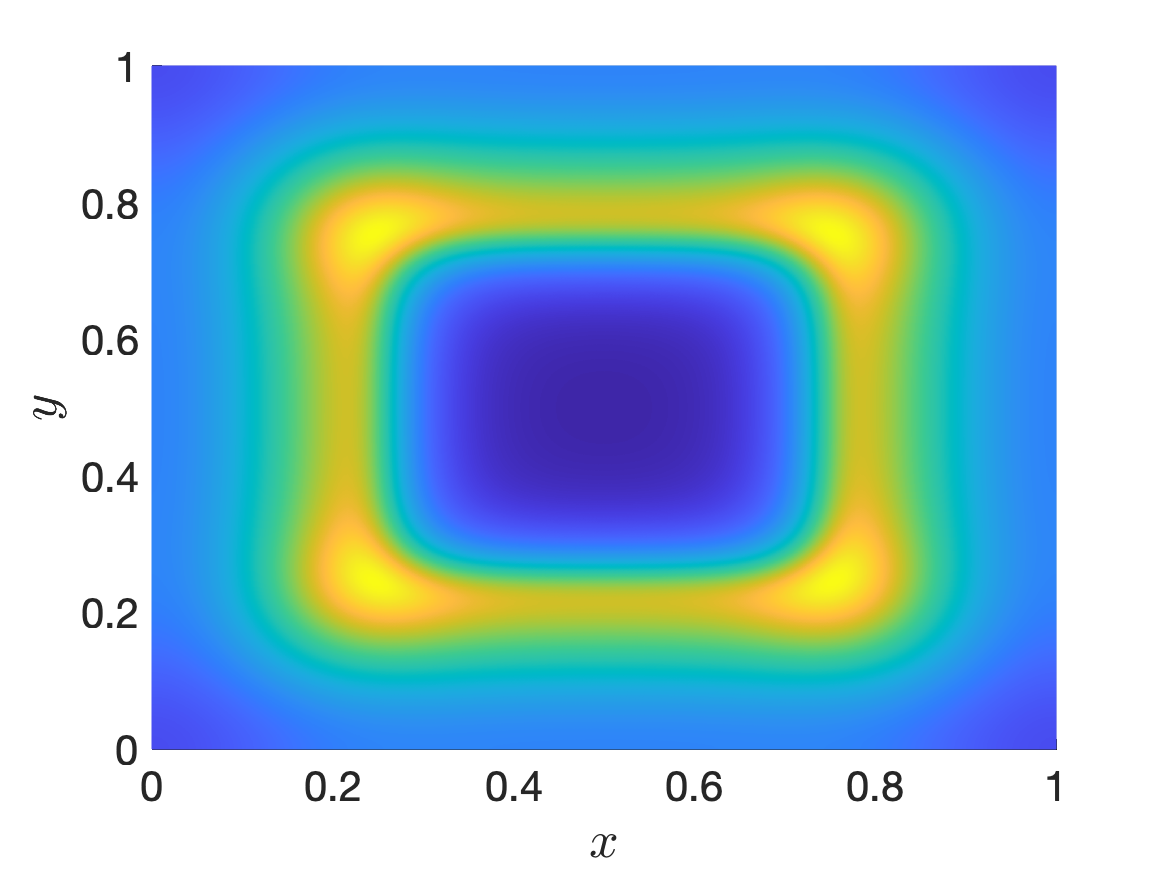}\\
   \includegraphics[width=0.2\textwidth]{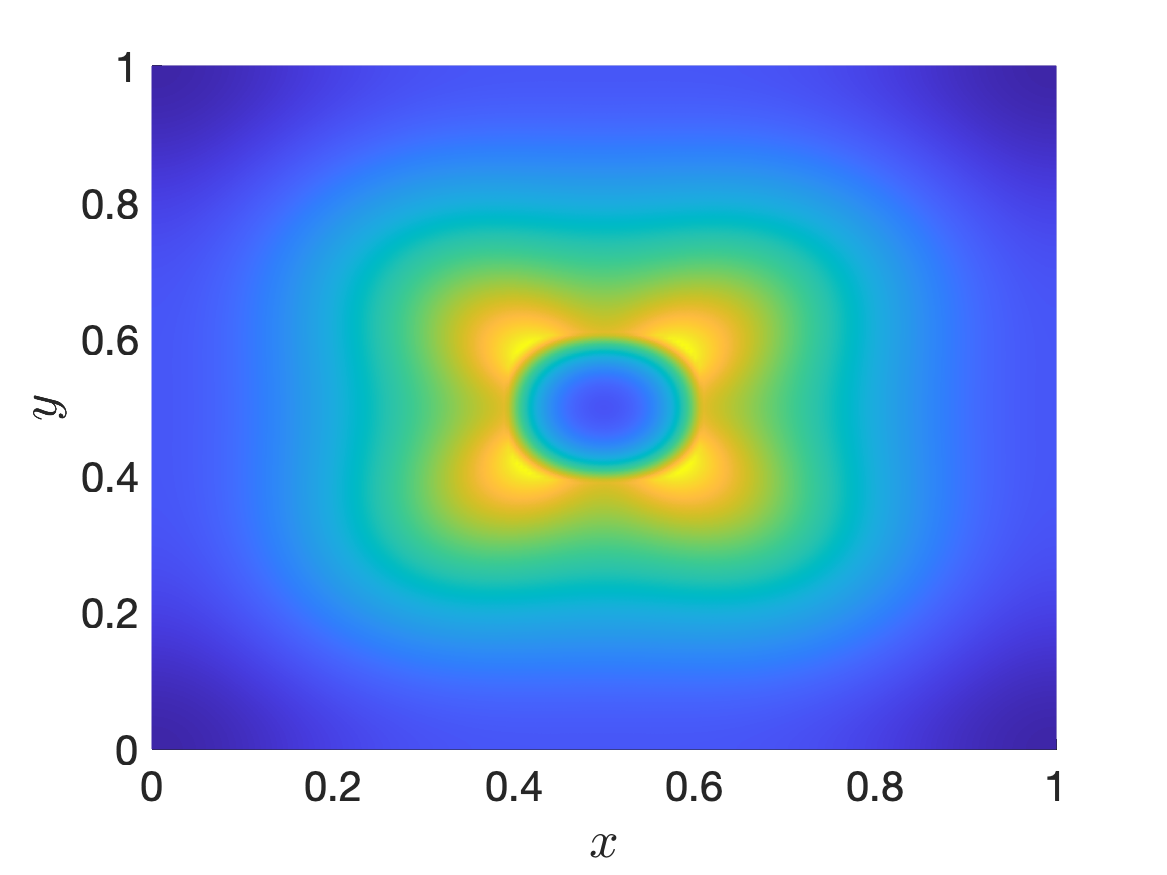}
    \includegraphics[width=0.2\textwidth]{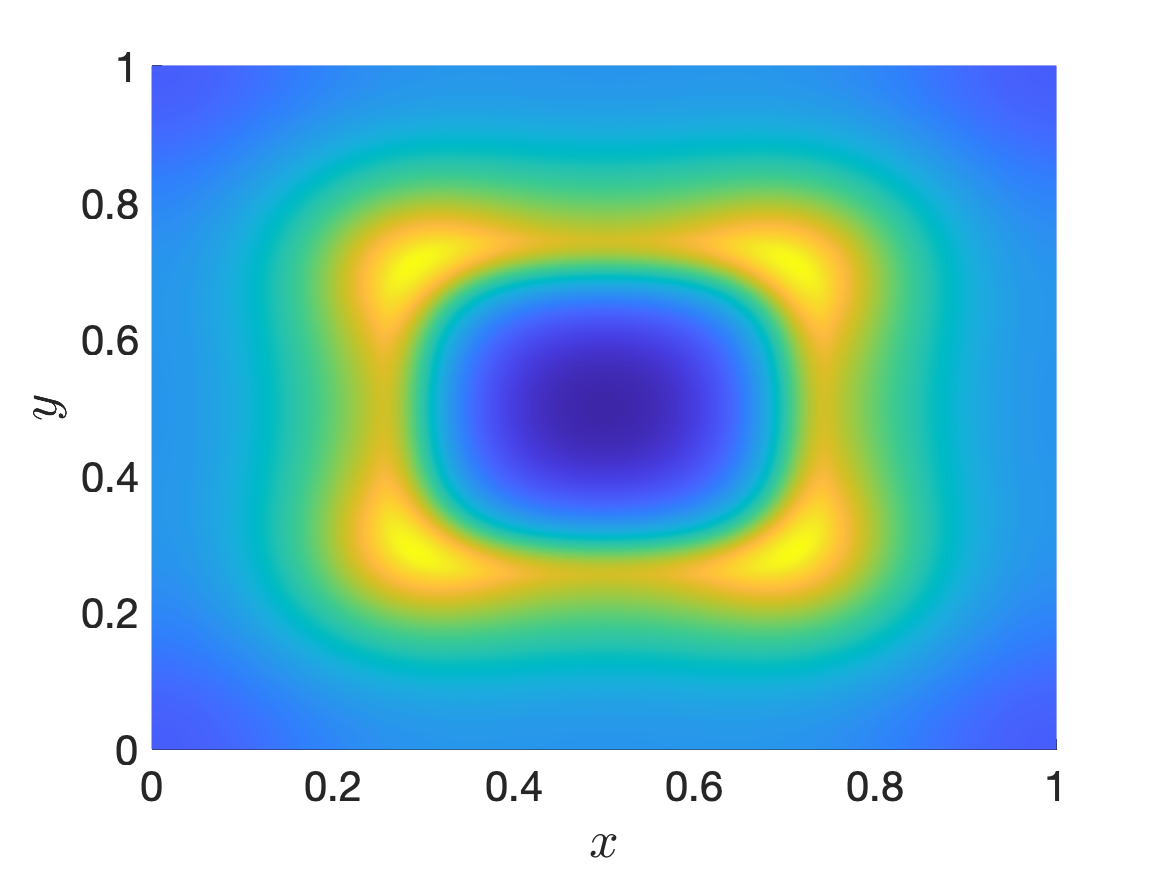}
    \includegraphics[width=0.2\textwidth]{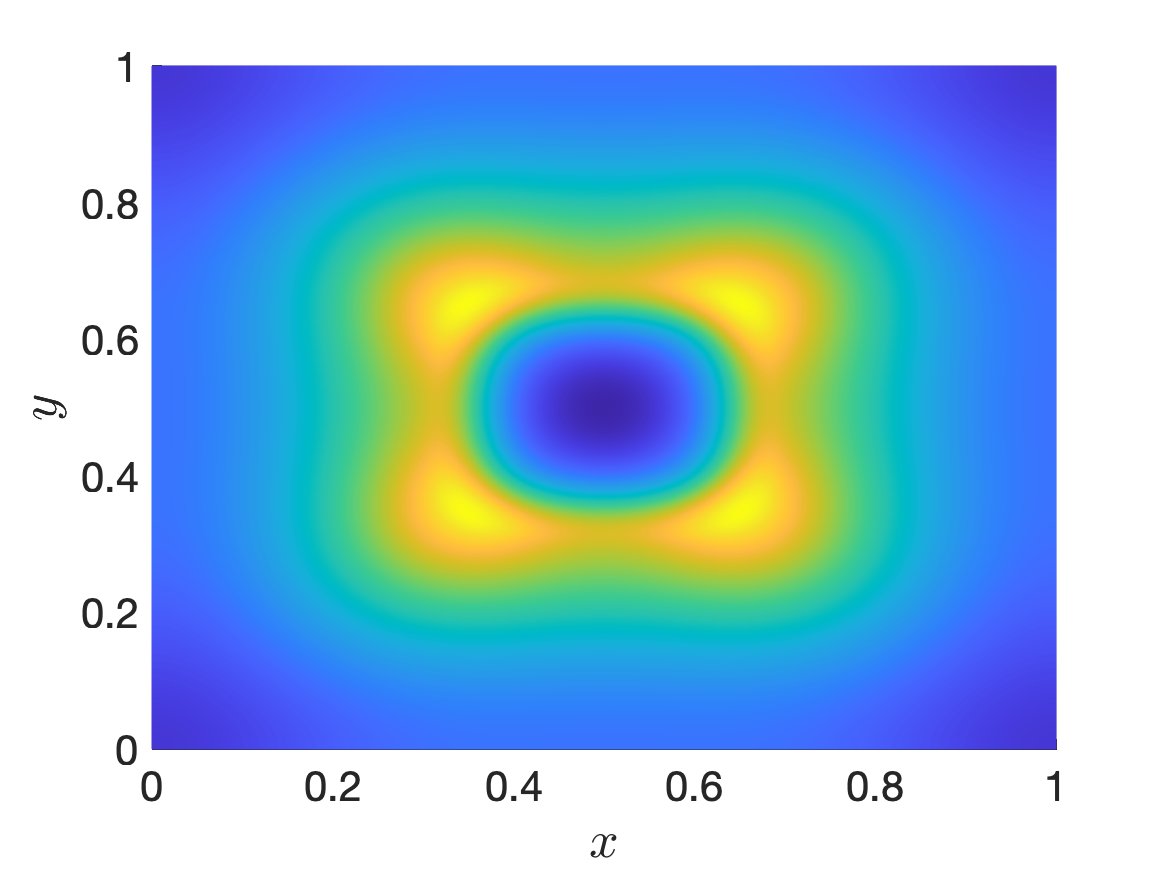}
    \includegraphics[width=0.2\textwidth] {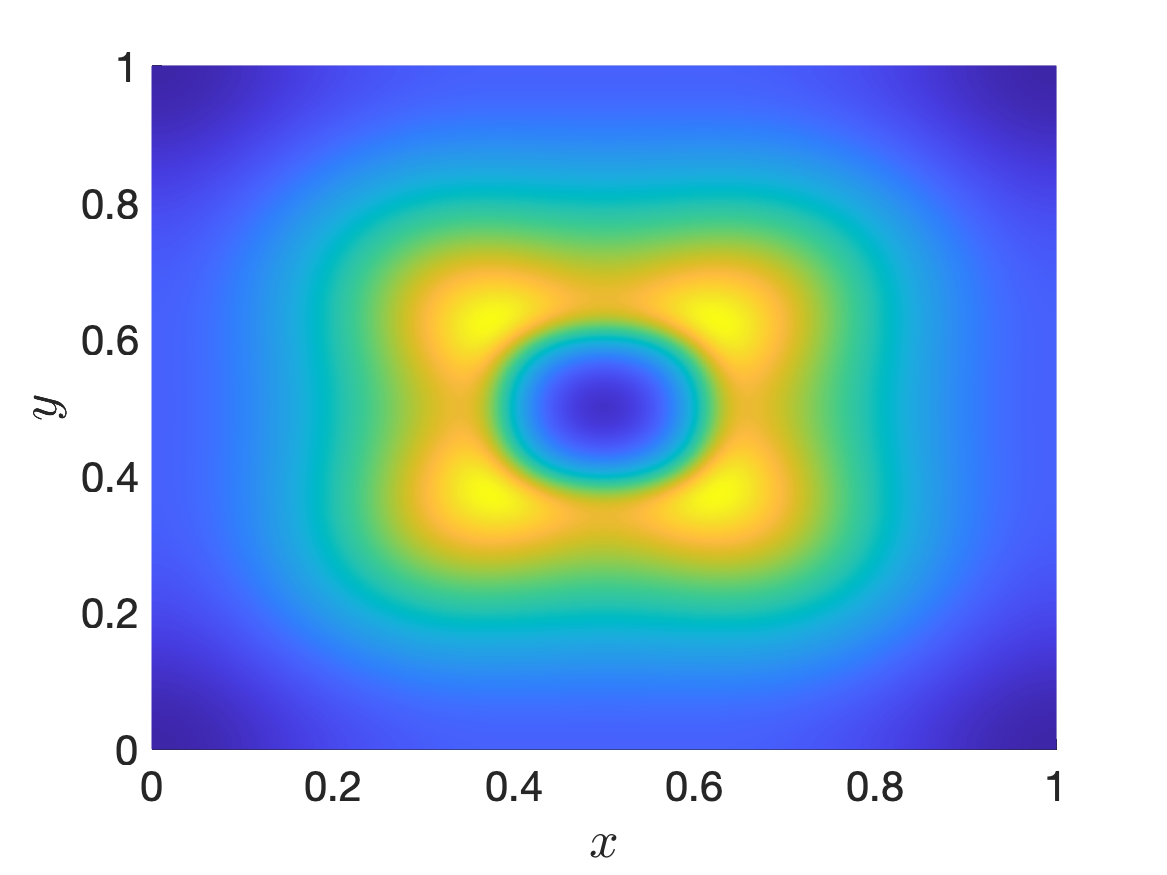}
 \caption{Snapshots of the temperature solutions at $t=[0.002,    0.008,    0.015]$ (from top to bottom). From left to right: nonlocal model with an obstacle potential for $\beta=0.002$ (first column) and $\beta=0$ (second column),  local $(\beta=0)$ model with an obstacle potential (third column) and regular potential (fourth column).}\label{fig:ex5a}
 \end{figure} 

  \begin{figure}[ht!]\centering
  \includegraphics[width=0.28\textwidth]{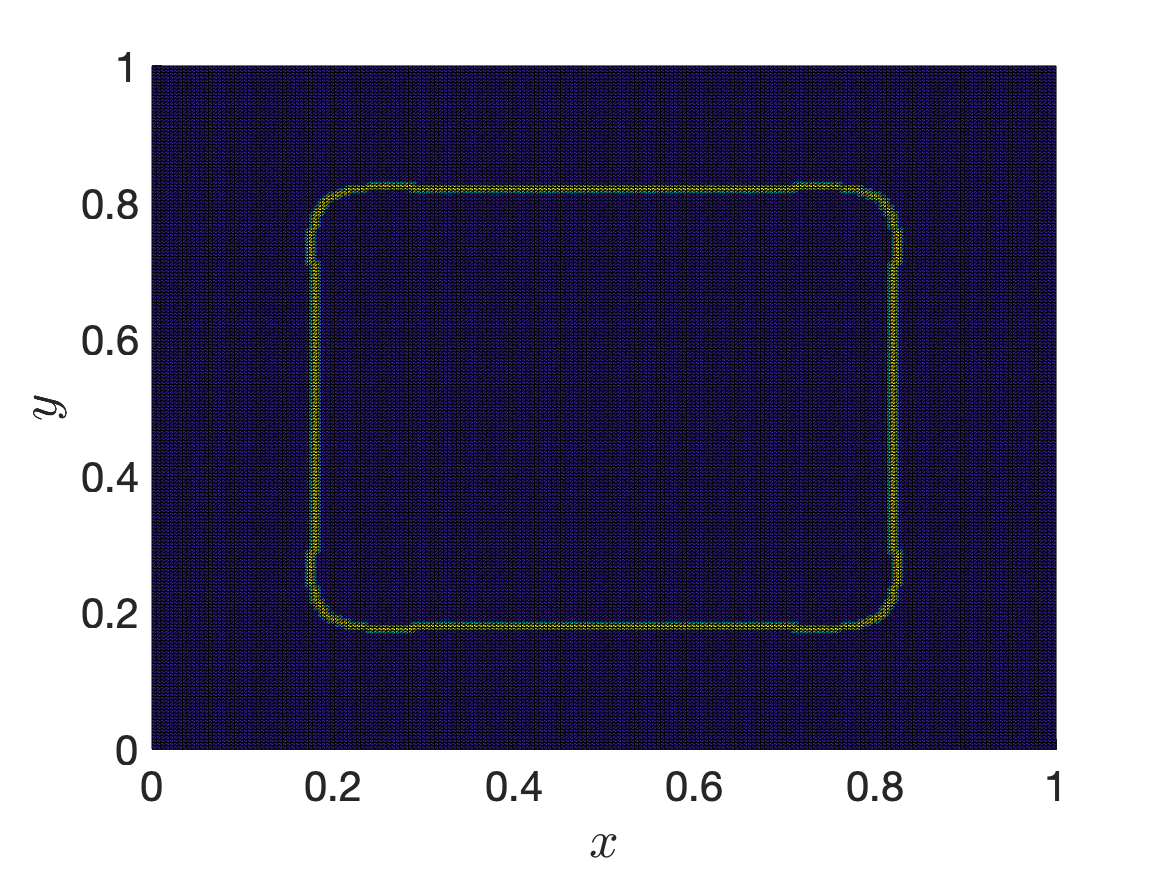}
   \includegraphics[width=0.28\textwidth]{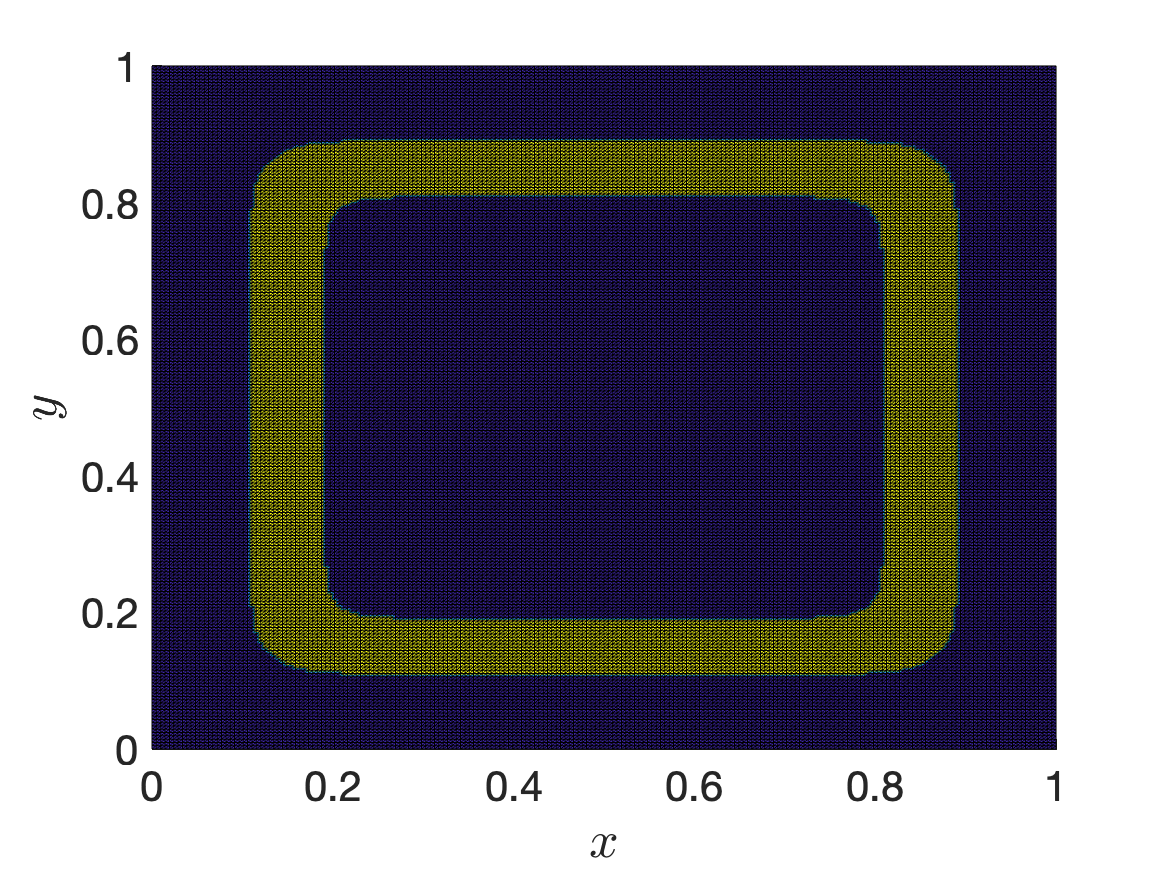}
       \includegraphics[width=0.28\textwidth]{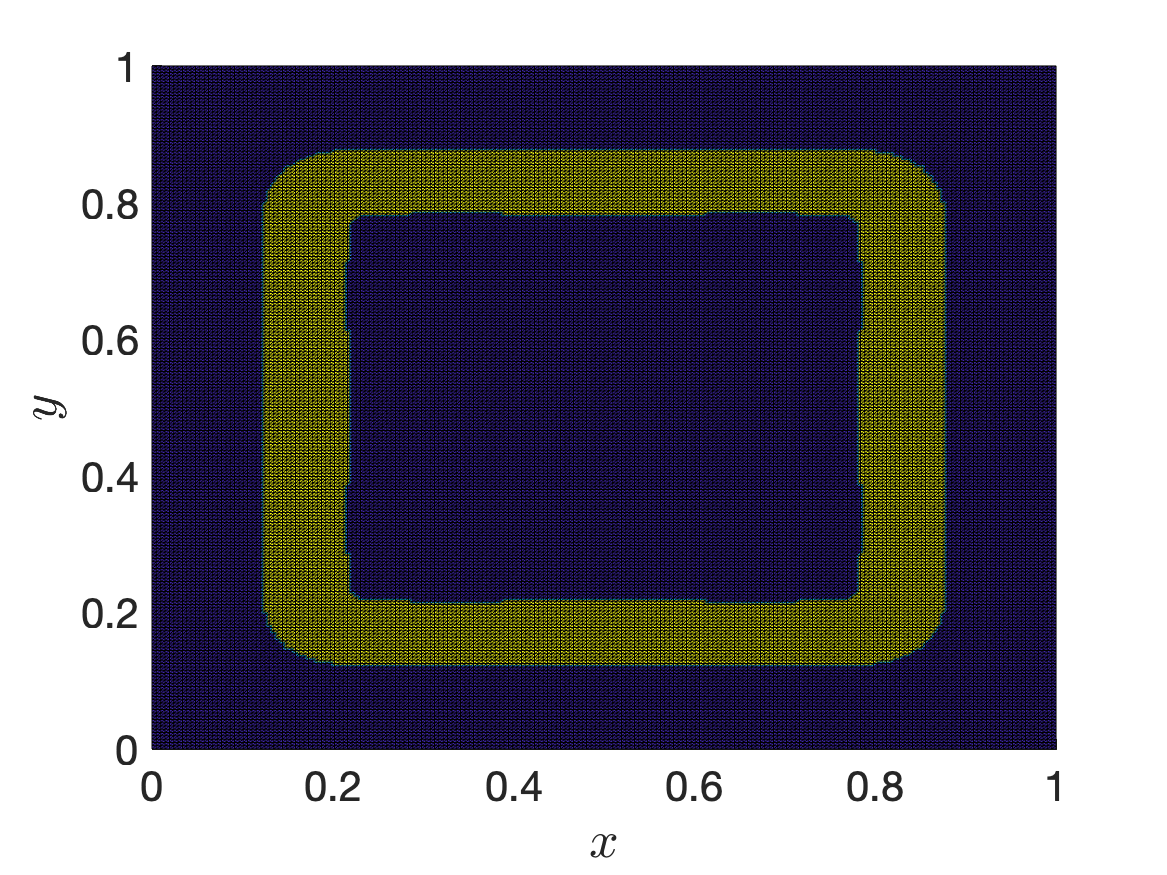}
 \caption{Width of the interface region, $0<u<1$, of the nonlocal ($\beta=0.002$ (left), $\beta=0$ (middle)) and local ($\beta=0$ (right)) phase-field solutions with the obstacle potential at $t=0.0041$.  }\label{fig:ex5b}
 \end{figure} 

 \section{Conclusion}\label{sec:conclusion}
 In this work we have analyzed nonlocal phase-field models of Cahn-Hilliard and Allen-Cahn type coupled to a temperature evolution equation.  The novel model based on the non-mass conserving nonlocal Cahn-Hilliard equation can provide sharp interfaces in the solution that can evolve in time which can be contrasted with the nonlocal Allen-Cahn setting that allows for sharp interfaces only in a steady-state solution.  We have provided a detailed analysis of the well-posedness of the models and proved the conditions under which sharp interfaces are attained.  We have presented several numerical results that illustrate the theoretical findings.  
 
The results of this work show that the developed nonlocal phase-field model is a promising approach to better describe non-isothermal phase-transitions with sharp or very thin interfaces.  Our future work~\cite{BDGR23} will focus on the numerical investigation of the new model in the context of solidification of pure materials,  its performance and comparison to some existing models.

\section{Acknowledgments}
 The author would like to express sincere thanks to Stephen DeWitt,  Max Gunzburger and Balasubramaniam Radhakrishnan for numerous valuable discussions about modeling aspects in the context of solidification.

\end{document}